\documentclass[12pt,reqno]{amsart}
\usepackage{amsthm,amsfonts,amssymb,euscript}

\theoremstyle{plain}
  \newtheorem{theorem}[subsection]{Theorem}
  
  \newtheorem{proposition}[subsection]{Proposition}
  \newtheorem{lemma}[subsection]{Lemma}
  \newtheorem{corollary}[subsection]{Corollary}

\def\bphi{{\boldsymbol{\phi} }}
\def\bpsi{{\boldsymbol{\psi} }}
\def\F{{\mathcal F}}
\def\R{{\mathbb{R}}}

\def\A{{\mathcal A}}
\def\B{{\mathcal B}}

\def\Q{{\mathcal Q}}
\def\H{{\mathcal H}}
\def\N{{\mathcal N}}

\def\E{{\mathcal E}}

\def\ch{\mbox{ch} (k)}
\def\chbb{\mbox{chb} (k)}
\def\trigh{\mbox{trigh} (k)}

\def\sh{\mbox{sh} (k)}
\def\shh{\mbox{sh}}
\def\chh{\mbox{ch}}
\def\shhb{\overline {\mbox{sh}}}
\def\th{\mbox{th} (k)}
\def\chhb{\overline {\mbox{  ch}}}
\def\shbb{\mbox{shb} (k)}

\def\shb{\overline{\mbox{sh}(k)}}
\def\z{\zeta}
\def\zb{\overline{\zeta}}
\def\FF{{\mathbb{F}}}
\def\L{\Lambda}

\theoremstyle{remark}
  \newtheorem{remark}[subsection]{Remark}
  
\theoremstyle{definition}

\newcommand{\tr}{\operatorname{tr}}

\newcommand{\I}{\text{i}}

\begin{document}

\def\vac{{\big\vert 0\big>}}
\def\fpsi{{\big\vert\psi\big>}}
\def\bphi{{\boldsymbol{\phi} }}
\def\bpsi{{\boldsymbol{\psi} }}
\def\F{{\mathcal F}}
\def\R{{\mathbb{R}}}
\def\I{{\mathcal I}}
\def\Q{{\mathcal Q}}
\def\ch{\mbox{\rm ch} (k)}
\def\cht{\mbox{\rm ch} (2k)}
\def\chbb{\mbox{\rm chb} (k)}
\def\trigh{\mbox{\rm trigh} (k)}
\def\p{\mbox{p} (k)}
\def\sh{\mbox{\rm sh} (k)}
\def\sht{\mbox{\rm sh} (2k)}
\def\shh{\mbox{\rm sh}}
\def\chh{\mbox{ \rm ch}}
\def\th{\mbox{\rm th} (k)}
\def\thb{\overline{\mbox{\rm th} (k)}}
\def\shhb{\overline {\mbox{\rm sh}}}
\def\chhb{\overline {\mbox{\rm ch}}}
\def\shbb{\mbox{\rm shb} (k)}
\def\chbt{\overline{\mbox{\rm ch}  (2k)}}
\def\shb{\overline{\mbox{\rm sh}(k)}}
\def\shbt{\overline{\mbox{\rm sh}(2k)}}
\def\z{\zeta}
\def\zb{\overline{\zeta}}
\def\FF{{\mathbb{F}}}
\def\S{{\bf S}}
\def\Sr{{\bf S}_{red}}
\def\W{{\bf W}}
\def\WW{{\mathcal W}}
\def\L{{\mathcal L}}
\def\E{{\mathcal E}}
\def\X{{\tilde X}}
\def\Nor{{\rm Nor}}

\def \kb{{\bf k}}
\def\x{{\bf x}}
\def\fhat{{\widehat{f}}}
\def \ghat{{\widehat{g}}}
\def \fhatb{{\widehat{f}^{\ast}}}
\def \ghatb{{\widehat{g}^{\ast}}}
\def \Sb{\mathbb S}
\def\Zb {\mathbb Z}
\def\Lm{\Lambda}
\def \Lmb{\overline{\Lambda}}
\def \Gam{\Gamma}
\def\Gamb{\overline{\Gamma}}
\def\Lmh{\widehat{\Lambda}}
\def\Fhat{\widehat{F}}
\def\Fhatb{\overline{\widehat{F}}}
\def\E{{\mathcal E}}
\def\Sop{{\bf S}^{t}_{x_{1},x_{2}}}
\def\Wop{{\bf S}^{\pm,t}_{x_{1},x_{2}}}
\def\Sp{{\bf S}^{t}_{x_{1}}}
\def\nb{\nabla}
\def\Dl{\Delta}
\def\Acal{{\mathcal A}}
\def\ub{\overline{u}}
\def\tr{{\rm tr}}
\def\dig{{\rm diag}}

\def\Hbb{{\mathbb H}}
\def\Ubb{{\mathbb U}}
\def\Nbb{{\mathbb N}}

\def\zvec{\vec{z}}
\def\xvec{\vec{x}}
\def\zvecp{\vec{z}^{\prime}}


\title[Global estimates for HFB ]{Global  uniform in $N$  estimates
for solutions of a system of Hartree-Fock-Bogoliubov type in the case $\beta<1$}

\author{J. Chong}
 \address{University of Texas at Austin }
\email{jwchong@math.utexas.edu }

\author{X. Dong}
\address{University of Maryland, College Park}
\email{away@umd.edu}

\author{M. Grillakis}
\address{University of Maryland, College Park}
\email{mng@math.umd.edu}

\author{M. Machedon}
\address{University of Maryland, College Park}
\email{mxm@math.umd.edu}

\author{Z. Zhao}
\address{Department of Mathematics and Statistics, Beijing Institute of Technology, Beijing, China}
\address{MIIT Key Laboratory of Mathematical Theory and Computation in Information Security, Beijing, China}
\email{zzh@bit.edu.cn}

\subjclass{35Q55, 81}
\keywords{Hartree-Fock-Bogoliubov}
\date{\today}
\dedicatory{}
\commby{}
\commby{}

\maketitle

\begin{abstract} We extend the results of the 2019 paper by the third and fourth author globally in time. More precisely,
we prove  uniform in $N$  estimates for the solutions $\phi$, $\Lambda$ and $\Gamma$  of a coupled system of Hartree–Fock–Bogoliubov type with interaction potential
$V_N(x-y)=N^{3 \beta}v(N^{\beta}(x-y))$ with $\beta<1$.  The potential satisfies some technical conditions, but is not small.
The initial conditions have finite energy and the  ``pair correlation" part satisfies a smallness condition, but are otherwise general functions in suitable Sobolev spaces, and the expected correlations in $\Lambda$ develop dynamically in time. The estimates are expected to improve the Fock space bounds from the 2021 paper of the first and fifth author. This will be addressed in a different paper.

\end{abstract}

\section{Introduction}
The general motivation for this paper is the evolution of $N$ Bosons under a mean-field Hamiltonian
\begin{align*}
-\sum_{k=1}^N \Delta_{x_i} + \frac{1}{2N} \sum_{k \neq l} V_N(x_k-x_l)
\end{align*}
where $x_i \in \mathbb R^3$, $N$ is large and
\begin{align*}
V_N(x)=N^{3 \beta}v(N^{\beta}x)
\end{align*}
and the potential $v$ is discussed below. (The notation $v_M(x)$ will also be used in sections \ref{Statement}-\ref{endlin} ,
 with a different meaning.)
The initial conditions are (exactly or approximately) a tensor product $\phi \otimes \cdots \otimes \phi$.

The exact evolution of the system is approximated by a construction involving just two functions: the condensate $\phi(t, x)$ and a pair excitation function
$k(t, x, y)$, and it is
\begin{equation}\label{approx}
\psi_{approx}:=e^{-\sqrt{N}\A(\phi(t)}e^{-\B(k(t))}\Omega
\end{equation}
where
\begin{equation}
\A(\phi):=\int dx\left\{\bar{\phi}(x)a_{x}-\phi(x)a^{\ast}_{x}\right\} \label{meanfield-2}
\end{equation}
 and $e^{-\sqrt{N}\A(\phi)}$ is a unitary operator on Fock space,  the Weyl operator.
 and
 \begin{align}
&\B(k):=
\frac{1}{2}\int dxdy\left\{\bar{k}(t,x,y)a_{x}a_{y}
-k(t,x,y)a^{\ast}_{x}a^{\ast}_{y}\right\}\ . \label{pairexcit-2}
\end{align}
The unitary operator $e^{\B(k)}$ is the representation of an (infinite dimensional) real symplectic matrix.
Also, $\Omega$ is the vacuum.
See for instance \cite{GM2} for background on this construction.

In order for $\psi_{approx}$ to be an approximation to the exact evolution, $\phi$ and $k$ must satisfy certain PDEs.
In the math literature, they were introduced in \cite{GM2} and independently and in a different context in \cite{BBCFS}. They were studied in \cite{G-M2017}, \cite{G-M2019}, \cite{CGMZ}, as well as \cite{BSS}.

To write down the equations it is convenient to consider a self-adjoint kernel
\begin{align*}
&\Gamma(t, x, y)= \bar \phi(t, x) \phi(t, y)+\frac{1}{N}\left(\shb \circ \sh\right)(t, x, y):=\Gamma_c+\Gamma_p\\
&\mbox{and a symmetric kernel}\\
 &\Lambda(t, x, y)= \phi(t, x) \phi(t, y)+\frac{1}{2N}\sht (t, x, y):=\Lambda_c+\Lambda_p
 \end{align*}
  where
 \begin{align*}
\sh &:=k + \frac{1}{3!} k \circ \overline k \circ k + \ldots~, \\
\ch &:=\delta(x-y) + \frac{1}{2!}\overline k  \circ k + \ldots~
\end{align*}
The functions $\Lambda$ and $\Gamma$ have the conceptual meaning of reduced density matrices.
Here, $(u \circ v)(x, y)=\int u(x, z)v(z, y) dz$).
There are several equivalent ways of expressing the equations. In this section we give a compact, matrix formulation.

For the current paper we separate the condensate part from the pair interaction part: define $\Gamma_c = \bar \phi \otimes \phi$,
$\Lambda_c =  \phi \otimes \phi$,
$\Gamma_p=\frac{1}{N} \shb\circ \sh$ and
 $\Lambda_p=\frac{1}{2N} \sht$.
Also, denote $\rho(t, x)=\Gamma(t, x, x)$.

To write  the Hartree-Fock-Bogoliubov equations in matrix notation, define

\begin{align*}
\Omega=\left(
\begin{matrix}
- \Gamma &-\bar{\Lambda}\\
\Lambda & \bar{\Gamma}
\end{matrix}
\right):=\Psi+\Phi
\end{align*}
where
\begin{align*}
\Psi=\left(
\begin{matrix}
- \Gamma_p &-\bar{\Lambda}_p\\
\Lambda_p & \bar{\Gamma}_p
\end{matrix}
\right)
\end{align*}
\begin{align*}
\Phi=\left(
\begin{matrix}
- \Gamma_c &-\bar{\Lambda}_c\\
\Lambda_c & \bar{\Gamma}_c
\end{matrix}
\right)
\end{align*}

Finally, let
\begin{align*}
S_3=\left(
\begin{matrix}
- I &0\\
0 & I
\end{matrix}
\right)
\end{align*}
where $I$ is the identity operator.

The evolution equations for $\Omega$ and $\Psi$ (for $t>0$, with initial conditions at $t=0$) are
\begin{align}
&\frac{1}{i} \partial_t \Phi - [\Delta_x \delta(x-y) S_3, \Phi]\notag\\
&=-[( V_N*\rho(t, x))\delta(x-y) S_3, \Phi]
 - [V_N \Psi^*, \Phi]\label{rhs1}\\
&\frac{1}{i} \partial_t \Psi
- [\Delta_x \delta(x-y) S_3, \Psi]\label{rhs2}\\
&=-[( V_N*\rho)\delta(x-y) S_3, \Psi]
-\frac{1}{2N}[S_3, V_N \Psi]-
[V_N \Omega^*, \Psi]-\frac{1}{2N}[S_3, V_N \Phi]\notag
\end{align}

In addition, the condensate $\phi$ satisfies
\begin{align*}
&  \left\{\frac{1}{i}\partial_{t}-\Delta_{x_1}\right\}\phi(x_{1})  \notag\\
&=-\int dy\left\{
v_N(x_{1}-y)\Gamma(y,y)\right\}\phi(x_{1})
\\
&-\int dy\big\{v_N(x_{1}-y)\Gamma_p(y,x_{1})\phi(y)
\\
&+ \int dy\big\{v_N(x_{1}-y)\Lambda_p(x_{1},y)\big\}\overline{\phi}(y )
\end{align*}

Here $A^*(x, y)=\bar A(y, x)$, $[A, B]= A \circ B -B \circ A$ and $V_N$ acts by pointwise multiplication by $V_N(x-y)$.
We will write down these equations in scalar form later, see \eqref{rel1}-\eqref{rel4}). Also, we will write down a simplified model at the end of the introduction.

The arguments of this paper will involve non-local fractional time derivatives, so the values of the solutions at negative times also matter. It is convenient to replace \eqref{rhs1}, \eqref{rhs2} by
\begin{align}
&\frac{1}{i} \partial_t \Phi - [\Delta_x \delta(x-y) S_3, \Phi]=h(t)\times RHS \eqref{rhs1}\label{rhs1cut}\\
&\frac{1}{i} \partial_t \Psi
- [\Delta_x \delta(x-y) S_3, \Psi]=h(t)\times RHS\eqref{rhs2} \label{rhs2cut}
\end{align}
where $h(t)$ is the characteristic function of $[0, \infty)$ and $RHS\eqref{rhs1} $ stands for the right hand side of equation \eqref{rhs1}. This is the usual solution one gets from Duhamel's formula, and solutions to \eqref{rhs1}, \eqref{rhs2} agree with solutions to \eqref{rhs1cut}, \eqref{rhs2cut} fot $t>0$, provided they have the same initial conditions.

Next, we review the conserved quantities of these equations. See \cite{GM2} for details.
The first conserved quantity is the total number of particles (normalized by division by $N$):
\begin{align}
{\rm tr}\left\{\Gamma(t)\right\}= \|\phi(t, \cdot)\|^2_{L^2(dx)} +\frac{1}{N} \|\sh (t, \cdot, \cdot)\|^2_{L^2(dxdy)}=1\ . \label{b5-nbr0}
\end{align}

The second conserved quantity is the energy per  particle
\begin{align}
E(t):=&\ {\rm tr}\left\{\nabla_{x_{1}}\cdot\nabla_{x_{2}}\Gamma(t)\right\}
+\frac{1}{2}
\int dx_{1}dx_{2}\left\{V_{N}(x_{1}-x_{2})\big\vert\Lm(t,x_{1},x_{2})\big\vert^{2}
\right\}
\label{energy}
\\
&+\frac{1}{2}
\int dx_{1}dx_{2}\left\{
V_{N}(x_{1}-x_{2})\left(\big\vert\Gamma(t,x_{1},x_{2})\big\vert^{2}
+\Gamma(t,x_{1},x_{1})\Gamma(t,x_{2},x_{2})\right)
\right\}
\nonumber
\\
&-\int dx_{1}dx_{2}\left\{V_{N}(x_{1}-x_{2})
\vert\phi(t,x_{1})\vert^{2}\vert\phi(t,x_{2}\vert^{2}
\right\}\ .
\notag
\end{align}
The above holds for any Schwartz potential $v$. In addition, in order to use the estimates of
\cite{CGMZ},
 we assume
\begin{align}
&\mbox{$v$ is spherically symmetric and }\label{Vhyp} \\
&v \ge 0, \, v \in C_0^{\infty}, \,  \frac{\partial v}{\partial r}( r ) \le 0. \notag
\end{align}
For the initial conditions, we assume there exist a constant $C$ (independent of $N$) and $\alpha>\frac{1}{2}$ such that
\begin{align}
&{\rm tr}\left\{\Gamma(0)\right\} \le C \notag \\
&E(0) \le C \, \, \label{datahyp}\\
&\|\big<\nabla_x\big>^{\alpha}\big<\nabla_y\big>^{\alpha}\Gamma(0, x, y)\|_{L^2} \le C
\, \mbox{(this follows from the previous condition,}\notag\\
&\mbox{ and will be preserved by the time evolution)}
\notag\\
&\|\big<\nabla_x\big>^{\alpha}\big<\nabla_y\big>^{\alpha}\Lambda(0, x, y)\|_{L^2} \le C
\, \, \mbox{(this will also be preserved by the time evolution } \notag\\
&\mbox{for all $\frac{1}{2} \le \alpha \le \alpha_0$ for some $\alpha_0>\frac{1}{2}$, as was shown in \cite{CGMZ})}
\notag\\
&\||\nabla_x||\nabla_y| \Lambda(0, x, y)\|_{L^2} \le C N. \notag
\end{align}
The data is also assumed to be in a high $H^s$ space (but not uniformly in $N$).

In addition, there will be a smallness assumption on the initial conditions for the ``pair" components of $\Gamma$ and $\Lambda$.
Under the above assumptions,  the arguments of \cite{CGMZ} imply that for all $\alpha > \frac{1}{2}$, sufficiently close to
$\alpha > \frac{1}{2}$,
there exists $\epsilon_3>0$ such that
\begin{align}
&\int
\big\vert|\nb_x|^{\alpha}|\nb_y|^{\alpha}\Lm_p(t, x, y)\big\vert^{2}dx dy \label{lambdaestimp}
\le \frac{C}{N^{\epsilon_3}}
\end{align}
uniformly in $t$ and $N$.
This follows by interpolating between (14) in \cite{CGMZ}  and Theorem 1.2 in that paper.

 In addition, we assume
\begin{align}
 \big\|\big<\nabla_x\big>^{\alpha}\big<\nabla_y\big>^{\alpha}\Gamma_p(0, \cdot, \cdot) \big\|_{L^2} \le \frac{1}{N^{\epsilon_3}}\label{smallhyp}
 \end{align}
 $\alpha$ is a number slightly bigger that $\frac{1}{2}$, to be chosen later.

Also, from Proposition 3.4 in \cite{CGMZ} we have a Morawetz type estimate
 \begin{align}
||\Gamma(t,x,x)||_{L^2_{t,x}}  \lesssim 1 \label{morawetzGamma}
\end{align}
while from conservation of energy and the trace theorem,
 \begin{align}
||\Gamma(t,x,x)||_{L^{\infty}(dt)L^2(dx)}  \lesssim 1. \label{morawetzGammaenergy}
\end{align}

 In order to state the main result of this paper in the simplest possible form, we define the following partial Strichartz norms:

\begin{align}
&\|\Lambda\|_{\mathcal{S}_{x, y}}\label{restr}\\
&=  \sup_{  p, q \, \, admissible, }\|\Lambda\|_{L^{p}(dt) L^{q}(dx) L^2(dy)}\notag\\
&\quad +\sup_{  p, q \, \notag \, admissible, \, }\|\Lambda\|_{L^{p}(dt) L^{q}(dy) L^2(dx)}.
\end{align}

Recall $p, q$ are admissible in $3+1$ dimensions if $\frac{2}{p}+ \frac{3}{q} = \frac{3}{2}$, $2 \le p \le \infty$. Thus
\begin{align*}
\|\Lambda\|_{\mathcal{S}_{x, y}} \sim
\|\Lambda\|_{L^2(dt)L^6(dx)L^2(dy)}
+\|\Lambda\|_{L^2(dt)L^6(dy)L^2(dx)}+\|\Lambda\|_{L^{\infty}(dt)L^2(dxdy)}.
\end{align*}

The main result of this paper is
\begin{theorem}\label{intro1} Let $\Lambda=\Lambda_{p }+\Lambda_{ c}$, $\Gamma=\Gamma_{p  }+\Gamma_c$ be solutions of \eqref{rhs1cut}, \eqref{rhs2cut}
(or, equivalently, \eqref{rel1}-\eqref{rel4}), where  the potential satisfies \eqref{Vhyp}, and the initial conditions satisfy \eqref{datahyp} and \eqref{smallhyp}.
Then  we have the a priori estimates
\begin{align}
&\|\Gamma\|_{L^8(dt)L^{\infty}(d(x-y))L^{\frac{4}{3}}(d(x+y))} \le C \label{apriori1} \\
&\|\nabla_{x+y}\Gamma\|_{L^8(dt)L^{\infty}(d(x-y))L^{\frac{4}{3}}(d(x+y))} \label{apriori2} \le C\\
&\mbox{ and thus also} \notag \\
&\|\big<\nabla_{x+y}\big>^{\alpha}\Gamma\|_{L^8(dt)L^{\infty}(d(x-y))L^{\frac{4}{3}}(d(x+y))} \label{apriori3} \le C.
\end{align}
 (In what follows, $\alpha> \frac{1}{2}$, close to $\frac{1}{2}$, will be fixed.)

Also, there exists $N_0$, and  $ \alpha > \frac{1}{2}$  and $C$ independent of $N$ such that
\begin{align}
&\|\big<\nabla_x\big>^{\alpha}\big<\nabla_y\big>^{\alpha}\Lambda\|_{\mathcal{S}_{x, y}}+\|\big<\nabla_x\big>^{\alpha}\big<\nabla_y\big>^{\alpha}\Gamma\|_{\mathcal{S}_{x, y}}\label{est0}\\
&+ \|\big<\nabla_{x+y}\big>^{\alpha}\Lambda\|_{L^2(dt)L^{\infty}(d(x-y))L^{2}(d(x+y))\label{est1}}\\
&+ \||\partial_t|^{\frac{1}{4}}\Lambda\|_{L^2(dt)L^{\infty}(d(x-y))L^{2}(d(x+y))}\label{est2}\\
&+\sup_{x-y} \||\nabla_{x+y}|^{\alpha}\Gamma\|_{L^2(dt d(x+y))} \le C
\end{align}
for all $N \ge N_0$.
In addition, if
\begin{align*}
\big\|\big<\nabla_x\big>^{\alpha}\big<\nabla_y\big>^{\alpha}|\nabla_{x+y}|^j \Lambda(0, \cdot)\big\|_{L^2}+
\big\|\big<\nabla_x\big>^{\alpha}\big<\nabla_y\big>^{\alpha}|\nabla_{x+y}|^j \Gamma(0, \cdot)\big\|_{L^2} \le C
\end{align*}
for all $j=1, \cdots, j_0$, then also
\begin{align*}
&\|\big<\nabla_x\big>^{\alpha}\big<\nabla_y\big>^{\alpha}|\nabla_{x+y}|^j\Lambda\|_{\mathcal{S}_{x, y}}+\|\big<\nabla_x\big>^{\alpha}\big<\nabla_y\big>^{\alpha}|\nabla_{x+y}|^j\Gamma\|_{\mathcal{S}_{x, y}}\\
&+ \|\big<\nabla_{x+y}\big>^{\alpha+j}\Lambda\|_{L^2(dt)L^{\infty}(d(x-y))L^{2}(d(x+y))}\\
&+ \||\partial_t|^{\frac{1}{4}}|\nabla_{x+y}|^{j}\Lambda\|_{L^2(dt)L^{\infty}(d(x-y))L^{2}(d(x+y))}\\
&+\sup_{x-y} \||\nabla_{x+y}|^{\alpha+j}\Gamma\|_{L^2(dt d(x+y))} \le C.
\end{align*}
The above hold for both the condensate and pair functions.
\end{theorem}
\begin{remark}  We don't know if $\sup_{x-y} \||\nabla_{x+y}|^{\alpha}\Gamma\|_{L^2(dt d(x+y))}$ can be replaced by the stronger norm
$\||\nabla_{x+y}|^{\alpha}\Gamma\|_{L^2(dt)L^{\infty}(d(x-y))L^{2}(d(x+y))}$.\end{remark}
For a proof of \eqref{apriori1} and \eqref{apriori2} see  Lemma \ref{firstsmallness}. These estimates depend on initial conditions of trace class (or in a Schatten space) and could not be true, even for a linear equation, with just $\dot H^s$ initial conditions.  See \cite{Frank-Sabin}, \cite{Du-M}

 We also have a theorem for $\sht$ (without dividing it by $N$):
\begin{theorem}\label{intro2} Let $\Lambda$, $\Gamma$, $\phi$ be solutions of  \eqref{rhs1cut}, \eqref{rhs2cut}, where the potential satisfies \eqref{Vhyp} and the initial conditions satisfy \eqref{datahyp} and \eqref{smallhyp}.
Assume also that
\begin{align*}
\|\sht(0, \cdot, \cdot)\|_{L^2}
+\|(\shb \circ \sh)(0, \cdot, \cdot))\|_{L^2}\le C.
\end{align*}
Then, for all $N \ge N_0$ (as in Theorem \ref{intro1})
\begin{align*}
\|\sht\|_{\mathcal{S}_{x, y}}+ \|\shb \circ \sh\|_{\mathcal{S}_{x, y}}
 \le C \log N.
\end{align*}
Also, assume that for all  $j=1, \cdots, j_0$ we have
\begin{align*}
\||\nabla_{x+y}|^j\sht(0, \cdot, \cdot)\|_{L^2}+\||\nabla_{x+y}|^j(\shb \circ \sh)(0, \cdot, \cdot))\|_{L^2}
\le C.
\end{align*}
Then also
\begin{align*}
\|||\nabla_{x+y}|^j\sht\|_{\mathcal{S}_{x, y}}+ \||\nabla_{x+y}|^j(\shb \circ \sh)\|_{\mathcal{S}_{x, y}} \le C \log N.
\end{align*}
\end{theorem}
\begin{remark} The above estimates also imply some estimates for $\sh$. In particular,
\begin{align}
\|\sh\|_{L^p(dx)L^2(dy)} \le C \|\sht\|_{L^p(dx)L^2(dy)}.
\end{align}
This is because $\sh = \frac{1}{2} \sht \circ \ch^{-1}$ and $\ch^{-1}$ has bounded operator norm.
\end{remark}

Finally, we also have estimates for $\phi$.

Define the standard Strichartz spaces
\begin{align*}
\|\phi\|_{\mathcal S} =  \sup_{  p, q \, \, admissible
}\|\phi\|_{L^{p}(dt) L^{q}(dx)}.
\end{align*}

\begin{corollary}\label{corphi}
 Under the assumptions of Theorem \ref{intro1},
and the additional assumptions $\|\big<\nabla\big>^{\alpha+j}\phi(t=0)\|_{L^2} \le C$ for all $j=0, 1, \cdots, j_0$,
 we have
\begin{align*}
\|\big<\nabla\big>^{\alpha+j}\phi\|_{\mathcal S} \le C
\end{align*}

\end{corollary}

We expect the above theorems to have immediate applications to proving a global improved Fock space estimate. This will be addressed in a different paper. We expect to be able to prove

\begin{align*}
&\|\psi_{exact}(t)-\psi_{approx}(t)\|_{\F}:=\|e^{i t \H}e^{-\sqrt{N}
\A(\phi_{0})}e^{-\B(k(0))}\Omega-
e^{i \chi(t)}e^{-\sqrt{N}\A(\phi(t))}e^{-\B(k(t))}\Omega\|_{\F}\\
&\le \frac{C P(t)}{N^{\frac{1-\beta}{2} }}\notag
\end{align*}
for a polynomial $P(t)$, and $0<\beta <1$. Currently, the best bounds for growth in time for the above construction are
of the form $\frac{e^{C t}}{N^{\frac{1-\beta}{2} }}$. See \cite{C-Z} for the proof and background material.

Finally, we mention the difficulties surrounding equations \eqref{rhs1}, \eqref{rhs2}.

Denote
\begin{align*}
&\S= \frac{1}{i}\frac{\partial}{\partial t} - \Delta_x -\Delta_y\\
&\S_{\pm}= \frac{1}{i}\frac{\partial}{\partial t} - \Delta_x +\Delta_y.
\end{align*}

Schematically, treating $V_N$ as $\delta$ and ignoring constants, the equations become
\begin{align*}
&\S \Lambda_c =\Gamma(t, x, x) \Lambda_c(t, x, y)\}+ \Lambda_p(t, x, x)  \Gamma_c (t, x, y)\\
&\S_{\pm} \Gamma_c=\Gamma(t, x, x) \Gamma_c(t, x, y)\}+\bar\Lambda_p(t, x, x) \Lambda_c(t, x, y)
\\
&\S \Lambda_p
+
\frac{V_N}{N} \Lambda_p
=\Gamma(t, x, x) \Lambda_p(t, x, y)\}+
\Lambda_p(t, x, x)  \Gamma_p (t, x, y)-\frac{V_N}{N} \Lambda_c\\
&+\Lambda_c(t, x, x)  \Gamma_p (t, x, y) \\
&\S_{\pm} \Gamma_p = \Gamma(t, x, x) \Gamma_p(t, x, y)+ \Lambda_p(t, x, x) \Lambda_p (t, x, y)+\bar \Lambda_c(t, x, x)  \Lambda_p (t, x, y).
\end{align*}
Our method for treating the nonlinear terms requires (roughly) Strichartz estimates for $|\nabla_x|^{\frac{1}{2}}|\nabla_y|^{\frac{1}{2}}\Lambda_{p \,\,  or \, c}$, $|\nabla_x|^{\frac{1}{2}}|\nabla_y|^{\frac{1}{2}}\Gamma_{p \, \,  or \, c}$.
But if we apply $|\nabla_x|^{\frac{1}{2}}|\nabla_y|^{\frac{1}{2}}$ to the forcing term $\frac{V_N}{N} \Lambda_c$ in the equation for $\Lambda_p$, we get a singularity which approached $\delta(x-y)\Lambda_c$ which cannot be treated by standard $X^{-\frac{1}{2}}$ type techniques. 

\subsection{Acknowledgment}
The second and third authors thank Daniel Tataru for the suggestion that the resulting singularities mentioned above are sufficiently special that they can be treated by other methods \cite{Tataru}. Also, we thank Xiaoqi Huang for suggesting several improvements to this paper. He is currently working on extending our current results to the case $\beta=1$.

J. Chong was supported by the NSF through the RTG grant DMS-RTG 1840314. Z. Zhao was partially supported by the NSF grant of China (No. 12101046) and the Beijing Institute of Technology Research Fund Program for Young Scholars.

\section{Statement of the main linear estimates }\label{Statement}
 Let $x, y \in \mathbb R^3$, recall $\S=\frac{1}{i}\frac{\partial}{\partial t} - \Delta_x - \Delta_y$ and $h(t)$ denotes the Heaviside function.
  Consider the equation

  \begin{align}
&\S \Lambda(t, x, y) = h(t) \bigg(N^{3 \beta -1} v(N^{\beta}(x-y) \Lambda(t, x, y) +G(t, x, y)\notag\\
&+
 N^{3 \beta -1} v(N^{\beta}(x-y) H(t, x, y)
\label{maineqold}  \bigg)\\
&\Lambda(0, \cdot)=\Lambda_0\notag
\end{align}
for $0 < \beta <1$ with $v$ is Schwartz.

Recall the definition of $\mathcal S_{x, y}$ from \eqref{restr}.
Also define the full Strichartz norm (including
$L^{p}(dt) L^{q}(d(x-y)) L^2(d(x+y))$ )
\begin{align}
&\|\Lambda\|_{\mathcal{S}}\label{restrx-y}\\
&=  \sup_{  p, q \, \, admissible
}\|\Lambda\|_{L^{p}(dt) L^{q}(dx) L^2(dy)}\notag\\
&\quad +\sup_{  p, q \, \notag \, admissible}\|\Lambda\|_{L^{p}(dt) L^{q}(dy) L^2(dx)}\\
&\quad +\sup_{  p, q \, \notag \, admissible}\|\Lambda\|_{L^{p}(dt) L^{q}(d(x-y)) L^2(d(x+y))}
\end{align}
and the restricted dual Strichartz norm , excluding the end-points $p'=2$, $p'=1$: let $p_1$ large and
$p_0> 2$ but close to $2$ as above, and define
\begin{align*}
\|G\|_{\mathcal S'_r}= \inf_{  p, q \, \, admissible, p_1 \ge p \ge p_0>2} \{\|G\|_{L^{p'}(dt) L^{q'}(dx) L^2(dy)},
\|G\|_{L^{p'}(dt) L^{q'}(dy) L^2(dx)}\}.
\end{align*}
The reason for excluding $p'=2$ is that we don't know if we can flip $x$ and $y$ in the double end-point
 case in Theorem \ref{mainStrich}.
The reason $p'=1$ is excluded is the failure of sharp Sobolev estimates in $L^1$, see for instance the proof of Lemma \ref{timederlemma}.

Finally, define the ``collapsing norms"
\begin{align*}
\|\Lambda\|_{collapsing} =  \big\|\Lambda\big\|_{L^{\infty}(d(x-y))L^2(dt) L^2(d(x+y))}.
\end{align*}
We will also use the stronger norms $\big\|\Lambda\big\|_{L^2(dt)L^{\infty}(d(x-y)) L^2(d(x+y))}$.
For the reason we don't work only with this stronger norm see the comments regarding
\eqref{troubleterm1}. For the reason we don't work only with the collapsing norms, see Remark
\eqref{localization}.

The simplest form of our theorem  is

\begin{theorem}\label{mainthm}
Let $\Lambda$ satisfy \eqref{maineqold}, assume $v$  is Schwartz.
Let $0<\beta<1$ and $\alpha >\frac{1}{2}$ is sufficiently close to $\frac{1}{2}$ (so that \eqref{vestim}-\eqref{vestim1} hold).
Then there exists $\epsilon>0$ (depending on $\beta<1$) such that, for $N$ sufficiently large,
\begin{align*}
&\big\|\big<\nabla_x\big>^{\alpha}\big<\nabla_y\big>^{\alpha}\Lambda \big\|_{\mathcal{S}_{x, y}}
+\big\|\big<\nabla_{x+y}\big>^{\alpha}\Lambda\big\|_{L^2(dt)L^{\infty}(d(x-y)) L^2(d(x+y))}\\
&
+\big\|\big|\partial_t\big|^{\frac{1}{4}}\Lambda\big\|_{L^2(dt)L^{\infty}(d(x-y)) L^2(d(x+y))}\\
&\lesssim
\big\|\big<\nabla_x\big>^{\alpha}\big<\nabla_y\big>^{\alpha}G \big\|_{\mathcal{S}_r'}
+N^{-\epsilon}\big\|\big<\nabla_x\big>^{\alpha}\big<\nabla_y\big>^{\alpha}H \big\|_{L^2(dt)L^{6}(d(x-y))L^2(d(x+y))}\\
&+N^{-\epsilon}\big\|\big|\partial_t\big|^{\frac{1}{4}}H\big\|_{collapsing}
+N^{-\epsilon}\big\|\big<\nabla_{x+y}\big>^{\alpha}H\big\|_{collapsing}\\
&+ \big\|\big<\nabla_x\big>^{\alpha}\big<\nabla_y\big>^{\alpha}\Lambda_0 \big\|_{L^2}.
\end{align*}

\end{theorem}
\begin{remark} Notice that the LHS involves the stronger norm \\ $L^2(dt)L^{\infty}(d(x-y)) L^2(d(x+y))$ , while the RHS has the weaker
``collapsing" norm
$L^{\infty}(d(x-y))L^2(dt) L^2(d(x+y))$.
\end{remark}

We first reduce the proof to $\hat v$ compactly supported. Let $0<\epsilon << 1-\beta$ to be chosen below.
Start with $v \in \mathcal S $
 and $\hat \psi \in C_0^{\infty}$, supported in a ball of radius $\frac{1}{10}$, and $\hat \psi=1$
 on a neighborhood of $0$ and define $v_{main}$ and $v_{tail}$ (depending on $N$)
 by $\hat v_{main}= \hat v(\xi) \hat \psi\left(\frac{\xi}{N^{\epsilon}}\right)$ and
$\hat v_{tail}= \hat v(\xi)\left(1- \hat \psi\left(\frac{\xi}{N^{\epsilon}}\right)\right)$. Since $\hat v$ is Schwartz,
for any $p$, $|\xi^{\alpha}D_{\xi}^{\beta}\hat v_{tail}(\xi)| \le C_{p, \alpha, \beta} N^{-p}$.
Thus we also have $|x^{\alpha}D_{x}^{\beta}v_{tail}(x)| \le C_{p, \alpha, \beta} N^{-p}$ (with a different $C_{p, \alpha, \beta} $, of course).

 In all calculations that follow, $N^{3 \beta -1} v_{tail}(N^{\beta}(x-y))$ and its derivatives can be treated as  error terms.

It is also simpler to change the notation to $M=N^{\beta+\epsilon}$. Then the Fourier transform of $v_{main}(N^{\beta} x)$
is $\hat v(\frac{\xi}{N^{\beta}}) \hat \psi\left(\frac{\xi}{N^{\beta+\epsilon}}\right)$ and is supported in $|\xi| < \frac{M}{10}$.
Also, define
\begin{align*}
&v_M(x)=N^{3 \beta -1}v(N^{\beta} x)\\
&v^1_M(x)=N^{3 \beta -1}v_{main}(N^{\beta} x)\\
&v^2_M(x)=N^{3 \beta -1}v_{tail}(N^{\beta} x)
\end{align*}
This definition simplifies the notation during the proof of the main linear theorem (up to section \ref{endlin}).
When we deal with the nonlinear equations (starting with section \ref{scalar}) we will use the notation
$V_N(x)=N^{3 \beta }v(N^{\beta} x)$.

Thus $v^1_M(x)$ is slightly less singular than $M^2 v(Mx)$ as $M \to \infty$, and its Fourier transform
is supported in $|\xi| < \frac{M}{10}$, while
\begin{align}
 \|\big<\nabla\big>^{n}v^2_M\|_{L^p} \le C_{n, p} N^{-10} \label{vestim2}
\end{align}
for any $1 \le p \le \infty$, $n \ge 0$. The reader willing to assume $\hat v$ is compactly supported in a small neighborhood of $0$ can take $M=N$.

At this stage we also choose $\alpha > \frac{1}{2}$,  a number $1+$ (slightly bigger than 1),
a number $\frac{6}{5}+$ (slightly bigger than $\frac{6}{5}+$)
and also $\epsilon_0 >0$ so that
\begin{align}
&\| v^1_M\|_{L^{\frac{3}{2}}} \lesssim M^{-\epsilon_0} \label{vestim}\\
&\|<\nabla>^{\alpha+ \delta_0} v^1_M\|_{L^{\frac{6}{5}+}} \lesssim M^{\alpha+ \delta_0} \|v^1_M\|_{L^{\frac{6}{5}+}} \lesssim M^{-\epsilon_0}\notag\\
&
\|\sqrt N \big<\nabla\big>^{\alpha}v^1_M\|_{L^{1+}}+
\| \big<\nabla\big>^{2 \alpha}v^1_M\|_{L^{1+}}\\
& \lesssim
 \sqrt N M^{ \alpha (1+ \delta_0)}\|v^1_M\|_{L^{1+}}+
M^{2 \alpha (1+ \delta_0)}\|v^1_M\|_{L^{1+}}
 \lesssim M^{-\epsilon_0}.
\label{vestim1}
\end{align}
The above are also true for $v^2_M$, with a bound of $ M^{-n}$ on the right hand side, for any $n$.

All the implicit constants in $\lesssim$ depend on $\beta<1$, (which determines the numbers $\alpha $, $\delta_0$, $\epsilon_0$, $1+$,
$\frac{6}{5}+$ described above), and the exponents $p_1, p_0$ defining  and $\mathcal S_r'$, but are independent of $N$ (for $N$ large).

\section{Estimates in rotated coordinates}
In order to prove Theorem \ref{mainthm}
we will need to adapt standard Sobolev, Bernstein, square function and maximal function estimates to rotated coordinates.

The argument is based on the following lemma:
\begin{lemma}\label{rotationlemma}
Let
\begin{align}
R=\frac{1}{\sqrt 2} \left(
\begin{matrix} \label{Rdef}
1 & 1\\
-1 & 1
\end{matrix}
\right)
\end{align}
(where $1$ stands for the $3 \times 3$ identity matrix) so that $\|f\circ R\|_{L^p(dx)L^q(dy)}=\|f\|_{L^p(d(x-y))L^q(d(x+y))}$.
Let $K$ be a distribution (possibly $l^2$ valued\footnote{In this case, $L^{p_1}(dx)L^q(dy)$ is replaced by $L^{p_1}(dx)L^q(dy)l^2$.}) acting in the $x$ variable, and denote $K \delta=K(x)\delta(y)$ and $\delta K=\delta(x)K(y)$ (tensor products).
Assume the following estimate holds, for some $1\le p_1, p_2, q \le \infty$:
\begin{align*}
\|(K\delta) * f \|_{L^{p_1}(dx)L^q(dy)} \lesssim \|f\|_{L^{p_2}(dx)L^q(dy)}
\end{align*}
Then
\begin{align*}
\|(K\delta) * f \|_{L^{p_1}(d(x-y))L^q(d(x+y))} \lesssim \|f\|_{L^{p_2}(d(x-y))L^q(d(x+y))}
\end{align*}
or, equivalently
\begin{align}
\|\left((K\delta) *(f\circ R^{-1}) \right) \circ R\|_{L^{p_1}(dx)L^q(dy)} \lesssim \|f\|_{L^{p_2}(dx)L^q(dy)}\label{rotatedkernel}
\end{align}
Also,
\begin{align*}
\|(\delta K) * f \|_{L^{p_1}(d(x-y))L^q(d(x+y))} \lesssim \|f\|_{L^{p_2}(d(x-y))L^q(d(x+y))}
\end{align*}
or, equivalently
\begin{align}
\|\left((\delta K) *(f\circ R^{-1}) \right) \circ R\|_{L^{p_1}(dx)L^q(dy)} \lesssim \|f\|_{L^{p_2}(dx)L^q(dy)}\label{rotatedkernel1}.
\end{align}
\end{lemma}
\begin{proof}
In order to prove \eqref{rotatedkernel} we use a nonsingular lower triangular matrix $L_1$
such that
\begin{align}
(RL_1)^{-1}=\left(
\begin{matrix}
1 & a\\
0 & b.
\end{matrix}
\right)
\end{align}
Using the invariance of $L^{p_1}(dx)L^q(dy)$ under transformations given by lower triangular matrices,
\eqref{rotatedkernel}
is equivalent to
\begin{align*}
&\|\left((K \delta) * (f\circ ( RL_1)^{-1})\right)(RL_1(x, y))\|_{L^{p_1}(dx)L^q(dy)}\lesssim \|f\|_{L^{p_2}(dx)L^q(dy)}
\end{align*}
but, by direct calculation (see Lemma \ref{Kf} in the appendix),
\begin{align*}
&\left((K \delta) * (f\circ ( RL_1)^{-1})\right)(RL_1(x, y))=((K \delta)*f)(x, y).\\
\end{align*}
In order to prove \eqref{rotatedkernel1} we use the same argument, based on a nonsingular lower triangular matrix $L_2$
such that
\begin{align}
(RL_2)^{-1}=\left(
\begin{matrix}
c & 1\\
d & 0
\end{matrix}
\right)
\end{align}
and the calculation
\begin{align*}
\left((\delta K ) * (f\circ ( RL_2)^{-1})\right)(RL_2(x, y))=((K \delta)*f)(x, y).
\end{align*}
\end{proof}

A first consequence is the ``Sobolev at an angle" estimate
\begin{lemma} \label{sobangle}
Let $\alpha > 0$, $1 \le p, q, \le \infty$ and assume the Sobolev estimate $\|u\|_{L^p(dx)}\lesssim \|\big<\nabla_{x}\big>^{\alpha}\|_{L^q(dx)}$ holds. Then
\begin{align}
&\|\Lambda\|_{L^{p}(d(x-y))L^2(d(x+y))}\notag\\
& \lesssim \min \{ \|\big<\nabla_{x}\big>^{\alpha}\Lambda\|_{L^{q}(d(x-y))L^2(d(x+y))}
\|\big<\nabla_{y}\big>^{\alpha}\Lambda\|_{L^{q}(d(x-y))L^2(d(x+y))}\}
\label{sobangle1}
\end{align}
and also
\begin{align*}
&\|\Lambda\|_{L^{p}(dx)L^2(dy)}
 \lesssim  \|\big<\nabla_{x+y}\big>^{\alpha}\Lambda\|_{L^{q}(dx))L^2(dy)}.
\end{align*}

\end{lemma}

\begin{proof}
This follows by using $K$ the kernel of $\big<\nabla_{x}\big>^{-\alpha}$.

\end{proof}

Another consequence is Bernstein's inequality in rotated coordinates.

Recall the standard Littlewood-Paley decomposition.
Let $\phi(x)$ such that $\hat \phi \in C_0^{\infty}$ and $\hat \phi(\xi)=1$ in $|\xi|<1$, $\hat \phi(\xi)=0$ in $|\xi|>2$.
Define $\phi_k$ for $k \ge 0$ by $\hat \phi_k(\xi)= \hat\phi(\frac{\xi}{2^k})$ and denote
\begin{align*}
 P_{|\xi| \le 2^k}f = f * \phi_k
\end{align*}
so that the inverse Fourier transform of  $\hat\phi(\frac{\xi}{2^k})\hat f$ is $P_{|\xi| \le 2^k}f$.

Next, let $\psi_0=\phi$ and define $\psi_k$ for $k \ge 1$ by $\hat \psi_k(\xi)= \hat\phi(\frac{\xi}{2^k})-\hat\phi(\frac{\xi}{2^{k-1}})$
We also denote
\begin{align*}
 P_{|\xi| \sim 2^k}f = f * \psi_k
\end{align*}
and note that $\sum_{k=0}^{l} \hat\psi_k(\xi)= \hat\phi(\frac{\xi}{2^l})$
and
\begin{align}
\sum_{k=0}^{l} f* \psi_k \to f
  \label{LPformula}
\end{align}
 in all $L^p$ spaces ($1\le p < \infty$).

More generally, sometimes we will denote by
$\hat \psi_k(\xi)= \hat\psi(\frac{\xi}{2^k})$ for $k \ge 1$ and any $\hat \psi \in C_0^{\infty}(\mathbb R^3)$, vanishing on a neighborhood of $0$. $\phi=\psi_0$ will only be required to have $C_0^{\infty}$ Fourier transform.
In that case \eqref{LPformula} will not be true, but the Bernstein and square function estimates listed below still hold.

 The classical Bernstein inequalities are
\begin{align*}
&\|<\nabla>^{\alpha} (\phi_k *f)\|_{L^p(dx)} \lesssim 2^{\alpha k} \|\phi_k*f\|_{L^p(dx)}\\
&\|<\nabla>^{\alpha} (\psi_k *f)\|_{L^p(dx)} \sim 2^{\alpha k} \|\psi_k*f\|_{L^p(dx)} \, \, (\mbox{ if} \, k \ge 1)
\end{align*}
($\alpha \ge 0$, $1 \le p \le \infty$). See, for instance, \cite{Taobook}. The (elementary) proof immediately implies (for
$1 \le p, q \le \infty$)
\begin{align*}
&\|<\nabla_x>^{\alpha} ((\phi_k \delta) *f)\|_{L^p(dx)L^q(dy)} \lesssim 2^{\alpha k} \|(\phi_k \delta) *f\|_{L^p(dx)L^q(dy)}\\
&\|<\nabla_x>^{\alpha} ((\psi_k \delta) *f)\|_{L^p(dx)L^q(dy)} \sim 2^{\alpha k} \|(\psi_k \delta) *f\|_{L^p(dx)L^q(dy)}.
\end{align*}
Using Lemma \ref{rotationlemma} we get

\begin{lemma} \label{bernstein}
The following estimates hold
\begin{align*}
&\|<\nabla_x>^{\alpha} ((\phi_k\delta) *f)\|_{L^p(d(x-y))L^q(d(x+y)} \lesssim 2^{\alpha k} \|( \phi_k\delta )*f\|_{L^p(d(x-y))L^q(d(x+y))}\\
&\|<\nabla_y>^{\alpha} ((\delta\phi_k) *f)\|_{L^p(d(x-y))L^q(d(x+y)} \lesssim 2^{\alpha k} \|(\delta \phi_k)*f\|_{L^p(d(x-y))L^q(d(x+y))}\\
&\|<\nabla_x>^{\alpha} ((\psi_k\delta) *f)\|_{L^p(d(x-y))L^q(d(x+y)} \sim 2^{\alpha k} \|( \psi_k\delta )*f\|_{L^p(d(x-y))L^q(d(x+y))}\\
&\|<\nabla_y>^{\alpha} ((\delta\psi_k) *f)\|_{L^p(d(x-y))L^q(d(x+y)} \sim 2^{\alpha k} \|(\delta \psi_k)*f\|_{L^p(d(x-y))L^q(d(x+y))}
\end{align*}
\end{lemma}
Finally, we state two square function estimates.
For a function depending only on $x$, the classical estimate is (for $1<p<\infty$)
\begin{align*}
\|\left(\sum_{k=0}^{\infty} |f*\psi_{k} |^2\right)^{\frac{1}{2}}\|_{L^p(dx)}\sim \|f\|_{L^p(dx)}.
\end{align*}
The proof can be modified to apply to $L^2$ valued functions and we have
\begin{lemma} \label{square2}
Let $1<p<\infty$ .
Define $\F \left(P_{|\xi-\eta|\sim 2^k}f \right)= \hat f(\xi, \eta)\hat \psi\left(\frac{\xi-\eta}{2^k}\right)$ for $k \ge 1$
and $\F \left(P_{|\xi-\eta|\sim 2^0}f \right)= \hat f(\xi, \eta)\hat \phi\left(\xi-\eta\right)$.
Then the following estimate holds (for functions which also depend on $t$)
\begin{align*}
\|\left(\sum_{k=0}^{\infty} |P_{|\xi-\eta|\sim 2^k}f  )|^2\right)^{\frac{1}{2}}\|_{L^p(d(x-y))L^2(d(x+y)dt)}\sim \|f\|_{L^p(d(x-y))L^2(d(x+y)dt)}.
\end{align*}
\end{lemma}
The proof is the same as the standard square function estimate.

Also, we have a result for a ``double square function" in rotated coordinates:

\begin{lemma} \label{square1}
Let $1<p<\infty$ . Then the following estimate holds
\begin{align*}
\|\left(\sum_{k', k''=0}^{\infty} |f*(\psi_{k'} \delta)*(\delta \psi_{k''} )|^2\right)^{\frac{1}{2}}\|_{L^p(d(x-y))L^2(d(x+y))}\sim \|f\|_{L^p(d(x-y))L^2(d(x+y))}.
\end{align*}
\end{lemma}
\begin{proof} For $\lesssim$ see Lemma \ref{squarerotated} in the Appendix. The opposite inequality is a standard duality argument.
\end{proof}

\section{Preliminary estimates for solutions to the linear Schr\"odinger equation  }

We will use the following Strichartz estimate
(proved in Theorem 2.4, 2.5 of \cite{CGMZ}     ).
In $6+1$ dimensions,
 \begin{theorem} \label{mainStrich}
 Let $\S u= f+g$, $u(0, \cdot)=u_0$. Then
\begin{align*}
&\|u\|_{\mathcal S} \lesssim \|f\|_{L^2(dt) L^{\frac{6}{5}}(d(x-y)) L^2(d(x+y))}
+\|g\|_{\mathcal{S}_r'} + \|u_0\|_{L^2}.
\end{align*}
\end{theorem}
In the applications that follow, $u$ will be $\Lambda$ (or $\Lambda_p$ or $\Lambda_c$, or suitable fractional derivatives of $\Lambda$),  $f$ will be
$v_M(x-y)\Lambda(x, y)$ (or suitable derivatives) and $g$ will be $G$ (or suitable derivatives).

After our paper \cite{CGMZ}  was published, we learned about \cite{Hong} which contains closely related results (proved with different methods).

\begin{remark} Another way of obtaining Strichartz estimates on the LHS will be given in Propositions \ref{Suufixed}, \ref{Suusharp} below.
\end{remark}

Using Theorem \ref{mainStrich} and  Lemma \ref{sobangle}, we can get a (non-sharp) collapsing estimate.
\begin{lemma} \label{collapsinglemma}
If $\S u= f +g$, $u(0, \cdot)=u_0$, and let $\alpha> \frac{1}{2}$. Then
\begin{align*}
&  \|u\|_{L^2(dt)L^{\infty}(d(x-y))L^2(d(x+y))}\\
&\lesssim \min \{\|\big<\nabla_x\big>^{\alpha}f\|_{L^2(dt)L^{\frac{6}{5}}(d(x-y))L^2(d(x+y))}+
\|\big<\nabla_x\big>^{\alpha}g\|_{\mathcal{S}_r'}
+\|\big<\nabla_x\big>^{\alpha}u_0\|_{L^2}
, \\
&\|\big<\nabla_y\big>^{\alpha}f\|_{L^2(dt)L^{\frac{6}{5}}(d(x-y))L^2(d(x+y))}
+\|\big<\nabla_y\big>^{\alpha}g\|_{\mathcal{S}_r'}+\|\big<\nabla_y\big>^{\alpha}u_0\|_{L^2}.
\}
\end{align*}
\end{lemma}

\begin{proof}

\begin{align*}
& \|u\|_{L^2(dt)L^{\infty}(d(x-y))L^2(d(x+y))} \lesssim \|\big<\nabla_x\big>^{\alpha}u\|_{L^2(dt)L^6(d(x-y))L^2(d(x+y))}\\
&\lesssim \|\big<\nabla_x\big>^{\alpha}f\|_{L^2(dt)L^{\frac{6}{5}}(d(x-y))L^2(d(x+y))}\\
& +
\|\big<\nabla_x\big>^{\alpha}g\|_{\mathcal{S}_r'}
+\|\big<\nabla_x\big>^{\alpha}u_0\|_{L^2},
\end{align*}
and, of course, $\big<\nabla_x\big>^{\alpha}$ can be replaced by $\big<\nabla_y\big>^{\alpha}$.

\end{proof}

We  record  that the above  implies
\begin{lemma}\label{3collapsing}
If
 $\S u= f+g$, $u(0, \cdot)=u_0$. Then
\begin{align*}
& \|\big<\nabla_{x}\big>^{\alpha}u\|_{L^2(dt)L^{\infty}(d(x-y))L^2(d(x+y))}
+\|\big<\nabla_{y}\big>^{\alpha}u\|_{L^2(dt)L^{\infty}(d(x-y))L^2(d(x+y))}\\
& \lesssim
 \|\big<\nabla_x\big>^{\alpha}\big<\nabla_y\big>^{\alpha}f\|_{L^2(dt)L^{\frac{6}{5}}(d(x-y))L^2(d(x+y))}+
 \|\big<\nabla_x\big>^{\alpha}\big<\nabla_y\big>^{\alpha}g\|_{\mathcal{S}_r'}\\
 &+\|\big<\nabla_x\big>^{\alpha}\big<\nabla_y\big>^{\alpha}u_0\|_{L^2}.
\end{align*}
\end{lemma}

We will also need

\begin{lemma}\label{4collapsing}
If
 $\S u= f+g$, $u(0, \cdot)=u_0$. Then
\begin{align*}
&\|\big<\nabla_{x+y}\big>^{\alpha}u\|_{L^2(dt)L^{\infty}(d(x-y))L^2(d(x+y))}\\
& \lesssim
 \|\big<\nabla_x\big>^{\alpha}\big<\nabla_y\big>^{\alpha}f\|_{L^2(dt)L^{\frac{6}{5}}(d(x-y))L^2(d(x+y))}+
 \|\big<\nabla_x\big>^{\alpha}\big<\nabla_y\big>^{\alpha}g\|_{\mathcal{S}_r'}\\
 & +\|\big<\nabla_x\big>^{\alpha}\big<\nabla_y\big>^{\alpha}u_0\|_{L^2}.
\end{align*}
\end{lemma}

\begin{proof}

The estimate for the homogeneous equation follows from standard Strichartz estimates and ``Sobolev at an angle", as in the proof for the inhomogeneous estimate.
Thus we can assume $u_0=0$.
Let $u= \sum_{k, k'=0}^{\infty} u_{k, k'}$ be a ``double" Littlewood-Paley decomposition
$u = u*(\psi_k \delta) * (\delta \psi_{k'})$ (see Lemma \ref{square1})
so that
$\hat u_{k, k'}(\xi, \eta)$ is supported in $|\xi| \sim 2^k$, $|\eta| \sim 2^{k'}$ if $k, k' \ge 1$. We only treat the sum over $k \le k'$ with $k'\ge 1$ (so $x$ corresponds to the ``low" frequency), the remaining part being similar.
We use the standard procedure of reducing a Strichartz estimate to a frequency localized estimate, but in the context of mixed coordinates.
We have
\begin{align*}
&\|\big<\nabla_{x+y}\big>^{\alpha}u\|_{L^2(dt)L^{\infty}(d(x-y))L^2(d(x+y))}\\
&\lesssim \|\big<\nabla_{x+y}\big>^{\alpha}\big<\nabla_{x}\big>^{\alpha}u\||_{L^2(dt)L^{6}(d(x-y))L^2(d(x+y))} \, \, \mbox{ (Lemma \ref{sobangle})}\\
&=\|\big<\nabla_{x+y}\big>^{\alpha}\big<\nabla_{x}\big>^{\alpha}\sum_{0 \le k \le k'} u_{k, k'}\|_{L^2(dt)L^{6}(d(x-y))L^2(d(x+y))}\\
&\lesssim \bigg\|\left(\sum_{0 \le k \le k'}\big| \big<\nabla_{x+y}\big>^{\alpha}\big<\nabla_{x}\big>^{\alpha}u_{k, k'}\big|^2\right)^{\frac{1}{2}}\bigg\|_{L^2(dt)L^{6}(d(x-y))L^2(d(x+y))} \, \, \mbox{ (square function estimate)}\\
&\lesssim\left(\sum_{0 \le k \le k'} \bigg\| \big<\nabla_{x+y}\big>^{\alpha}\big<\nabla_{x}\big>^{\alpha}u_{k, k'}\bigg\|_{L^2(dt)L^{6}(d(x-y))L^2(d(x+y))}^2\right)^{\frac{1}{2}} \, \, \mbox{ (Minkowski)}\\
&=\left(\sum_{0 \le k \le k'} \bigg\| \big<\nabla_{x+y}\big>^{\alpha}\big<\nabla_{x}\big>^{\alpha}P_{|\xi+\eta| \lesssim 2^{k'}}u_{k, k'}\bigg\|_{L^2(dt)L^{6}(d(x-y))L^2(d(x+y))}^2\right)^{\frac{1}{2}} \\
&\lesssim\left(\sum_{0 \le k \le k'} \bigg\| 2^{k' \alpha} \big<\nabla_{x}\big>^{\alpha}u_{k, k'}\bigg\|_{L^2(dt)L^{6}(d(x-y))L^2(d(x+y))}^2\right)^{\frac{1}{2}} \, \, \mbox{ (Plancherel)}\\
&\lesssim\left(\sum_{0 \le k \le k'} \bigg\|  \big<\nabla_{x}\big>^{\alpha}
\big<\nabla_{y}\big>^{\alpha}
u_{k, k'}\bigg\|_{L^2(dt)L^{6}(d(x-y))L^2(d(x+y))}^2\right)^{\frac{1}{2}} \, \, \mbox{ (Lemma \ref{bernstein} )}
\end{align*}
\begin{align*}
&\lesssim \bigg(\sum_{0 \le k \le k'} \bigg\|  \big<\nabla_{x}\big>^{\alpha}
\big<\nabla_{y}\big>^{\alpha}
f_{k, k'}\bigg\|_{L^2(dt)L^{\frac{6}{5}}(d(x-y))L^2(d(x+y))}^2\bigg)^{\frac{1}{2}} \\
&+\bigg(\sum_{0 \le k \le k'} \bigg\|  \big<\nabla_{x}\big>^{\alpha}
\big<\nabla_{y}\big>^{\alpha}
g_{k, k'}\bigg\|_{\mathcal S'_r}^2\bigg)^{\frac{1}{2}} \, \, \mbox{ (Thm. \ref{mainStrich})}\\
&\lesssim \bigg\|\left(\sum_{0 \le k \le k'}\big| \big<\nabla_{x}\big>^{\alpha}\big<\nabla_{y}\big>^{\alpha}f_{k, k'}\big|^2\right)^{\frac{1}{2}}\bigg\|_{L^2(dt)L^{\frac{6}{5}}(d(x-y))L^2(d(x+y))} \\
&+\bigg\|\left(\sum_{0 \le k \le k'}\big| \big<\nabla_{x}\big>^{\alpha}\big<\nabla_{y}\big>^{\alpha}g_{k, k'}\big|^2\right)^{\frac{1}{2}}\bigg\|_{\mathcal S'_r} \, \, \mbox{ (Minkowski)}\\
&\lesssim \|  \big<\nabla_{x}\big>^{\alpha}
\big<\nabla_{y}\big>^{\alpha}f\|_{L^2(dt)L^{\frac{6}{5}}(d(x-y))L^2(d(x+y))}
+ \|  \big<\nabla_{x}\big>^{\alpha}
\big<\nabla_{y}\big>^{\alpha}g\|_{\mathcal S'_r}\, \, \mbox{ (square function estimate)}.
\end{align*}

\end{proof}

\begin{lemma} \label{timederlemma}
Let
 $\S u= f+g$, $u(0, \cdot)=u_0$, and let $\frac{1}{2}   <\alpha$. Then
\begin{align}
&\||\partial_t|^{\frac{1}{4}}u\|_{L^2(dt)L^{\infty}(d(x-y))L^2(d(x+y))}\label{timeder}
 \lesssim \|\big<\nabla_x\big>^{\alpha}\big<\nabla_y\big>^{\alpha}f\|_{L^2(dt)L^{\frac{6}{5}}(d(x-y))L^2(d(x+y))}\\
 &+
 \|\big<\nabla_x\big>^{\alpha}\big<\nabla_y\big>^{\alpha}g\|_{\mathcal{S}'_r}
 +\|\big<\nabla_x\big>^{\alpha}\big<\nabla_y\big>^{\alpha}u_0\|_{L^2}.
\notag
\end{align}

\end{lemma}

\begin{proof} The estimate for the homogeneous equation follows from Strichartz estimates and the fact that $|\tau|=|\xi|^2 + |\eta|^2$ on the Fourier support of $u$,
 so we can assume $u_0=0$.
 Let $u= \sum_{k, k'=0}^{\infty} u_{k, k'}$ be a "double" Littlewood-Paley decomposition
$u = u*(\psi_k \delta) * (\delta \psi_{k'})$
so that
$\hat u_{k, k'}(\xi, \eta)$ is supported in $|\xi| \sim 2^k$, $|\eta| \sim 2^{k'}$ if $k, k' \ge 1$. The proof of \eqref{timeder}
will use two additional numbers $\alpha'$, $\alpha''$ satisfying
$\frac{1}{2} <\alpha'<\alpha'' <\alpha$. Start by fixing $\frac{1}{2} <\alpha'' <\alpha$.
We will  prove the frequency localized estimate
\begin{align}
&\||\partial_t|^{\frac{1}{4}}u_{k, k'}\|_{L^2(dt)L^{\infty}(d(x-y))L^2(d(x+y))} \label{timederloc}\\
 &\lesssim \|\big<\nabla_x\big>^{\alpha''}\big<\nabla_y\big>^{\alpha''}f_{k, k'}\|_{L^2(dt)L^{\frac{6}{5}}(d(x-y))L^2(d(x+y))}+
 \|\big<\nabla_x\big>^{\alpha''}\big<\nabla_y\big>^{\alpha''}g_{k, k'}\|_{\mathcal{S}_r'}\notag.
\end{align}

Summing the pieces will be easy because $\alpha'' <\alpha$.
To prove \eqref{timederloc}, assume, without loss of generality, $1 \le k \le k'$.

    Now we localize $u_{k, k'}$ in $\tau$. This changes the initial conditions, but in a controlled way. From Theorem \ref{mainStrich} we have
     \begin{align*}
      & \|\big<\nabla_x\big>^{\alpha''}\big<\nabla_y\big>^{\alpha''}u_{k, k'}\|_{L^{\infty}(dt)L^2(dxdy)}\\
&\lesssim \|\big<\nabla_x\big>^{\alpha''}\big<\nabla_y\big>^{\alpha''}f_{k, k'}\|_{L^2(dt)L^{\frac{6}{5}}(d(x-y))L^2(d(x+y))}+
     \|\big<\nabla_x\big>^{\alpha''}\big<\nabla_y\big>^{\alpha''}g_{k, k'}\|_{\mathcal{S}_r'}.
    \end{align*}
Since $P_{|\tau| < 100 \, 2^{2 k'}}$ acts by convolution in time with a function which is in $L^1$ uniformly in $k'$, we also have
 \begin{align*}
 & \|\big<\nabla_x\big>^{\alpha''}\big<\nabla_y\big>^{\alpha''}P_{|\tau| < 100 \, 2^{2 k'}}u_{k, k'}(0, \cdot, \cdot)\|_{L^2(dxdy)}\\
&\lesssim \|\big<\nabla_x\big>^{\alpha''}\big<\nabla_y\big>^{\alpha''}f_{k, k'}\|_{L^2(dt)L^{\frac{6}{5}}(d(x-y))L^2(d(x+y))}+
     \|\big<\nabla_x\big>^{\alpha''}\big<\nabla_y\big>^{\alpha''}g_{k, k'}\|_{\mathcal{S}_r'}.
    \end{align*}
We have
\begin{align*}
\S P_{|\tau| < 100 \, 2^{2 k'}}u_{k, k'}=P_{|\tau| < 100 \, 2^{2 k'}}f_{k, k'} + P_{|\tau| < 100 \, 2^{2 k'}}g_{k, k'}
\end{align*}
with initial conditions as discussed above.
Using Lemma \ref{bernstein} we get
\begin{align*}
&
\||\partial_t|^{\frac{1}{4}}P_{|\tau| < 100 \, 2^{2 k'}}u_{k, k'}\|_{L^2(dt)L^{\infty}(d(x-y))L^2(d(x+y))}\\
& \lesssim 2^{\frac{ k'}{2}}\||P_{|\tau| < 100 \, 2^{2 k'}}u_{k, k'}\|_{L^2(dt)L^{\infty}(d(x-y))L^2(d(x+y))}\\
&\lesssim
\|\big<\nabla_{y}\big>^{\frac{ 1}{2}}P_{|\tau| < 100 \, 2^{2 k'}}u_{k, k'}\|_{L^2(dt)L^{\infty}(d(x-y))L^2(d(x+y))}\\
& \lesssim
 \|\big<\nabla_x\big>^{\alpha'}\big<\nabla_y\big>^{\alpha'}
 f_{k, k'}\|_{L^2(dt)L^{\frac{6}{5}}(d(x-y))L^2(d(x+y))}\\
 &+
 \|\big<\nabla_x\big>^{\alpha'}\big<\nabla_y\big>^{\alpha'}g_{k, k'}\|_{\mathcal{S}_r'}.
\end{align*}
(See Lemma \ref{3collapsing}.)

Next, consider
\begin{align*}
\S P_{|\tau| > 100 \, 2^{2 k'}}u_{k, k'}=P_{|\tau| > 100 \, 2^{2 k'}}f_{k, k'} + P_{|\tau| > 100 \, 2^{2 k'}}g_{k, k'}
\end{align*}
and $2^k$ (the frequency of $x$) is less than $2^{k'}$ (the frequency of $y$).

Call either function on the RHS $P_{|\tau| > 100 \, 2^{2 k'}}h_{k, k'}$. The point is that $\S \sim \big<\partial_t\big> $ is an elliptic operator on the Fourier support of $u_{k, k'}$, and, at the level of symbols and on $L^2(dt dx dy)$,
$\big<\partial_t\big> \ge \big<\nabla_y\big>^2 \ge \big<\nabla_x\big>^2 $.

Let $\frac{1}{2}< \alpha'<\alpha''<\alpha$, with $\alpha'$ to be chosen later (the choice will be $\alpha'=\frac{\frac{1}{2}+ 2 \alpha''}{3}$).
Using Lemma \ref{sobangle},
\begin{align*}
&\|\big<\partial_t\big>^{\frac{1}{4}}P_{|\tau| > 100 \, 2^{2 k'}}u_{k, k'}\|_{L^2(dt)L^{\infty}(d(x-y))L^2(d(x+y))}\\
& \lesssim \|\big<\partial_t\big>^{\frac{1}{4}}\big<\nabla_x\big>^{3 \alpha'}P_{|\tau| > 100 \, 2^{2 k'}}u_{k, k'}\|_{L^2(dt dx dy)}\\
&\lesssim \|\big<\partial_t\big>^{\frac{1}{4}-1}\big<\nabla_x\big>^{3 \alpha'}\S P_{|\tau| > 100 \, 2^{2 k'}} u_{k, k'}\|_{L^2(dt dx dy)}\\
&= \|\big<\partial_t\big>^{-\frac{3}{4}}\big<\nabla_x\big>^{3 \alpha'} P_{|\tau| > 100 \, 2^{2 k'}} h_{k, k'}\|_{L^2(dt dx dy)}\\
&=\|\big<\partial_t\big>^{-\frac{1}{2}}\bigg(\big<\partial_t\big>^{-\frac{1}{4}}\big<\nabla_x\big>^{3 \alpha'-2\alpha''}\bigg)\big<\nabla_x\big>^{ \alpha''}\big<\nabla_y\big>^{ \alpha''} P_{|\tau| > 100 \, 2^{2 k'}} h_{k, k'}\|_{L^2(dt dx dy)}\\
&\lesssim \|\big<\partial_t\big>^{-\frac{1}{2}}\big<\nabla_x\big>^{ \alpha''}\big<\nabla_y\big>^{ \alpha''} P_{|\tau| > 100 \, 2^{2 k'}} h_{k, k'}\|_{L^2(dt dx dy)}
\end{align*}
since the choice of $\alpha'$ insures $3 \alpha'-2\alpha''=\frac{1}{2}$ so that
$\big<\partial_t\big>^{-\frac{1}{4}}\big<\nabla_x\big>^{3 \alpha'-2\alpha''}$ is bounded on $L^2$ on the Fourier support of $P_{|\tau| > 100 \, 2^{2 k'}} h_{k, k'}$.

By applying Sobolev estimates in $t$ or $x$, $y$ or $x-y$ , and recalling
$\big<\partial_t\big> \ge \big<\nabla_y\big>^2 \ge \big<\nabla_x\big>^2 $, $\big<\partial_t\big> \ge \big<\nabla_{x-y}\big>^2$
we get
\begin{align*}
&\|\big<\partial_t\big>^{-\frac{1}{2}}\big<\nabla_x\big>^{ \alpha''}\big<\nabla_y\big>^{ \alpha''}
P_{|\tau| > 100 \, 2^{2 k'}}h_{k, k'}\|_{L^2(dt dx dy)}\\
&\lesssim \min \bigg\{\|\big<\nabla_x\big>^{ \alpha''}\big<\nabla_y\big>^{ \alpha''}
P_{|\tau| > 100 \, 2^{2 k'}}h_{k, k'}\|_{\mathcal{S'}_r},\\
&
\|\big<\nabla_x\big>^{ \alpha''}\big<\nabla_y\big>^{ \alpha''}
P_{|\tau| > 100 \, 2^{2 k'}}h_{k, k'}\|_{L^2(dt)L^{\frac{6}{5}}(d(x-y))L^2(d(x+y))}\bigg \}.
\end{align*}
Recalling the definition of $h_{k, k'}$ the above is dominated by
$$
\|\big<\nabla_x\big>^{ \alpha''}\big<\nabla_y\big>^{ \alpha''}
P_{|\tau| > 100 \, 2^{2 k'}}f_{k, k'}\|_{_{L^2(dt)L^{\frac{6}{5}}(d(x-y))L^2(d(x+y))}}+
\|\big<\nabla_x\big>^{ \alpha''}\big<\nabla_y\big>^{ \alpha''}
P_{|\tau| > 100 \, 2^{2 k'}}g_{k, k'}\|_{\mathcal{\S}'_r}$$.

To sum the pieces, use $\alpha'' < \alpha$.  For instance,
\begin{align*}
&\sum_{k, k'=0}^{\infty} \|\big<\nabla_x\big>^{\alpha''}\big<\nabla_y\big>^{\alpha''}f_{k, k'}\|_{L^2(dt)L^{\frac{6}{5}}(d(x-y))L^2(d(x+y))}\\
&\lesssim
\sum_{k, k'=0}^{\infty} 2^{(\alpha''-\alpha)k} 2^{(\alpha''-\alpha)k'} \|\big<\nabla_x\big>^{\alpha}\big<\nabla_y\big>^{\alpha}(f*(\psi_k \delta)*(\delta \psi_{k'}) )\|_{L^2(dt)L^{\frac{6}{5}}(d(x-y))L^2(d(x+y))}\\
&\lesssim \sup_{k, k'}\|\big<\nabla_x\big>^{\alpha}\big<\nabla_y\big>^{\alpha}f*(\psi_k^1 \delta)*(\delta \psi_{k'})\|_{L^2(dt)L^{\frac{6}{5}}(d(x-y))L^2(d(x+y))}\\
&\lesssim \|\big<\nabla_x\big>^{\alpha}\big<\nabla_y\big>^{\alpha}f\|_{L^2(dt)L^{\frac{6}{5}}(d(x-y))L^2(d(x+y))}.
\end{align*}

\end{proof}
Next, we prove a frequency localized sharp result.

\begin{proposition} \label{Suufixed}
Let
\begin{align*}
\S u_k=f_k
\end{align*}
with $0$ initial conditions, and assume  $u_k$ (and thus also $f_k$) is supported, in Fourier space, at $|\xi-\eta| \sim 2^k$. Then
\begin{align*}
&\|u_k\|_{\mathcal {S}_{x, y}} \lesssim
\big\||\nabla_{x+y}|^{\frac{1}{2}}f_k\big\|_{L^{1}(d(x-y))L^2(dt) L^2(d(x+y))} + \big\||\partial_t\big|^{\frac{1}{4}}f_k\big\|_{L^{1}(d(x-y))L^2(dt) L^2(d(x+y))}.
\end{align*}

\end{proposition}
\begin{proof}
 As we did earlier, we decompose $u_k=u_k^1+u_k^2+u_k^3$, where
\begin{align*}
&\S u_k^1=P_{10|\tau|^{\frac{1}{2}} \ge  2^k} f_k:=f_k^1 \, \, \mbox{ with initial conditions $0$}\\
&\F u_k^2=\frac{ \F \left(P_{10 |\tau|^{\frac{1}{2}} \le  2^k} f_k\right)}{\tau+ |\xi|^2+|\eta|^2}
\mbox{ (this no longer has initial conditions $0$)}
\\
&\S u_k^3=0,  \, \, \mbox{a correction so that }\,  u_k^2+u_k^3 \, \,  \mbox{has initial conditions $0$}.
\end{align*}

For $u_k^1$ the argument is based on the Strichartz and Sobolev estimates at fixed frequency:

\begin{align*}
&\|u_k^1\|_{\mathcal{S}} \lesssim \|f_k^1\|_{L^2(dt) L^{\frac{6}{5}}(d(x-y))L^2(d(x+y))} \le
\|f_k^1\|_{L^{\frac{6}{5}}(d(x-y))L^2(dt) L^2(d(x+y))}\\
&\lesssim \||\nabla_{x-y}|^{\frac{1}{2}}f_k^1\|_{L^{1}(d(x-y))L^2(dt) L^2(d(x+y))}\lesssim
2^{\frac{k}{2}}\|f_k^1\|_{L^{1}(d(x-y))L^2(dt) L^2(d(x+y))}\\
&\lesssim
\||\partial_t|^{\frac{1}{4}}f_k^1\|_{L^{1}(d(x-y))L^2(dt) L^2(d(x+y))}\\
&\lesssim
\||\partial_t|^{\frac{1}{4}}f_k\|_{L^{1}(d(x-y))L^2(dt) L^2(d(x+y))}.
\end{align*}
We  used Plancherel and the Fourier support of $f_k^1$.

For $ u_k^2$, the denominator is comparable with $|\xi-\eta|^2+|\xi+\eta|^2 \ge 2^{2k} \ge 100 |\tau|$, and
\begin{align*}
\|u_k^2\|_{L^2(dt dx dy) }\lesssim \||\nabla_{x-y}|^{-\frac{3}{2}}|\nabla_{x+y}|^{-\frac{1}{2}}f_k\|_{L^2(dt dx dy) }
\end{align*}
and
\begin{align*}
&\||\nabla_{x+y}|u_k^2\|_{L^2(dt dx dy) }\lesssim \||\nabla_{x-y}|^{-\frac{3}{2}}|\nabla_{x+y}|^{\frac{1}{2}}f_k\|_{L^2(dt dx dy) }\\
&\lesssim \||\nabla_{x+y}|^{\frac{1}{2}}f_k\|_{L^{1}(d(x-y))L^2(dt d(x+y)) }.
\end{align*}
Thus, using Lemma \ref{sobangle} (Sobolev estimates at an angle)
we have
\begin{align}
&\|u_k^2\|_{L^2(dt)L^6(dx)L^2(dy)}
+\|u_k^2\|_{L^2(dt)L^6(dy)L^2(dx)}
\lesssim \||\nabla_{x+y}|u_k^2\|_{L^2(dt dx dy) }\notag\\
&\lesssim \label{unloc1}
\||\nabla_{x+y}|^{\frac{1}{2}}f_k\|_{L^{1}(d(x-y))L^2(dt d(x+y)) }.
\end{align}
Similarly, we have
\begin{align*}
&\||\partial_t|^{\frac{1}{2}}u_k^2\|_{L^2(dt dx dy) }\lesssim \||\nabla_{x-y}|^{-\frac{3}{2}}|\partial_t|^{\frac{1}{4}}f_k\|_{L^2(dt dx dy) }\\
&\lesssim \||\partial_t|^{\frac{1}{4}}f_k\|_{L^{1}(d(x-y))L^2(dt d(x+y)) }.
\end{align*}
Unfortunately, the desired $L^{\infty}(dt)$ Sobolev estimate is false, so we proceed slightly differently:

 \begin{align*}
 &\|u_k^2\|_{L^{\infty}(dt)L^2(dxdy)} \lesssim 2^{-\frac{k}{2}}
 \big\| \int_{\tau \in [-2^{2 k}, 2^{2 k}]} \frac{1}{|\tau|^{\frac{1}{4}}}|\tau|^{\frac{1}{4}}|\xi-\eta|^{- \frac{3}{2}} |\widetilde{f_k}| d \tau \big\|_{L^2(d\xi d \eta)}\\
& \lesssim  \|||\tau|^{\frac{1}{4}}|\xi-\eta|^{- \frac{3}{2}}|\widetilde{f_k}\|_{L^2(\tau \in [-2^{2 k}, 2^{2 k}] d(\xi d\eta))}\\
&= c\||\nabla_{x-y}|^{-\frac{3}{2}}|\partial_{t}\big|^{\frac{1}{4}}f_k\|_{L^2(dt dx dy) }
  \lesssim
  \|\big|\partial_{t}\big|^{\frac{1}{4}}f_k\|_{L^{1}(d(x-y))L^2(dt d(x+y)) }.
 \end{align*}

Finally, by interpolation,
\begin{align*}
&\|u_k^2\|_{\mathcal{S}_{x, y}} \lesssim
\|u_k^2\|_{L^2(dt)L^6(dx)L^2(dy)}
+\|u_k^2\|_{L^2(dt)L^6(dy)L^2(dx)}+|u_k^2\|_{L^{\infty}(dt)L^2(dxdy)}
\end{align*}
and we get the desired result for $u_k^2$.
Since $\|u_k^2(t=0)\|_{L^2}=\|u_k^3(t=0)\|_{L^2}$, the result for $u_k^3$ is trivial.

\end{proof}

We also record the following version:
\begin{proposition} \label{Suusharp}
Let
\begin{align*}
\S u=N^3v(N(x-y)\Lambda:=f
\end{align*}
with $0$ initial conditions and $v\in \mathcal S$.
\begin{align*}
&\|u\|_{\mathcal {S}_{x, y}} \lesssim \log N \left(
\big\|\big<\nabla_{x+y}\big>^{\frac{1}{2}}\Lambda\big\|_{collapsing} + \big\||\partial_t|^{\frac{1}{4}}\Lambda\big\|_{collapsing}\right).
\end{align*}
The same result holds for the equation
\begin{align*}
\S u=P_{|\xi|<M}P_{|\eta|<M}\left(N^3v(N(x-y)\Lambda\right)
\end{align*}
for all $M>0$, with implicit constants independent of $M$.
\end{proposition}
(see \eqref{projdef} for the definition of $P_{|\xi|<M}P_{|\eta|<M}$).
\begin{proof}
First we consider $P_{|\xi-\eta|>N} u$, and in this case we get the result without a $\log$.
In that case
\begin{align*}
\|u\|_{\mathcal{S}_{x, y}} \lesssim \|f\|_{L^2(dt)L^{\frac{6}{5}}(d(x-y))L^2(d(x+y))} \lesssim N^{\frac{1}{2}}\|\Lambda\|_{collapsing}.
\end{align*}

If we localize to the region where $|\tau|^{\frac{1}{2}} +|\xi+\eta|>\frac{N}{10}$, the result follows from Plancherel.
Thus we can assume $\tilde u$ is supported in $|\tau|^{\frac{1}{2}} +|\xi+\eta|<\frac{N}{10}$, $|\xi-\eta|>N$ and in particular we can solve $\S$ by dividing by the symbol, as in the previous proof. Let $u_1$ be the solution (which differs from $u$ by a solution to the homogeneous equation). Then we get
\begin{align*}
\|u_1\|_{L^2(dt dx dy)} \lesssim \frac{1}{N^2} \|f\|_{L^2} \lesssim \frac{1}{\sqrt N} \|\Lambda\|_{collapsing}
\end{align*}
and also
\begin{align*}
&\|u_1\|_{L^2(dt) L^6(dx)L^2 (dy)}+
\|u_1\|_{L^2(dt) L^6(dy)L^2 (dx)}\lesssim \||\nabla_{x+y}|u_1\|_{L^2(dt dx dy)}\\
&  \lesssim \frac{1}{\sqrt N} \||\nabla_{x+y}|\Lambda\|_{collapsing}
  \lesssim \big\||\nabla_{x+y}|^{\frac{1}{2}}\Lambda\big\|_{collapsing}
\end{align*}
and also, using Cauchy-Schwartz in $\tau$,
 \begin{align*}
 &\|u_1\|_{L^{\infty}(dt)L^2(dxdy)} \lesssim
 \big\| \int_{\tau \in [-N^2, N^2]} \frac{1}{|\tau|^{\frac{1}{4}}}|\tau|^{\frac{1}{4}} |\widetilde{u_1}| d \tau \big\|_{L^2(d\xi d \eta)}\\
& \lesssim \sqrt N \|||\tau|^{\frac{1}{4}}|\widetilde{u_1}\|_{L^2(\tau \in [-N^2, N^2] d(\xi d\eta))}
  \lesssim
  \|\big|\partial_{t}\big|^{\frac{1}{4}}\Lambda\|_{collapsing}.
 \end{align*}
 Next, we consider $P_{|\xi-\eta|<N}u = P_{|\xi-\eta|<1}u + \sum P_{|\xi-\eta|\sim 2^i}u$ where the sum has about $\log N$ terms.
 Proposition \ref{Suufixed} applies to each term, and the result follows.

\end{proof}

   \section{Proof of Theorem \ref{mainthm}}\label{endlin}
Recall the ``projections"  $P_{|\xi|>M} $ and $P_{|\xi|<M} $
 by
\begin{align}
\F \left(P_{|\xi|<M} \Lambda\right)(\xi, \eta)=\hat \phi \left(\frac{|\xi|}{M}\right)\F \Lambda(\xi, \eta)\label{projdef}
\end{align}
and $P_{|\xi|>M}=1-P_{|\xi|<M}$.
The function $\hat \phi$ is a $C_0^{\infty}$ function (which can change from line to line).

 Here $\F$ denotes the Fourier transform.
These multipliers are bounded on $L^q(dx)L^2(dy)$, $L^q(dy)L^2(dx)$ and $L^q(d(x-y))L^2(d(x+y))$ for $1 \le q \le \infty$
 (uniformly in $M$).
Also, we adopt the convention that $P$ and $\phi$ may change from line to line. 

   We have to show:
   If $h(t)$ is the Heaviside function and
     \begin{align*}
&\S \Lambda(t, x, y) = h(t) \bigg(N^{3 \beta -1} v(N^{\beta}(x-y)) \Lambda(t, x, y) +G(t, x, y)\notag\\
&+
 N^{3 \beta -1} v(N^{\beta}(x-y)) H(t, x, y)\bigg)\\
&\Lambda(0, \cdot)=\Lambda_0\notag
\end{align*}
then
\begin{align*}
&\big\|\big<\nabla_x\big>^{\alpha}\big<\nabla_y\big>^{\alpha}\Lambda \big\|_{\mathcal{S}_{x, y}} \\
&+\big\|\big|\nabla_{x+y}\big|^{\alpha}\Lambda\big\|_{L^2(dt)L^{\infty}(d(x-y))L^2(d(x+y))}   \\
&+\big\|\big|\partial_t\big|^{\frac{1}{4}}\Lambda\big\|_{L^2(dt)L^{\infty}(d(x-y))L^2(d(x+y))}\\
&\lesssim
\big\|\big<\nabla_x\big>^{\alpha}\big<\nabla_y\big>^{\alpha}G \big\|_{\mathcal{S}_r'}
+N^{-\epsilon}\big\|\big<\nabla_x\big>^{\alpha}\big<\nabla_y\big>^{\alpha}H \big\|_{L^2(dt)L^{6}(d(x-y))L^2(d(x+y))}\\
&+N^{-\epsilon}\big\|\big|\partial_t\big|^{\frac{1}{4}}H\big\|_{L^{\infty}(d(x-y))L^2(dt) L^2(d(x+y))}\\
&+N^{-\epsilon}\big\|\big|\nabla_{x+y}\big|^{\alpha}H\big\|_{L^{\infty}(d(x-y))L^2(dt) L^2(d(x+y))}\\
&+ \big\|\big<\nabla_x\big>^{\alpha}\big<\nabla_y\big>^{\alpha}\Lambda_0 \big\|_{L^2}
\end{align*}
and $N^{3 \beta -1} v(N^{\beta}(x-y)$ is denoted by $v_M(x-y)$,
 and $M$ is of the form $N^{1+\epsilon}$ for some small $\epsilon>0$.

Before starting the proof we remark that  we have the following Strichartz estimate (see Theorem \ref{mainStrich})
\begin{align*}
&\|\Lambda\|_{L^2(dt) L^6(d(x-y)) L^2(d(x+y))} \lesssim \|v_M(x-y)\Lambda\|_{L^2(dt) L^{\frac{6}{5}}(d(x-y)) L^2(d(x+y))}\\
&+\|v_M(x-y)H\|_{L^2(dt) L^{\frac{6}{5}}(d(x-y)) L^2(d(x+y))}
+\|G\|_{\mathcal{S}_r'} + \|\Lambda_0\|_{L^2}.
\end{align*}
Using H\"older's inequality and the fact that $\|v_M\|_{L^{\frac{3}{2}}} = O(M^{-\epsilon_0})$ as $M \to \infty$ (see \eqref{vestim}) we can treat the potential term as a perturbation and get
\begin{align}
&\|\Lambda\|_{\mathcal{S}} \lesssim \|\Lambda_0\|_{L^2}
+\|v_M(x-y)H\|_{L^2(dt) L^{\frac{6}{5}}(d(x-y)) L^2(d(x+y))}
+\|G\|_{\mathcal{S}_r'} \label{noderiv}\\
&\lesssim \|\Lambda_0\|_{L^2}
+M^{-\epsilon_0}\|H\|_{L^2(dt) L^{6}(d(x-y)) L^2(d(x+y))}
+\|G\|_{\mathcal{S}_r'} \notag.
\end{align}
We can do the same after taking $|\nabla_{x}|^{\alpha}$, but we need a suitable Leibniz rule. Using the outline of  \cite{ChWein},  \cite{Muscalu1}.  \cite{Ward}, it is possible to prove
\begin{theorem} \label{leibniz}
Let $\alpha, \beta \ge 0$.  Let $f(x, y)=v(x-y)$ with $v \in \mathcal S$.  Then

\begin{align*}
&\||\nabla_{x}|^{\alpha}|\nabla_{y}|^{\beta} \left(f(x, y) g(x, y)\right)\|_{L^{r}(d(x-y))L^{2}(d(x+y))}\\
&\lesssim  \|<\nabla_{x}>^{\alpha+\beta}v\|_{L^{p_1}}\|g\|_{L^{q_1}(d(x-y))L^{2}(d(x+y))}\\
&+  \|<\nabla_{x}>^{\alpha}v\|_{L^{p_2}}\|<\nabla_{y}>^{\beta}g\|_{L^{q_2}(d(x-y))L^{2}(d(x+y))}\\
&+  \|<\nabla_{y}>^{\beta}v\|_{L^{p_3}}\|<\nabla_{x}>^{\alpha}g\|_{L^{q_3}(d(x-y))L^{2}(d(x+y))}\\
&+ \|v\|_{L^{p_4}}\|<\nabla_{x}>^{\alpha}<\nabla_{y}>^{\beta}g\|_{L^{q_4}(d(x-y))L^{2}(d(x+y))}\\
\end{align*}
($\frac{1}{r}=\frac{1}{p_i} + \frac{1}{q_i}$, $1<r,  p_i, q_i < \infty$). In addition, if $\hat v$ is supported in $|\xi| \le \frac{ M}{10}$ and $\hat g (\xi, \eta)$ is supported in $|\xi|>10M$, then the $|\nabla_{x}|^{\alpha}$ derivatives only fall on $g$ and we have
\begin{align*}
&\||\nabla_{x}|^{\alpha}|\nabla_{y}|^{\beta} \left(f(x, y) g(x, y)\right)\|_{L^{r}(d(x-y))L^{2}(d(x+y))}\\
&\lesssim
\|<\nabla_{y}>^{\beta}v\|_{L^{p_3}}\|<\nabla_{x}>^{\alpha}g\|_{L^{q_3}(d(x-y))L^{2}(d(x+y))}\\
&+ \|v\|_{L^{p_4}}\|<\nabla_{x}>^{\alpha}<\nabla_{y}>^{\beta}g\|_{L^{q_4}(d(x-y))L^{2}(d(x+y))}.\\
\end{align*}
\end{theorem}
However, for our purposes it is enough to use the following version, which is easier to prove. We thank Xiaoqi Huang for suggesting this approach.

\begin{theorem} \label{leibnizX}
Let $\alpha, \beta \ge 0$.  Let  $v_M $ as above with $\hat v_M$ is supported in $|\xi| \le \frac{ M}{10}$, and let
$\frac{1}{r}=\frac{1}{p} + \frac{1}{q}$, $1<r,  p, q < \infty$.
 Then

\begin{align}
&\||\nabla_{x}|^{\alpha} \left(v_M(x-y)  g(x, y)\right)\|_{L^{r}(d(x-y))L^{2}(d(x+y))}\label{X1}\\
&\lesssim  M^{\alpha }\| v_M\|_{L^{r}}\|g\|_{L^{\infty}(d(x-y))L^{2}(d(x+y))}\label{X2}\\
&+\|v_M\|_{L^{p_1}}\|<\nabla_{x}>^{\alpha}g\|_{L^{q_1}(d(x-y))L^{2}(d(x+y))} \label{X3}
\end{align}
and similarly for $|\nabla_{y}|^{\alpha}$.
Also, if $\frac{1}{r}=\frac{1}{p_i} + \frac{1}{q_i}$, $1<r,  p_i, q_i < \infty$,

\begin{align}
&\||\nabla_{x}|^{\alpha} \nabla_{y}|^{\beta}\left(v_M(x-y)) P_{|\eta|>M} g(x, y)\right)\|_{L^{r}(d(x-y))L^{2}(d(x+y))}\label{XX1}\\
&\lesssim M^{\alpha }  \|v_M\|_{L^{p_1}}\| \nabla_{y}|^{\beta}g\|_{L^{q_1}(d(x-y))L^{2}(d(x+y))} \label{XX2}\\
&+\|v\|_{L^{p_2}}\|<\nabla_{x}>^{\alpha}<\nabla_{y}>^{\beta} g\|_{L^{q_2}(d(x-y))L^{2}(d(x+y))}\label{XX3}
\end{align}
and also
\begin{align}
&\||\nabla_{x}|^{\alpha} \nabla_{y}|^{\beta}\left(v_M(x-y)) P_{|\eta|<M} P_{|\eta|<M} g(x, y)\right)\|_{L^{r}(d(x-y))L^{2}(d(x+y))} \notag\\
&\lesssim M^{\alpha+\beta }  \|v_M\|_{L^{p_3}}\|g\|_{L^{q_3}(d(x-y))L^{2}(d(x+y))}. \label{XX4}
\end{align}

\end{theorem}
\begin{proof}
We use  standard Littlewood-Paley operators $P_{ |\xi|\sim 2^i M}$, $P_{ |\xi|\lesssim 2^i M}$. In the name of simplicity of notation, we allow the implicit constants in $\sim$, $\lesssim$ to be different in different instances. However, the modified projections will be denoted by $\bar P$.
Thus, for instance, $\bar P_{ |\xi|\sim 2^i M}P_{ |\xi|\sim 2^i M}
=P_{ |\xi|\sim 2^i M}
$. Also, we use the notation $|\nabla|^{\alpha}P_{ |\xi|\sim 2^i M}g= (2^{i}M)^{\alpha} \bar P_{ |\xi|\sim 2^i M}g$. In the first instance, the multiplier is of the form $\psi(\frac{\xi}{M})$ with $\psi \in C_0^{\infty} (\mathbb R^n)$, $\psi(\xi)=0$ in a neighborhood of $0$. In the second case, $\psi $ is replaced by $|\xi|^{\alpha} \psi(\xi)$ which has the same properties. When $\bar P$ is further modified, it is still denoted $\bar P$.
This is the convention used in \cite{Muscalu1}.
Finally, denote $|\nabla|^{\alpha}P_{ |\xi| \lesssim  M}g= M^{\alpha}\bar P_{ |\xi|\lesssim M}g$, where the exact definition of $\bar P$ is seen on the Fourier transform side. The corresponding multiplier is not smooth near $0$, but the corresponding kernel is in $L^1$.
 After the modifications described above, the identity $g= \bar P_{ |\xi|\lesssim M}g + \sum_{i=1}^{\infty}  \bar P_{ |\xi|\sim 2^i M}g$ is no longer true. Square function estimates have to be used instead, and
 the square function operators constructed using $\bar P_{ |\xi|\sim 2^i M}$ have the same mapping properties as those using
$ P_{ |\xi|\sim 2^i M}$.

For \eqref{X1}, decompose $g= P_{ |\xi|<M} g +  P_{ |\xi|>M} g$. Arguing as in the proof of Bernstein's inequality,
\begin{align*}
&|\nabla_{x}|^{\alpha} \left(v_M(x-y) P_{ |\xi|<M} g(x, y)\right)=
|\nabla_{x}|^{\alpha}\bar P_{ |\xi|\lesssim M} \left(v_M(x-y) P_{ |\xi|<M} g(x, y)\right)\\
&=M^{\alpha}\bar P_{ |\xi|\lesssim M} \left(v_M(x-y) P_{ |\xi|<M} g(x, y)\right).
\end{align*}
Since $\bar P_{ |\xi|\lesssim M}$ is given by convolution with a kernel which is in $L^1(dx)$ uniformly in $M$,
\begin{align*}
&\|\bar P_{ |\xi|\lesssim M} \left(v_M(x-y) P_{ |\xi|<M} g(x, y)\right)\|_{L^{r}(d(x-y))L^{2}(d(x+y))}\\
&\lesssim \|v_M(x-y) P_{ |\xi|<M} g(x, y)\|_{L^{r}(d(x-y))L^{2}(d(x+y))}\\
&\lesssim \| v_M\|_{L^{r}}\|g\|_{L^{\infty}(d(x-y))L^{2}(d(x+y))}
\end{align*}
For $P_{ |\xi|>M} g$, decompose it as $P_{ |\xi|>M} g = \sum_{i=1}^{\infty}P_{ |\xi|\sim 2^i M}g$.
Then $|\nabla|^{\alpha}\left(v_M(x-y)P_{ |\xi|\sim 2^i M}g\right)= (2^{i}M)^{\alpha} \bar P_{ |\xi|\sim 2^i M}\left(v_M(x-y) P_{ |\xi|\sim 2^i M}g\right)$.

Proving the estimate by duality involves using a test function $H$ with $\|H \|_{ L^{r'}(dx) L^{2}(dy)}=1$ and looking at
\begin{align}
&\int |\nabla_{x}|^{\alpha} \big(v_M(x-y)P_{ |\xi|>M} g(x, y)\big)h(x, y) dx dy \label{intt1}
\end{align}
with $h=H\circ R^{-1}$, or $\|h\|_{L^{r'}(d(x-y))L^{2}(d(x+y))} =1$
We have
\begin{align*}
&|\eqref{intt1}| =\bigg| \int v_M(x-y) \sum (2^{i}M)^{\alpha}  P_{ |\xi|\sim 2^i M}g \bar P_{ |\xi|\sim 2^i M}h\bigg|\\
&\le  \int |v_M(x-y)| \left(\sum |(2^{i}M)^{\alpha} P_{ |\xi|\sim 2^i M}g|^2 \right)^{\frac{1}{2}} \left(\sum |\bar P_{ |\xi|\sim 2^i M}h|^2\right)^{\frac{1}{2}}\\
&\le \|v_M\|_{L^{p_1}} \bigg\| \left(\sum |(2^{i}M)^{\alpha} P_{ |\xi|\sim 2^i M}g|^2 \right)^{\frac{1}{2}} \bigg\|_{L^{q_1}(d(x-y))L^{2}(d(x+y))}\\
&\times \bigg\|
\left(\sum | \bar P_{ |\xi|\sim 2^i M}h|^2\bigg|\right)^{\frac{1}{2}}\bigg\|_{ L^{r'}(d(x-y)) L^{2}(d(x+y))}\\
&\lesssim \|v_M\|_{L^{p_1}}\bigg\|\left(\sum ||\nabla_x|^{\alpha} \bar P_{ |\xi|\sim 2^i M}g|^2 \right)^{\frac{1}{2}}\bigg\|_{L^{q_1}(d(x-y))L^{2}(d(x+y))} \\
&\times \bigg\|
\left(\sum |  \bar P_{ |\xi|\sim 2^i M}h|^2\bigg|\right)^{\frac{1}{2}}\bigg\|_{ L^{r'}(d(x-y)) L^{2}(d(x+y))}\\
&\lesssim \|v_M\|_{L^{p_1}} \|<\nabla_{x}>^{\alpha}g\|_{L^{q_1}(d(x-y))L^{2}(d(x+y))}\|h\|_{L^{r'}(d(x-y))L^{2}(d(x+y))}.
\end{align*}
This uses the  square function estimate in rotated coordinates.

For \eqref{XX1},
let
$\frac{1}{r}=\frac{1}{p_1} + \frac{1}{q_1}$, $1<r,  p_1, q_1 < \infty$. Then, use the same duality argument as in the previous proof.
with $\|h\|_{L^{r'}(d(x-y))L^{2}(d(x+y))} =1$. The bound \eqref{XX2} corresponds to
\begin{align*}
&\||\nabla_{x}|^{\alpha} | \nabla_{y}|^{\beta}\left(v_M(x-y))P_{|\xi|<M} P_{|\eta|>M} g(x, y)\right)\|_{L^{r}(d(x-y))L^{2}(d(x+y))}\\
&=\bigg\|\sum_{i=1}^{\infty}
\bar P_{|\eta|\sim 2^iM} \bar P_{|\xi|\lesssim M} \bigg(M^{\alpha}v_M(x-y))P_{|\xi|<M} ( 2^iM)^{\beta} P_{|\eta|\sim 2^iM} g(x, y)\bigg)\bigg\|_{L^{r}(d(x-y))L^{2}(d(x+y))}\\
&=\bigg|\int \sum_{i=1}^{\infty}
\bar P_{|\eta|\sim 2^iM} \bar P_{|\xi|\lesssim M} \bigg(M^{\alpha}v_M(x-y))P_{|\xi|<M} ( 2^iM)^{\beta} P_{|\eta|\sim 2^iM} g(x, y)\bigg) h\bigg|\\
&=\bigg|\int \sum_{i=1}^{\infty}
 \bigg(M^{\alpha}v_M(x-y))P_{|\xi|<M} ( 2^iM)^{\beta} P_{|\eta|\sim 2^iM} g(x, y)\bigg)\bar P_{|\eta|\sim 2^iM} \bar P_{|\xi|\lesssim M} h\bigg|\\
 &\le \int |M^{\alpha}v_M(x-y))| \left(\sum_{i=1}^{\infty}
 |P_{|\xi|<M} ( 2^iM)^{\beta} P_{|\eta|\sim 2^iM} g|^2\right)^{\frac{1}{2}}\\
 &\times \left(\sum_{i=1}^{\infty}|\bar P_{|\eta|\sim 2^iM} \bar P_{|\xi|\lesssim M} h|^2\right)^{\frac{1}{2}}\\
 &\le \|M^{\alpha}v_M\|_{L^{p_1}}\bigg\|\left(\sum_{i=1}^{\infty}
 |P_{|\xi|<M} |\nabla_y|^{\beta} \bar P_{|\eta|\sim 2^iM} g|^2\right)^{\frac{1}{2}}\bigg\|_{L^{q_1}(d(x-y))L^{2}(d(x+y))}\\
 &\bigg\|\left(\sum_{i=1}^{\infty}|\bar P_{|\eta|\sim 2^iM} \bar P_{|\xi|\lesssim M} h|^2\right)^{\frac{1}{2}}\bigg\|_{L^{r'}(d(x-y))L^{2}(d(x+y))}
\end{align*}
and the last factor is $\lesssim 1$.

The bound \eqref{XX3} corresponds to
\begin{align*}
&\||\nabla_{x}|^{\alpha} | \nabla_{y}|^{\beta}\left(v_M(x-y))P_{|\xi|>M} P_{|\eta|>M} g(x, y)\right)\|_{L^{r}(d(x-y))L^{2}(d(x+y))}\\
&\lesssim  \|v_M\|_{L^{p_2}}\| |\nabla_{x}|^{\alpha}|\nabla_{y}|^{\beta}g\|_{L^{q_2}(d(x-y))L^{2}(d(x+y))}.
\end{align*}
To prove this, do a double Littlewood-Paley decomposition
$P_{|\xi|>M} P_{|\eta|>M} g=\sum_{i=1}^{\infty} \sum_{j=1}^{\infty} P_{|\eta|\sim 2^iM} P_{|\eta|\sim 2^jM}g$ and proceed as before, using
``double square function estimates in rotated coordinates", Lemma \ref{squarerotated}.
Finally, the bound \eqref{XX4} follows from Bernstein's inequality.

\end{proof}
Continuing with the comments preceding  the proof of Theorem \ref{mainthm},
\begin{align*}
&\||\nabla_{x}|^{\alpha}\Lambda\|_{L^2(dt) L^6(d(x-y)) L^2(d(x+y))}\\
&\lesssim M^{\alpha}\|v_M\|_{L^{\frac{6}{5}}}\|\Lambda\|_{L^2(dt) L^{\infty}(d(x-y)) L^2(d(x+y))}\\
&+
 M^{\alpha}\|v_M\|_{L^{\frac{6}{5}}}\|H\|_{L^2(dt) L^{\infty}(d(x-y)) L^2(d(x+y))}
\\
&+\|v_M\|_{L^{\frac{3}{2}}}\|\big<\nabla_x\big>^{\alpha}\Lambda\|_{L^2(dt) L^{6}(d(x-y)) L^2(d(x+y))}\\
&+\|v_M\|_{L^{\frac{3}{2}}}\|\big<\nabla_x\big>^{\alpha}H\|_{L^2(dt) L^{6}(d(x-y)) L^2(d(x+y))}
\\
&+\|\big<\nabla_x\big>^{\alpha}G\|_{\mathcal{S}_r'} + \|\big<\nabla_x\big>^{\alpha}\Lambda_0\|_{L^2}\notag.
\end{align*}
Using the  ``Sobolev at an angle" estimate (Lemma \ref{sobangle})
\begin{align*}
\|\Lambda\|_{L^2(dt)L^{\infty}(d(x-y))L^2(d(x+y))} \lesssim \|\big<\nabla_{x}\big>^{\alpha}\Lambda\|_{L^2(dt)L^{6}(d(x-y))L^2(d(x+y))}
\end{align*}
and similarly for $H$.
Using $\|v_M\|_{L^{\frac{3}{2}}}+  M^{\alpha} \|v_M\|_{L^{\frac{6}{5}}}=O(M^{-\epsilon_0})$ as $M \to \infty$
we can treat the two terms involving the potential and $\Lambda$ as perturbations and get
\begin{align}
&\|\big<\nabla_{x}\big>^{\alpha}\Lambda\|_{\mathcal{S}}\lesssim \|\big<\nabla_{x}\big>^{\alpha}\Lambda_0\|_{L^2}\label{onederiv}\\
&
+M^{-\epsilon_0}\big\|\big<\nabla_x\big>^{\alpha}H \big\|_{L^2(dt)L^{6}(d(x-y))L^2(d(x+y))}
+\|\big<\nabla_{x}\big>^{\alpha}G\|_{\mathcal{S}_r'} \notag.
\end{align}
Finally, we can repeat the argument with $|\nabla_{x}|^{\alpha}|\nabla_{y}|^{\alpha}$. Now we are forced to estimate
$M^{2\alpha} \|v_M\|_{L^{\frac{6}{5}}}=O(M)$ and get a sub-optimal estimate
\begin{align}
&\frac{1}{M}\|\big<\nabla_{x}\big>^{\alpha}\big<\nabla_{y}\big>^{\alpha}\Lambda\|_{\mathcal{S}} \lesssim \frac{1}{M}\|\big<\nabla_{x}\big>^{\alpha}\big<\nabla_{y}\big>^{\alpha}\Lambda_0\|_{L^2}\\
&+M^{-\epsilon_0}\big\|\big<\nabla_x\big>^{\alpha}\big<\nabla_y\big>^{\alpha}H \big\|_{L^2(dt)L^{6}(d(x-y))L^2(d(x+y))}
+\|\big<\nabla_{x}\big>^{\alpha}\big<\nabla_{y}\big>^{\alpha}G\|_{\mathcal{S}_r'} \label{badtwoderiv}.
\end{align}
This will help control lower order terms.

To continue, we need a frequency decomposition.
Let $\phi(x)$ such that $\hat \phi \in C_0^{\infty}$ and $\hat \phi(\xi)=1$ in $|\xi|<1$, $\hat \phi(\xi)=0$ in $|\xi|>2$.

Theorem \ref{mainthm} follows from the next two more detailed theorems.
\begin{theorem}\label{mainthmnew}
Let $\Lambda$ satisfy \eqref{maineqold}, and
let $1+$ denote $1+\delta_0$ with $\delta_0 >0$ satisfying \eqref{vestim}, \eqref{vestim1}.
Then, at high frequencies,
\begin{align}
&\big\|\big<\nabla_x\big>^{\alpha}\big<\nabla_y\big>^{\alpha}P_{|\xi|>M^{1+}}\Lambda \big\|_{\mathcal{S}}
+\big\|\big<\nabla_x\big>^{\alpha}\big<\nabla_y\big>^{\alpha}P_{|\eta|>M^{1+}}\Lambda \big\|_{\mathcal{S}}\notag\\
&+
\big\|\big|\nabla_{x+y}\big|^{\alpha}P_{|\xi|>M^{1+}}\Lambda\big\|_{L^2(dt)L^{\infty}(d(x-y))L^2(d(x+y))}\notag \\
&+\big\|\big|\nabla_{x+y}\big|^{\alpha}P_{|\eta|>M^{1+}}\Lambda\big\|_{L^2(dt)L^{\infty}(d(x-y))L^2(d(x+y))}\notag \\
&+
\big\|\big|\partial_t\big|^{\frac{1}{4}}P_{|\xi|>M^{1+}}\Lambda\big\|_{L^2(dt)L^{\infty}(d(x-y))L^2(d(x+y))}\notag \\
&+\big\|\big|\partial_t\big|^{\frac{1}{4}}P_{|\eta|>M^{1+}}\Lambda\big\|_{L^2(dt)L^{\infty}(d(x-y))L^2(d(x+y))}\notag \\
&\lesssim
\big\|\big<\nabla_x\big>^{\alpha}\big<\nabla_y\big>^{\alpha}G \big\|_{\mathcal{S}_r'} \label{highest}
+M^{-\epsilon_0}\big\|\big<\nabla_x\big>^{\alpha}\big<\nabla_y\big>^{\alpha}H \big\|_{L^2(dt)L^{6}(d(x-y))L^2(d(x+y))}\notag\\
&+\big\|\big<\nabla_x\big>^{\alpha}\big<\nabla_y\big>^{\alpha}\Lambda_0 \big\|_{L^2}.
\end{align}
In addition, the proof will show

\begin{align*}
&\|\big<\nabla_x\big>^{\alpha}\big<\nabla_y\big>^{\alpha}\bigg(v_M(x-y)\big(P_{|\xi|>M^{1+} } \Lambda\big) \bigg)
\|_{L^2(dt) L^{\frac{6}{5}}(d(x-y)) L^2(d(x+y))}\\
&\lesssim
\big\|\big<\nabla_x\big>^{\alpha}\big<\nabla_y\big>^{\alpha}G \big\|_{\mathcal{S}_r'}  \label{fullequation}
+M^{-\epsilon_0}\big\|\big<\nabla_x\big>^{\alpha}\big<\nabla_y\big>^{\alpha}H \big\|_{L^2(dt)L^{6}(d(x-y))L^2(d(x+y))}\\
&+ \big\|\big<\nabla_x\big>^{\alpha}\big<\nabla_y\big>^{\alpha}\Lambda_0 \big\|_{L^2}.
\end{align*}

\end{theorem}

\begin{proof}

Roughly speaking, $\big\|\big<\nabla_x\big>^{\alpha}\big<\nabla_y\big>^{\alpha}P_{|\xi|>M^{1+}}\big(v^1_M(x-y)\Lambda \big)\big\|_{L^2(dt) L^{\frac{6}{5}}(d(x-y)) L^2(d(x+y))}
$ can be treated as a perturbation because $\big<\nabla_x\big>^{\alpha}$ only falls on $\Lambda$. Rigorously,
we have
\begin{align}
&\S P_{|\xi|>M^{1+}}\Lambda = P_{|\xi|>M^{1+}} \bigg(v_M^1(x-y)\big(P_{|\xi|>M^{1+}-M } \Lambda\big) \bigg) \\
&+P_{|\xi|>M^{1+}} \bigg(v_M^2(x-y) \Lambda \bigg)
+P_{|\xi|>M^{1+}} G +P_{|\xi|>M^{1+}} (v_M H).\label{triv}
\end{align}
We used the fact that $\hat v^1_M$ is supported in $|\xi|< \frac{M}{10}$.
Next we use  the Strichartz  estimate of Theorem \ref{mainStrich} and the collapsing estimate of Lemma \ref{4collapsing}
and Lemma \ref{timederlemma}.

\begin{align*}
&\big\|\big<\nabla_x\big>^{\alpha}\big<\nabla_y\big>^{\alpha}P_{|\xi|>M^{1+}}\Lambda \big\|_{L^2(dt) L^6(d(x-y)) L^2(d(x+y))}\\
&+
\big\|\big|\nabla_{x+y}\big|^{\alpha}P_{|\xi|>M^{1+}}\Lambda\big\|_{L^2(dt)L^{\infty}(d(x-y))L^2(d(x+y))}\\
&+\big\|\big|\partial_t\big|^{\frac{1}{4}}P_{|\xi|>M^{1+}}\Lambda\big\|_{L^2(dt)L^{\infty}(d(x-y))L^2(d(x+y))}\\
\notag\\
&\lesssim \|\big<\nabla_x\big>^{\alpha}\big<\nabla_y\big>^{\alpha}\bigg(v_M^1(x-y)\big(P_{|\xi|>M^{1+}-M } \Lambda\big) \bigg)
\|_{L^2(dt) L^{\frac{6}{5}}(d(x-y)) L^2(d(x+y))}\\
&+\|\big<\nabla_x\big>^{\alpha}\big<\nabla_y\big>^{\alpha}\bigg(v_M^2(x-y) \Lambda \bigg)
\|_{L^2(dt) L^{\frac{6}{5}}(d(x-y)) L^2(d(x+y))}\\
&+\|\big<\nabla_x\big>^{\alpha}\big<\nabla_y\big>^{\alpha}\bigg(v_M^1(x-y)\big(P_{|\xi|>M^{1+}-M } H\big) \bigg)
\|_{L^2(dt) L^{\frac{6}{5}}(d(x-y)) L^2(d(x+y))}\\
&+\|\big<\nabla_x\big>^{\alpha}\big<\nabla_y\big>^{\alpha}\bigg(v_M^2(x-y) H \bigg)
\|_{L^2(dt) L^{\frac{6}{5}}(d(x-y)) L^2(d(x+y))}\\
&+\big\|\big<\nabla_x\big>^{\alpha}\big<\nabla_y\big>^{\alpha}G \big\|_{\mathcal{S}_r'}
+ \big\|\big<\nabla_x\big>^{\alpha}\big<\nabla_y\big>^{\alpha}\Lambda_0 \big\|_{L^2}.
\end{align*}

For the  terms involving $v_M^1$ we use Theorem \ref{leibnizX} and the fact $M^{1+}-M > 10M$ to conclude that the $\big<\nabla_x\big>^{\alpha}$ derivative only falls on
$P_{|\xi|>M^{1+}-M } \Lambda$ and $P_{|\xi|>M^{1+}-M } H$. Then H\"older's inequality,  ``Sobolev at an angle" (see Lemma  \ref{sobangle})
and our estimates on $v_M$ (see \eqref{vestim})
show
\begin{align}
&\|\big<\nabla_x\big>^{\alpha}\big<\nabla_y\big>^{\alpha}\bigg(v_M^1(x-y)\big(P_{|\xi|>M^{1+}-M } \Lambda\big) \bigg)
\|_{L^2(dt) L^{\frac{6}{5}}(d(x-y)) L^2(d(x+y))}\\
&\lesssim \|v_M\|_{L^{\frac{3}{2}}}
\|\big<\nabla_x\big>^{\alpha}\big<\nabla_y\big>^{\alpha}\big(P_{|\xi|>M^{1+}-M } \Lambda\big) \bigg)
\|_{L^2(dt) L^{6}(d(x-y)) L^2(d(x+y))}\\
&+ M^{\alpha}\|v_M\|_{L^{\frac{6}{5}+}}
\|\big<\nabla_x\big>^{\alpha}\big<\nabla_y\big>^{\alpha}\big(P_{|\xi|>M^{1+}-M } \Lambda\big) \bigg)\label{trouble}
\|_{L^2(dt) L^{6}(d(x-y)) L^2(d(x+y))}\\
&\lesssim M^{-\epsilon_0} \|\big<\nabla_x\big>^{\alpha}\big<\nabla_y\big>^{\alpha}P_{|\xi|>M^{1+}-M } \Lambda\|_{L^2(dt) L^{6}(d(x-y)) L^2(d(x+y))}.
\end{align}

Similarly,
\begin{align*}
&\|\big<\nabla_x\big>^{\alpha}\big<\nabla_y\big>^{\alpha}\bigg(v_M^1(x-y)\big(P_{|\xi|>M^{1+}-M } H\big) \bigg)
\|_{L^2(dt) L^{\frac{6}{5}}(d(x-y)) L^2(d(x+y))}\\
&\lesssim M^{-\epsilon_0} \|\big<\nabla_x\big>^{\alpha}\big<\nabla_y\big>^{\alpha}H\|_{L^2(dt) L^{6}(d(x-y)) L^2(d(x+y))}
\end{align*}
while for the term involving $v_M^2$ we have
\begin{align*}
&\|\big<\nabla_x\big>^{\alpha}\big<\nabla_y\big>^{\alpha}\bigg(v_M^2(x-y) \Lambda \bigg)
\|_{L^2(dt) L^{\frac{6}{5}}(d(x-y)) L^2(d(x+y))}\\
&\lesssim M^{-10} \|\big<\nabla_x\big>^{\alpha}\big<\nabla_y\big>^{\alpha} \Lambda\|_{L^2(dt) L^{6}(d(x-y)) L^2(d(x+y))}
\end{align*}
(and similarly for $H$). Thus
\begin{align*}
 &\big\|\big<\nabla_x\big>^{\alpha}\big<\nabla_y\big>^{\alpha}P_{|\xi|>M^{1+}}\Lambda \big\|_{\mathcal{S}}  \\
&+
\big\|\big|\nabla_{x+y}\big|^{\alpha}P_{|\xi|>M^{1+}}\Lambda\big\|_{L^2(dt)L^{\infty}(d(x-y))L^2(d(x+y))}
\notag\\
&+\big\|\big|\partial_t\big|^{\frac{1}{4}}P_{|\xi|>M^{1+}}\Lambda\big\|_{L^2(dt)L^{\infty}(d(x-y))L^2(d(x+y))}\\
&\lesssim
 M^{-\epsilon_0} \|\big<\nabla_x\big>^{\alpha}\big<\nabla_y\big>^{\alpha} P_{|\xi|>M^{1+}-M}\Lambda\|_{\mathcal S}\\
 &+M^{-10} \|\big<\nabla_x\big>^{\alpha}\big<\nabla_y\big>^{\alpha} \Lambda\|_{L^2(dt) L^{6}(d(x-y)) L^2(d(x+y))}\\
&+
\big\|\big<\nabla_x\big>^{\alpha}\big<\nabla_y\big>^{\alpha}G \big\|_{\mathcal{S}_r'}\\
&+M^{-\epsilon_0}\big\|\big<\nabla_x\big>^{\alpha}\big<\nabla_y\big>^{\alpha}H \big\|_{L^2(dt)L^{6}(d(x-y))L^2(d(x+y))}\\
&+ \big\|\big<\nabla_x\big>^{\alpha}\big<\nabla_y\big>^{\alpha}\Lambda_0 \big\|_{L^2}.
 \end{align*}
The last four terms on the RHS are acceptable, but the first one must be estimated further,
  by repeating the argument(with the same implicit constants in $\lesssim$) as long as $M^{1+}-kM > 10M$, which is essentially $\log\left(\frac{M^{1+}}{M}\right)$ times. At the $k$th step we get
 \begin{align*}
 &\big\|\big<\nabla_x\big>^{\alpha}\big<\nabla_y\big>^{\alpha}P_{|\xi|>M^{1+}- (k-1)M}\Lambda \big\|_{\mathcal{S}}\\
&\lesssim
 M^{-\epsilon_0} \|\big<\nabla_x\big>^{\alpha}\big<\nabla_y\big>^{\alpha} P_{|\xi|>M^{1+}-kM}\Lambda\|_{\mathcal S}\\
&+M^{-10} \|\big<\nabla_x\big>^{\alpha}\big<\nabla_y\big>^{\alpha} \Lambda\|_{L^2(dt) L^{6}(d(x-y)) L^2(d(x+y))}\\
&+
\big\|\big<\nabla_x\big>^{\alpha}\big<\nabla_y\big>^{\alpha}G \big\|_{\mathcal{S}_r'}
+M^{-\epsilon_0}\big\|\big<\nabla_x\big>^{\alpha}\big<\nabla_y\big>^{\alpha}H \big\|_{L^2(dt)L^{6}(d(x-y))L^2(d(x+y))}\\
&+ \big\|\big<\nabla_x\big>^{\alpha}\big<\nabla_y\big>^{\alpha}\Lambda_0 \big\|_{L^2}.
 \end{align*}
Putting together the above $k$ estimates we get
 \begin{align*}
 &\big\|\big<\nabla_x\big>^{\alpha}\big<\nabla_y\big>^{\alpha}P_{|\xi|>M^{1+}}\Lambda \big\|_{\mathcal{S}}
+\big\|\big<\nabla_x\big>^{\alpha}\big<\nabla_y\big>^{\alpha}P_{|\eta|>M^{1+}}\Lambda \big\|_{\mathcal{S}}\notag\\
&+
\big\|\big|\nabla_{x+y}\big|^{\alpha}P_{|\xi|>M^{1+}}\Lambda\big\|_{L^2(dt)L^{\infty}(d(x-y))L^2(d(x+y))}\\
&+\big\|\big|\nabla_{x+y}\big|^{\alpha}P_{|\eta|>M^{1+}}\Lambda\big\|_{L^2(dt)L^{\infty}(d(x-y))L^2(d(x+y))}\notag\\
&+
\big\|\big|\partial_t\big|^{\frac{1}{4}}P_{|\xi|>M^{1+}}\Lambda\big\|_{L^2(dt)L^{\infty}(d(x-y))L^2(d(x+y))}\\
&+\big\|\big|\partial_t\big|^{\frac{1}{4}}P_{|\eta|>M^{1+}}\Lambda\big\|_{L^2(dt)L^{\infty}(d(x-y))L^2(d(x+y))}\notag \\
&\lesssim
 M^{-k\epsilon_0} \|\big<\nabla_x\big>^{\alpha}\big<\nabla_y\big>^{\alpha} P_{|\xi|>M^{1+}-kM}\Lambda\|_{\mathcal S}\\
&+M^{-10} \|\big<\nabla_x\big>^{\alpha}\big<\nabla_y\big>^{\alpha} \Lambda\|_{L^2(dt) L^{6}(d(x-y)) L^2(d(x+y))}\\
&+
\big\|\big<\nabla_x\big>^{\alpha}\big<\nabla_y\big>^{\alpha}G \big\|_{\mathcal{S}_r'}
+M^{-\epsilon_0}\big\|\big<\nabla_x\big>^{\alpha}\big<\nabla_y\big>^{\alpha}H \big\|_{L^2(dt)L^{6}(d(x-y))L^2(d(x+y))}\\
&+ \big\|\big<\nabla_x\big>^{\alpha}\big<\nabla_y\big>^{\alpha}\Lambda_0 \big\|_{L^2}.
 \end{align*}

Once $k \epsilon_0 > 1$, we use \eqref{badtwoderiv} to complete the proof.
\end{proof}

Now we move to the low frequency part. Here the collapsing norm can be treated perturbatively.
\begin{theorem}\label{mainthmnewlow}
Let $\Lambda$ satisfy \eqref{maineqold}, and
let $1+$ denote $1+\delta_0$ with $\delta_0 >0$ satisfying \eqref{vestim}, \eqref{vestim1}. Then

\begin{align}
&\big\|\big<\nabla_x\big>^{\alpha}\big<\nabla_y\big>^{\alpha}\label{lowest}
P_{|\xi|<M^{1+}}P_{|\eta|<M^{1+}}\Lambda \big\|_{\mathcal{S}_{x, y}}  \\
&+\big\|\big|\nabla_{x+y}\big|^{\alpha}P_{|\xi|<M^{1+}}P_{|\eta|<M^{1+}}\Lambda\big\|_{L^2(dt)L^{\infty}(d(x-y))L^2(d(x+y))}\notag\\
&+\big\|\big|\partial_t\big|^{\frac{1}{4}}P_{|\xi|<M^{1+}}P_{|\eta|<M^{1+}}\Lambda\big\|_{L^2(dt)L^{\infty}(d(x-y))L^2(d(x+y))}\notag\\
&\lesssim
\big\|\big<\nabla_x\big>^{\alpha}\big<\nabla_y\big>^{\alpha}G \big\|_{\mathcal{S}_r'}
+M^{-\epsilon_0}\big\|\big|\nabla_{x+y}\big|^{\alpha}H\big\|_{collapsing}\notag
\\
&+M^{-\epsilon_0}\big\|\big|\partial_t\big|^{\frac{1}{4}}H\big\|_{collapsing}
+ \big\|\big<\nabla_x\big>^{\alpha}\big<\nabla_y\big>^{\alpha}\Lambda_0 \big\|_{L^2}.
\end{align}

\end{theorem}

We first prove the above theorem without including  $\mathcal S_{x, y}$ on the LHS.

\begin{proof} (excluding the term \eqref{lowest})
\begin{align*}
&\S P_{|\xi|< M^{1+}}P_{|\eta|< M^{1+}}\Lambda\\
& =P_{|\xi|< M^{1+}}P_{|\eta|< M^{1+}}h(t)\left(v_M(x-y)( \Lambda+H)\right)
+ P_{|\xi|< M^{1+}}P_{|\eta|< M^{1+}}h(t)G\\
&=P_{|\xi|< M^{1+}}P_{|\eta|< M^{1+}}h(t)\left(v^1_M(x-y) P_{|\xi|< M^{1+}}P_{|\eta|< M^{1+}}(\Lambda+H)\right)\\
&+P_{|\xi|< M^{1+}}P_{|\eta|< M^{1+}}h(t)\left(v^1_M(x-y) P_{|\xi| \, or \, |\eta|> M^{1+}}(\Lambda+H)\right)\\
&+ P_{|\xi|< M^{1+}}P_{|\eta|< M^{1+}}h(t)G\\
&+P_{|\xi|< M^{1+}}P_{|\eta|< M^{1+}}h(t)\left(v^2_M(x-y) (\Lambda+H)\right).
\end{align*}
We write
\begin{align*}
& P_{|\xi|< M^{1+}}P_{|\eta|< M^{1+}}\Lambda= \Lambda_1+ \Lambda_2 + \Lambda_3 + \Lambda_4 +\Lambda_5 \\
 &= P_{|\xi|<2 M^{1+}}P_{|\eta|<2 M^{1+}}\left(\Lambda_1+ \Lambda_2 + \Lambda_3 + \Lambda_4 +\Lambda_5\right)
 \end{align*}
where $\Lambda_1, \cdots \Lambda_5$ are defined by
\begin{align*}
&\S\Lambda_1 =P_{|\xi|< M^{1+}}P_{|\eta|< M^{1+}}h(t)\left(v_M^1(x-y) P_{|\xi|< M^{1+}}P_{|\eta|< M^{1+}}(\Lambda+H)\right)
\end{align*}
\begin{align*}
&\S \Lambda_2 =P_{|\xi|< M^{1+}}P_{|\eta|< M^{1+}}h(t)\left(v_M^1(x-y) P_{|\xi| \, or \, |\eta|> M^{1+}}
P_{|\xi|<10 M^{1+}}P_{|\eta|< 10 M^{1+}}(\Lambda+H)\right)
\end{align*}
\begin{align*}
&\S \Lambda_3 =h(t) P_{|\xi|< M^{1+}}P_{|\eta|< M^{1+}}G
\end{align*}
with initial conditions $0$, and $\S \Lambda_4=0$ with initial conditions $ P_{|\xi|< M^{1+}}P_{|\eta|< M^{1+}}\Lambda_0$, and finally
\begin{align*}
&\S \Lambda_5 = h(t)P_{|\xi|< M^{1+}}P_{|\eta|< M^{1+}}\left(v^2_M(x-y) (\Lambda+H)\right)
\end{align*}
with initial conditions $0$.
Putting together the five propositions below, we conclude
\begin{align*}
&\big\|\big|\partial_t\big|^{\frac{1}{4}}P_{|\xi|< M^{1+}}P_{|\eta|< M^{1+}}\Lambda\big\|_{L^2(dt)L^{\infty}(d(x-y))L^2(d(x+y))}\\
&
+\big\|\big|\nabla_{x+y}\big|^{\alpha}P_{|\xi|< M^{1+}}P_{|\eta|< M^{1+}}\Lambda\big\|_{L^2(dt)L^{\infty}(d(x-y))L^2(d(x+y))}\\
&\lesssim  M^{- \epsilon_0} \bigg(\big\|\big|\partial_t\big|^{\frac{1}{4}}P_{|\xi|< M^{1+}}P_{|\eta|< M^{1+}}\Lambda\big\|_{collapsing}
+\big\|\big|\nabla_{x+y}\big|^{\alpha}P_{|\xi|< M^{1+}}P_{|\eta|< M^{1+}}\Lambda\big\|_{collapsing}\bigg)\\
&+ M^{- \epsilon_0} \bigg(\big\|\big|\partial_t\big|^{\frac{1}{4}}P_{|\xi|< M^{1+}}P_{|\eta|< M^{1+}}H\big\|_{collapsing}
+\big\|\big|\nabla_{x+y}\big|^{\alpha}P_{|\xi|< M^{1+}}P_{|\eta|< M^{1+}}H\big\|_{collapsing}\bigg)\\
&+ \big\|\big<\nabla_x\big>^{\alpha}\big<\nabla_y\big>^{\alpha}G \big\|_{\mathcal{S}_r'}
+ \big\|\big<\nabla_x\big>^{\alpha}\big<\nabla_y\big>^{\alpha}\Lambda_0 \big\|_{L^2}.
\end{align*}
The first term on the RHS can be absorbed in the LHS, proving the result.
\end{proof}

 We have to prove the following propositions
 \begin{proposition}\label{l1}
Let
\begin{align*}
&\S\Lambda_1 =P_{|\xi|< M^{1+}}P_{|\eta|< M^{1+}}h(t)\left(v_M^1(x-y) P_{|\xi|< M^{1+}}P_{|\eta|< M^{1+}}(\Lambda+H)\right)
\end{align*}
with zero initial conditions. Then
 \begin{align*}
\|\Lambda_1\|_{L^2(dt)L^{\infty}(d(x-y))L^2(d(x+y))} \lesssim M^{-\epsilon_0} \|P_{|\xi|< M^{1+}}P_{|\eta|< M^{1+}}(\Lambda+H)\big\|_{collapsing}
 \end{align*}
 and similarly
 \begin{align*}
&\||\nabla_{x+y}|^{\alpha}\Lambda_1\|_{L^2(dt)L^{\infty}(d(x-y))L^2(d(x+y))}
 \lesssim M^{-\epsilon_0} \|P_{|\xi|< M^{1+}}P_{|\eta|< M^{1+}}|\nabla_{x+y}|^{\alpha}(\Lambda+H)\big\|_{collapsing}\\
&\||\partial_t|^{\frac{1}{4}}\Lambda_1\|_{L^2(dt)L^{\infty}(d(x-y))L^2(d(x+y))}
 \lesssim M^{-\epsilon_0} \|P_{|\xi|< M^{1+}}P_{|\eta|< M^{1+}}|\partial_t|^{\frac{1}{4}}(\Lambda+H)\big\|_{collapsing}.
 \end{align*}

 \end{proposition}
 Since the frequency localization of $\Lambda+H$ plays no role and only the frequency localization of $v_M^1(x-y)\Lambda$ is
 important, we record a slightly more general result, which implies Proposition \ref{l1} and will also be used later.

\begin{proposition}\label{l11}
Let
\begin{align*}
&\S u =P_{|\xi|< M^{1+}}P_{|\eta|< M^{1+}}h(t)\left(v_M(x-y) \Lambda\right)
\end{align*}
with zero initial conditions. Then

 \begin{align*}
\|u\|_{L^2(dt)L^{\infty}(d(x-y))L^2(d(x+y))} \lesssim M^{-\epsilon_0} \|\Lambda\|_{collapsing}
 \end{align*}
 and similarly
 \begin{align*}
&\||\nabla_{x+y}|^{\alpha}u\|_{L^2(dt)L^{\infty}(d(x-y))L^2(d(x+y))}
 \lesssim M^{-\epsilon_0} \||\nabla_{x+y}|^{\alpha}\Lambda\big\|_{collapsing}\\
&\||\partial_t|^{\frac{1}{4}}u\|_{L^2(dt)L^{\infty}(d(x-y))L^2(d(x+y))}
 \lesssim M^{-\epsilon_0} \||\partial_t|^{\frac{1}{4}}\Lambda\big\|_{collapsing}.
 \end{align*}
\end{proposition}
   \begin{proof}
Using Sobolev estimates and Theorem \ref{Suusharp} we get
\begin{align*}
\|u\|_{L^2(dt)L^{\infty}(d(x-y))L^2(d(x+y))} \lesssim
\lesssim M^{-\epsilon_0} \|\Lambda\|_{collapsing}.
\end{align*}
No modifications are needed to prove
\begin{align*}
&\||\nabla_{x+y}|^{\alpha}u\|_{L^2(dt)L^{\infty}(d(x-y))L^2(d(x+y))}
\lesssim M^{-\epsilon_0}
\||\nabla_{x+y}|^{\alpha}\Lambda\|_{collapsing}.
\end{align*}
However, the estimate for $|\partial_t|^{\frac{1}{4}}u$ requires extra care because time derivatives don't preserve initial conditions.
To simplify notation, let $F=P_{|\xi|< M^{1+}}P_{|\eta|< M^{1+}}\left(v_M(x-y) \Lambda\right)$.
 Let $E$ the fundamental solution of the Schr\"odinger equation supported in $t \ge 0$, and recall $h$ is the Heaviside function. The usual solution to
\begin{align*}
\S u =F
\end{align*}
with $0$ initial conditions is given (in the region $t>0$) by
\begin{align*}
u= E * (hF).
\end{align*}
 From the first part of this proof we get
\begin{align*}
&\||\partial_t|^{\frac{1}{4}}u\|_{L^2(dt)L^{\infty}(d(x-y))L^2(d(x+y))}
\lesssim M^{-\epsilon_0} \||\partial_t|^{\frac{1}{4}}(hF)\|_{L^{\infty}(d(x-y))L^2(dt)L^2(d(x+y))}.
\end{align*}
It is easy to show
\begin{align}
  \||\partial_t|^{\frac{1}{4}}(hF)\|_{L^{\infty}(d(x-y))L^2(dt)L^2(d(x+y))}
  \lesssim  \|(|\partial_t|^{\frac{1}{4}}F)\|_{L^{\infty}(d(x-y))L^2(dt)L^2(d(x+y))} .    \label{troubleterm}
  \end{align}
  This can be done by taking Fourier transform in $t$ and using  $A_2$ theory (see \cite{Stein93}), or else the equivalent definition (for $0<k<1$):
\begin{align}
\|u\|^2_{\dot H^k}=\int \int \frac{|u(t)-u(s)|^2}{|t-s|^{1+2 k}}dtds\label{equiv}
\end{align}
and the generalized Hardy's inequality from \cite{Z-Z}.

We remark that the corresponding estimate for the stronger norm $\|\Lambda\|_{L^2(dt)L^{\infty}(d(x-y))L^2(d(x+y))}$ might not be true. We do not know if
  \begin{align}
  \||\partial_t|^{\frac{1}{4}}(hF)\|_{L^2(dt)L^{\infty}(d(x-y))L^2(d(x+y))}
  \lesssim  \|(|\partial_t|^{\frac{1}{4}}F)\|_{L^2(dt)L^{\infty}(d(x-y))L^2(d(x+y))}\label{troubleterm1}
  \end{align}
  is true.

\end{proof}

\begin{proposition}
Recall $\Lambda$ satisfies \eqref{maineqold} and
 \begin{align*}
&\S \Lambda_2 =P_{|\xi|< M^{1+}}P_{|\eta|< M^{1+}}\left(h(t)v_M^1(x-y) P_{|\xi| \, or \, |\eta|> M^{1+}}
P_{|\xi|<10 M^{1+}}P_{|\eta|< 10 M^{1+}}(\Lambda+H)\right)
\end{align*}
with zero initial conditions.
Then
\begin{align*}
&\|\big<\nabla_{x}\big>^{\alpha}\big<\nabla_{y}\big>^{\alpha} \Lambda_2\|_{\mathcal S}+
\|\big<\nabla_{x+y}\big>^{\alpha} \Lambda_2\|_{L^2(dt)L^{\infty}(d(x-y))L^2(d(x+y))}\\
&
+\||\partial_t|^{\frac{1}{4}} \Lambda_2\|_{L^2(dt)L^{\infty}(d(x-y))L^2(d(x+y))} \\
&\lesssim  M^{-\epsilon_0}\|\big<\nabla_{x}\big>^{\alpha}\big<\nabla_{y}\big>^{\alpha}G\|_{\mathcal{S}_r'}\\
&+
M^{-\epsilon_0}\big\|\big<\nabla_x\big>^{\alpha}\big<\nabla_y\big>^{\alpha}H \big\|_{L^2(dt)L^{6}(d(x-y))L^2(d(x+y))}+M^{-\epsilon_0}\big\|\big<\nabla_x\big>^{\alpha}\big<\nabla_y\big>^{\alpha}\Lambda_0 \big\|_{L^2}.
 \end{align*}
\end{proposition}
 \begin{proof}
 Consider (by slight abuse of notation) just one of the contributions to $
\Lambda_2$ with $ P_{|\xi| > M^{1+}}$.
 \begin{align}
 \S \Lambda_2=P_{|\xi|< M^{1+}}P_{|\eta|< M^{1+}}
 \left(v_M^1(x-y)  P_{M^{1+}<|\xi|<10 M^{1+}}P_{|\eta|< 10 M^{1+}}(\Lambda+H)\right).\label{lambda2eq}
\end{align}
Using Theorem \ref{mainStrich} and Lemmas \ref{4collapsing}, \ref{timederlemma} we have
\begin{align*}
&\|\big<\nabla_{x}\big>^{\alpha}\big<\nabla_{y}\big>^{\alpha} \Lambda_2\|_{\mathcal S}+
\|\big<\nabla_{x+y}\big>^{\alpha}\Lambda_2\|_{L^2(dt)L^{\infty}(d(x-y))L^2(d(x+y))}\\
&+\||\partial_t|^{\frac{1}{4}} \Lambda_2\|_{L^2(dt)L^{\infty}(d(x-y))L^2(d(x+y))}
\lesssim \|\big<\nabla_x\big>^{\alpha}\big<\nabla_y\big>^{\alpha} RHS\|_{L^2(dt)L^{\frac{6}{5}}(d(x-y))L^2(d(x+y))}.
\end{align*}

Here we can estimate the RHS directly (using the fractional Leibniz rule
and the fact that $\big<\nabla_x\big>^{\alpha}$ only falls on $\Lambda+H$)
\begin{align*}
&\big\|\big<\nabla_x\big>^{\alpha}\big<\nabla_y\big>^{\alpha}
 \left(v^1_M(x-y) P_{|\xi| > M^{1+}} (\Lambda+H)\right)
 \big\|_{L^2(dt)L^{\frac{6}{5}}(d(x-y))L^2(d(x+y))}\\
 &\lesssim M^{-\epsilon_0}\big\|\big<\nabla_x\big>^{\alpha}\big<\nabla_y\big>^{\alpha} P_{|\xi| > M^{1+}}(\Lambda +H) \big\|_{L^2(dt)L^{6}(d(x-y))L^2(d(x+y))}.
 \end{align*}
This is in the high frequency range, and using Theorem \ref{mainthmnew},
the above is
\begin{align*}
&\lesssim M^{-\epsilon_0}\big\|\big<\nabla_x\big>^{\alpha}\big<\nabla_y\big>^{\alpha}G \big\|_{\mathcal{S}_r'}
+ M^{-\epsilon_0}\big\|\big<\nabla_x\big>^{\alpha}\big<\nabla_y\big>^{\alpha}\Lambda_0 \big\|_{L^2}\\
&+
M^{-\epsilon_0}\big\|\big<\nabla_x\big>^{\alpha}\big<\nabla_y\big>^{\alpha}H \big\|_{L^2(dt)L^{6}(d(x-y))L^2(d(x+y))}.
\end{align*}
\end{proof}

\begin{proposition}\label{l3}
Recall
\begin{align*}
&\S \Lambda_3 = P_{|\xi|< M^{1+}}P_{|\eta|< M^{1+}}G
\end{align*}
with zero initial conditions.
Then
 \begin{align*}
&\|\big<\nabla_{x}\big>^{\alpha}\big<\nabla_{y}\big>^{\alpha} \Lambda_3\|_{\mathcal S}
 +\|\big<\nabla_{x+y}\big>^{\alpha} \Lambda_3\|_{L^2(dt)L^{\infty}(d(x-y))L^2(d(x+y))}\\
 &+
\||\partial_t|^{\frac{1}{4}} \Lambda_3\|_{L^2(dt)L^{\infty}(d(x-y))L^2(d(x+y))}\\
&\lesssim
\|\big<\nabla_{x}\big>^{\alpha}\big<\nabla_{y}\big>^{\alpha}G\|_{\mathcal{S}_r'}.
 \end{align*}
\end{proposition}
\begin{proof}
This follows immediately from Theorem \ref{mainStrich} and Lemmas \ref{4collapsing}, \ref{timederlemma}.
\end{proof}
\begin{proposition}\label{l4}
Recall
\begin{align*}
&\S \Lambda_4 =0
\end{align*}
with  initial conditions $\Lambda_0$.
Then, for any $\alpha > \frac{1}{2}$,
 \begin{align*}
&\|\big<\nabla_{x}\big>^{\alpha}\big<\nabla_{y}\big>^{\alpha} \Lambda_4\|_{\mathcal S}+
 \|\big<\nabla_{x+y}\big>^{\alpha} \Lambda_4\|_{L^2(dt)L^{\infty}(d(x-y))L^2(d(x+y))}+
\||\partial_t|^{\frac{1}{4}} \Lambda_4\|_{L^2(dt)L^{\infty}(d(x-y))L^2(d(x+y))}\\
&\lesssim
 \big\|\big<\nabla_x\big>^{\alpha}\big<\nabla_y\big>^{\alpha}\Lambda_0 \big\|_{L^2}. 
 \end{align*}
\end{proposition}

\begin{proof}  This follows from Strichartz estimates, see for instance the proof of Lemma 5.3 in \cite{G-M2017}.
\end{proof}

\begin{proposition}\label{l5}
Recall $\Lambda$ satisfies \eqref{maineqold} and
\begin{align*}
&\S \Lambda_5 = P_{|\xi|< M^{1+}}P_{|\eta|< M^{1+}}\left(v^2_M(x-y) (\Lambda+H)\right)
\end{align*}
with initial conditions $0$. Then
\begin{align*}
&\|\big<\nabla_{x}\big>^{\alpha}\big<\nabla_{y}\big>^{\alpha} \Lambda_5\|_{\mathcal S}+
 \|\big<\nabla_{x+y}\big>^{\alpha} \Lambda_5\|_{L^2(dt)L^{\infty}(d(x-y))L^2(d(x+y))}\\
 &+
\||\partial_t|^{\frac{1}{4}} \Lambda_5\|_{L^2(dt)L^{\infty}(d(x-y))L^2(d(x+y))}\\
&\lesssim
 M^{-9}\|\big<\nabla_{x}\big>^{\alpha}\big<\nabla_{y}\big>^{\alpha}G\|_{\mathcal{S}_r'}
+ M^{-9}\big\|\big<\nabla_x\big>^{\alpha}\big<\nabla_y\big>^{\alpha}G \big\|_{\mathcal{S}_r'}
 \end{align*}
\end{proposition}
\begin{proof}
Using Proposition \ref{l3}, we have
\begin{align*}
&\|\big<\nabla_{x}\big>^{\alpha}\big<\nabla_{y}\big>^{\alpha} \Lambda_5\|_{\mathcal S}+
 \|\big<\nabla_{x+y}\big>^{\alpha} \Lambda_5\|_{L^2(dt)L^{\infty}(d(x-y))L^2(d(x+y))}+
\||\partial_t|^{\frac{1}{4}} \Lambda_5\|_{L^2(dt)L^{\infty}(d(x-y))L^2(d(x+y))}\\
&\lesssim \|\big<\nabla_{x}\big>^{\alpha}\big<\nabla_{y}\big>^{\alpha}P_{|\xi|< M^{1+}}P_{|\eta|< M^{1+}}\left(v^2_M(x-y) \Lambda\right)\|_{\mathcal{S}_r'}.
\end{align*}
Here we use the Leibniz rule and \eqref{onederiv}, \eqref{badtwoderiv} and the smallness of $v_M^2$ to conclude the above is
\begin{align*}
\lesssim M^{-9} \left(\|\big<\nabla_{x}\big>^{\alpha}\big<\nabla_{y}\big>^{\alpha}G\|_{\mathcal{S}_r'} +
\big\|\big<\nabla_x\big>^{\alpha}\big<\nabla_y\big>^{\alpha}\Lambda_0 \big\|_{L^2}\right).
\end{align*}
\end{proof}

\bigskip

To finish the proof the Theorem \ref{mainthm}, we also need
\begin{theorem}\label{finish}
Let
\begin{align*}
&\S\Lambda_1 =P_{|\xi|< M^{1+}}P_{|\eta|< M^{1+}}\left(v_M^1(x-y) P_{|\xi|< M^{1+}}P_{|\eta|< M^{1+}}\Lambda\right)
\end{align*}
with $0$ initial conditions.
 Then
\begin{align*}
&\big\|\big|\nabla_x\big|^{\alpha}\big|\nabla_y\big|^{\alpha}P_{|\xi|<2 M^{1+}}P_{|\eta|<2 M^{1+}}\Lambda_1 \big\|_{\mathcal{S}_{x, y}}\\
&\lesssim \big\|\big<\nabla_{x+y}\big>^{\alpha}\Lambda\big\|_{collapsing} +\big\|\big<\partial_t\big>^{\frac{1}{4}}\Lambda\big\|_{collapsing}.
\end{align*}
\end{theorem}

\begin{proof} This follows from Proposition \ref{Suusharp} and Bernstein's inequality.

\end{proof}

\section{Estimates for the nonlinear equation, step 1 \label{scalar}}

Recall the notation
\begin{align*}
\S_{\pm}=\frac{1}{i}\frac{\partial}{\partial t} -\Delta_x + \Delta_y
\end{align*}

From now on, $V_N(x)=N^{3 \beta}v(N^{\beta}x)$.
 As in \cite{CGMZ}
we assume the potential satisfies \eqref{Vhyp} and the initial conditions satisfy \eqref{datahyp} and \eqref{smallhyp}, and also \eqref{lambdaestimp} holds.

Define $\Gamma= \Gamma_c+\Gamma_p$, $\Lambda= \Lambda_c+\Lambda_p$, where $\Gamma_c = \bar \phi \otimes \phi$,
$\Lambda_c =  \phi \otimes \phi$,
$\Gamma_p=\frac{1}{N} \shb\circ \sh$ and
 $\Lambda_p=\frac{1}{2N} \sht$.
Let $\rho(t, x)=\Gamma(t, x, x)$.

The 4 relevant equations are
\begin{align}
&\S \Lambda_p +\{V_N * \rho, \Lambda_p\}+
\frac{V_N}{N} \Lambda_p \label{rel1}\\
&+ \left((V_N \bar \Gamma_p)\circ \Lambda_p + (V_N\Lambda_p)\circ  \Gamma_p \right)_{symm} \notag\\
&+ \left((V_N \bar \Gamma_c)\circ \Lambda_p + (V_N\Lambda_c)\circ  \Gamma_p \right)_{symm}=-\frac{V_N}{N} \Lambda_c\notag\\
&\S_{\pm} \Gamma_p + [V_N * \rho, \Gamma_p]+ \left((V_N\Gamma_p)\circ \Gamma_p + (V_N \bar \Lambda_p)\circ  \Lambda_p \right)_{skew} \notag\\
&+ \left((V_N\Gamma_c)\circ \Gamma_p + (V_N\bar \Lambda_c)\circ  \Lambda_p \right)_{skew}=0\ \label{rel2}\\
&\S \Lambda_c +\{V_N * \rho, \Lambda_c\}+ \left((V_N \bar \Gamma_p)\circ \Lambda_c + (V_N\Lambda_p)\circ  \Gamma_c \right)_{symm}=0 \label{rel3}\\
&\S_{\pm} \Gamma_c + [V_N * \rho, \Gamma_c]+ \left((V_N\Gamma_p)\circ \Gamma_c + (V_N\bar \Lambda_p)\circ  \Lambda_c \right)_{skew}=0 . \label{rel4}
\end{align}
Here $\left(A(x, y)\right)_{symm}=A(x, y)+A(y, x)$ and $\left(A(x, y)\right)_{skew}=A(x, y) -\bar A(y, x)$

The norms used for $\Lambda_p$ and $\Lambda_c$ are called $\N(\Lambda)$ and are
\begin{align*}
&\|\Lambda\|_{\N(\Lambda)}=\big\|\big<\nabla_x\big>^{\alpha}\big<\nabla_y\big>^{\alpha}\Lambda \big\|_{\mathcal{S}_{x, y}}\\
&+\big\|\big|\partial_t\big|^{\frac{1}{4}}\Lambda\big\|_{L^2(dt)L^{\infty}(d(x-y))L^2(d(x+y))}
+\big\|\big<\nabla_{x+y}\big>^{\alpha}\Lambda\big\|_{L^2(dt)L^{\infty}(d(x-y))L^2(d(x+y))}.\notag\\
\end{align*}
Because of the quarter time derivative, these norms cannot be localized in an obvious way, and we will devise a way to get around that.
The norms for $\Gamma$, $\Gamma_p$,  $\Gamma_c$  and also $\Lambda_c$ are
\begin{align*}
&\|F\|_{\N^1}
=  \sup_{  p, q \, \, admissible
}\|\big<\nabla_x\big>^{\alpha}\big<\nabla_y\big>^{\alpha}F\|_{L^{p}(dt) L^{q}(dx) L^2(dy)}\notag\\
&\quad +\sup_{  p, q \, \notag \, admissible}\|\big<\nabla_x\big>^{\alpha}\big<\nabla_y\big>^{\alpha}F\|_{L^{p}(dt) L^{q}(dy) L^2(dx)}\\
&+\big\||\nabla_{x+y}|^{\alpha}F\big\|_{L^{\infty}(d(x-y))L^2(dt)L^2(d(x+y))}.
\end{align*}
The estimates for the linear part of the $\Lambda$ equations have been studied in the previous sections. For the $\Gamma$ equation we will only use older, standard estimates
\begin{proposition}\label{estforGamma}
Let
\begin{align*}
&\S_{\pm} \Gamma = F \\
&\Gamma(0, \cdot)= \Gamma_0
\end{align*}
Then
\begin{align*}
& \sup_{  p, q \, \, admissible
}\|\Gamma\|_{L^{p}(dt) L^{q}(dx) L^2(dy)}\notag\\
&\quad +\sup_{  p, q \, \notag \, admissible}\|\Gamma\|_{L^{p}(dt) L^{q}(dy) L^2(dx)} \\
&\lesssim \|\Gamma_0\|_{L^2} +
 \inf_{  p, q \, \, admissible,  p \ge p_0>2} \{\|F\|_{L^{p'}(dt) L^{q'}(dx) L^2(dy)},
\|F\|_{L^{p'}(dt) L^{q'}(dy) L^2(dx)}\}\\
&\sup_z\||\nabla_x|^{\alpha}\Gamma(t, x+z, x)\|_{L^2(dt dx)}\\
&\lesssim \|\big<\nabla_x\big>^{\alpha}\big<\nabla_y\big>^{\alpha}\Gamma_0\|_{L^2}\\
&+
 \inf_{  p, q \, \, admissible,  p \ge p_0>2} \{\|\big<\nabla_x\big>^{\alpha}\big<\nabla_y\big>^{\alpha}F\|_{L^{p'}(dt) L^{q'}(dx) L^2(dy)},
\|F\|_{L^{p'}(dt) L^{q'}(dy) L^2(dx)}\}.
\end{align*}
\end{proposition}

\begin{proof} For a proof of the homogeneous estimate, see Lemmas 5.1, 5.3 in \cite{G-M2017}, and also \cite{CHP}.
The inhomogeneous estimate follows from the Christ-Kiselev lemma.
Let $T_1=e^{i t \Delta_{\pm}}$, so $T_1:L^2(\mathbb R^6) \to L^p(dt) L^q(dx)L^2(dy)$ and
$T_1^*: L^{p'}(dt)W^{\alpha, q'}(dx)H^{\alpha}(dy) \to H^{\alpha}(dx) H^{\alpha}(dy)$.
Fix $z$ and let
$T_2:H^{\alpha}(dx) H^{\alpha}(dy)\to L^2(dt) \dot H^{\alpha}(dx)$ be the operator $f \to \left(e^{i t \Delta_{\pm}}f\right)(t, x, x+z)$.
Then the inhomogeneous estimate follows by applying the Christ-Kiselev lemma to $T_2 T_1^*$.

\end{proof}

The first step in the analysis of the nonlinear equations uses a priori estimates for $\Gamma(t, x, x)$, see
\eqref{morawetzGamma} and \eqref{morawetzGammaenergy}.

\begin{lemma}\label{firstsmallness} Under the assumptions of Theorem \ref{intro1} we have
\begin{align}
&\|\Gamma\|_{L^8(dt))L^{\infty}(d(x-y))L^{\frac{4}{3}}(d(x+y))} \lesssim 1 \label{timeest1} \\
&\|\nabla_{x+y}\Gamma\|_{L^8(dt))L^{\infty}(d(x-y))L^{\frac{4}{3}}(d(x+y))} \label{timeest2} \lesssim 1
\end{align}
and thus,
for any $\epsilon_2>0$ there exist $n = n(\epsilon_2)$  intervals $[T_i, T_{i+1}]$ covering $[0, \infty)$ such that
\begin{align}
&\sup_z\|\big<\nabla_x\big>^{\alpha} \Gamma(t, x+z, x)\|_{L^8([T_j, T_{j+1}])L^{\frac{4}{3}}(dx)} \le \epsilon_2. \label{epsilon2}
\end{align}
\end{lemma}
The idea of proving estimates for NLS by using such a localization in time goes back to  Bourgain \cite{Bourgain98}.
\begin{remark}\label{localization} Notice that the "collapsing estimate" $\sup_z\|\big<\nabla_x\big>^{\alpha} \Gamma(t, x+z, x)\|_{L^2(dt dx)} \lesssim 1$
does not imply
there exist $n = n(\epsilon_2)$  intervals $[T_i, T_{i+1}]$ covering $[0, \infty)$ such that
\begin{align*}
&\sup_z\|\big<\nabla_x\big>^{\alpha} \Gamma(t, x+z, x)\|_{L^2([T_j, T_{j+1}])L^{2}(dx)} \le \epsilon_2.
\end{align*}
\end{remark}

\begin{proof}
We  have a pointwise estimate
\begin{align*}
||\Gamma(t, x+z, x-z)| \le |\Gamma(t, x+z, x+z)|^{\frac{1}{2}}|\Gamma(t, x-z, x-z)|^{\frac{1}{2}}
\end{align*}
and also
\begin{align*}
&||\nabla_x\Gamma(t, x+z, x-z)|\\
& \le |E_1(t, x+z, x+z)|^{\frac{1}{2}}|\Gamma(t, x-z, x-z)|^{\frac{1}{2}}+
|E_1(t, x-z, x-z)|^{\frac{1}{2}}|\Gamma(t, x+z, x+z)|^{\frac{1}{2}}
\end{align*}
where $E_1(t, x)= \nabla_x \cdot \nabla_y \Gamma(t, x, y)\big|_{x=y}$ is the kinetic energy density,
with $\int |E_1(t, x)| + \Gamma(t, x, x) dx$ uniformly bounded in time
and $\|\Gamma(t, x, x)\|_{L^4(dt)L^2(dx)} \lesssim 1$.
\eqref{timeest1} and \eqref{timeest2} follow by applying H\"older's inequality, and these imply \eqref{epsilon2}.
\end{proof}
The above estimates hold for $\Gamma_c$ and $\Gamma_p$ separately. Since $\Gamma_c = \bar \phi \otimes \phi $ and $\Lambda_c=
 \phi \otimes \phi $, they also hold for $\Lambda_c$. No such a priori estimates are available for $\Lambda_p$. However, we assume
 initial conditions for $\Lambda_p$ are small. Also, the forcing term in the equation for $\Lambda_p$ is $\frac{V_N}{2N} \Lambda_c$, and this is small in suitable norms,  thus $\|\Lambda_p\|_{\N(\Lambda)}$ will stay small. Here are the details:

In order to use the smallness of the above quantities, we have to localize our estimates to these intervals. However, the necessary norms involve a non-local term, so we have to proceed carefully. Also, we will use a continuity argument, so the right end of the interval must be a variable $T \in [T_i, T_{i+1}]$.
 Define $\Lambda_c^{i, T}$ $\Lambda_p^{i, T}$, $\Gamma_c^{i, T}$, $\Gamma_p^{i, T} $ be the solution to the standard equations with the RHS multiplied by
 $\chi_{[T_i, T]}$:

 \begin{align}
&\S \Lambda^{i, T}_p +
\frac{V_N}{N} \Lambda^{i, T}_p \label{rel11}\\
&=\chi_{[T_i, T]} \bigg(
-\{V_N * \rho, \Lambda_p\}
- \left((V_N \bar \Gamma_p)\circ \Lambda_p + (V_N\Lambda_p)\circ  \Gamma_p \right)_{symm} \notag\\
&- \left((V_N \bar \Gamma_c)\circ \Lambda_p + (V_N\Lambda_c)\circ  \Gamma_p \right)_{symm}
-\frac{V_N}{N} \Lambda_c\notag\bigg)\\
&\S_{\pm} \Gamma^{i, T}_p =\chi_{[T_i, T]}\bigg(- [V_N * \rho, \Gamma_p]- \left((V_N\Gamma_p)\circ \Gamma_p - (V_N \bar \Lambda_p)\circ  \Lambda_p \right)_{skew} \notag\\
&- \left((V_N\Gamma_c)\circ \Gamma_p - (V_N\bar \Lambda_c)\circ  \Lambda_p \right)_{skew}\bigg) \label{rel12}\\
&\S \Lambda^{i, T}_c =\chi_{[T_i, T]}\bigg(-\{V_N * \rho, \Lambda_c\}- \left((V_N \bar \Gamma_p)\circ \Lambda_c - (V_N\Lambda_p)\circ  \Gamma_c \right)_{symm} \bigg)\label{rel13}\\
&\S_{\pm} \Gamma^{i, T}_c =\chi_{[T_i, T]}\bigg(- [V_N * \rho, \Gamma_c]- \left((V_N\Gamma_p)\circ \Gamma_c - (V_N\bar \Lambda_p)\circ  \Lambda_c \right)_{skew}\bigg)  \label{rel14}
\end{align}
with
$\Lambda_c^{i, T}(T_i, \cdot)=\Lambda_c^{i-1, T_i}(T_i, \cdot)=\Lambda_c(T_i, \cdot)$, and similarly for the other three functions. Also,
Then $\Lambda_c^{i, T}=\Lambda_c$, etc. in $[T_i, T]$ (but not outside this interval), and
 similarly for the other three functions. Also, $\Lambda_c^{i, T_i}$, etc. satisfies a homogeneous linear equation.

We continue by estimating the four functions on the LHS. The most difficult one is $\Lambda_p^{i, T}$.

Later it will be convenient to have norms which can be made small on small time intervals, so we introduce the restricted Strichartz norms
\begin{align*}
&\|F\|_{\mathcal{S}^r_{x, y}} \sim
\|\Lambda\|_{L^2(dt)L^6(dx)L^2(dy)}
+\|\Lambda\|_{L^2(dt)L^6(dy)L^2(dx)}\\
&+\|F\|_{L^{4}(dt)L^{3}(dx)L^2(dy)}
+\|F\|_{L^{4}(dt)L^{3}(dx)L^2(dy)}.
\end{align*}
Notice this is not dual to $\mathcal S'_r$, the restrictions on the exponents are different.

\begin{theorem}\label{terms1}
 Let $[T_j, T]$, $T_i \le T \le T_{i+1}$, and $\epsilon_2$ be as in \eqref{epsilon2}. There exists a universal constant $C$ such that
 \begin{align*}
 &\|\Lambda_p^{i, T}\|_{\N(\Lambda)}\le C \big\|\big<\nabla_x\big>^{\alpha}\big<\nabla_y\big>^{\alpha}\Lambda_p^{i, T}(T_j, \cdot) \big\|_{L^2} + C \epsilon_2 \|\Lambda_p^{i, T}\|_{\N(\Lambda)}\\
 &+C \epsilon_2 \|\Gamma_p^{i, T}\|_{\N^1}+ C \|\Lambda_p^{i, T}\|_{\N(\Lambda)}\|\Gamma_p^{i, T}\|_{\mathcal S^r_{x, y}}\\
& + C N^{-\epsilon_1}\|\Lambda_c^{i, T}\|_{\N(\Lambda)}.
 \end{align*}
 \end{theorem}
\begin{remark} Of course $\|\Gamma_p^{i, T}\|_{\mathcal S^r_{x, y}} \le \|\Gamma_p^{i, T}\|_{\mathcal N^1}$, but $\|\Gamma_p^{i, T}\|_{\mathcal S^r_{x, y}}$ can also be made small on small time intervals. This will be useful when estimating higher order derivatives.
\end{remark}
 The proof is based on Theorem \ref{mainthm}:
there exists a constant $C$ and $\epsilon_1>0 $ such that
\begin{align}
&\|\Lambda_p^{i, T}\|_{\mathcal{N}(\Lambda)} \notag\\
&\le C \bigg(
\big\|\big<\nabla_x\big>^{\alpha}\big<\nabla_y\big>^{\alpha}\chi_{[T_i, T]}\big(
\{V_N * \rho, \Lambda_p\}
+ \left((V_N\bar \Gamma_p)\circ \Lambda_p + (V_N\Lambda_p)\circ  \Gamma_p \right)_{symm} \notag\\
&+ \left((V_N\bar\Gamma_c)\circ \Lambda_p + (V_N\Lambda_c)\circ  \Gamma_p \right)_{symm}\big)\notag
 \big\|_{\mathcal{S}_r'}\\
&+N^{-\epsilon_1}\big\|\big<\nabla_x\big>^{\alpha}\big<\nabla_y\big>^{\alpha}\chi_{[T_i, T]}\notag
\Lambda_c
 \big\|_{L^2(dt)L^{6}(d(x-y))L^2(d(x+y))}\\
&+N^{-\epsilon_1}\big\|\big|\partial_t\big|^{\frac{1}{4}}\chi_{[T_i, T]}\Lambda_c\big\|_{L^{\infty}(d(x-y))L^2(dt )L^2(d(x+y))} \label{tricky}\\
&+N^{-\epsilon_1}\big\|\big|\nabla_{x+y}\big|^{\alpha} \chi_{[T_i, T]}\Lambda_c\big\|_{L^{\infty}(d(x-y))L^2(dt)L^2(d(x+y))}\notag\\
&+ \big\|\big<\nabla_x\big>^{\alpha}\big<\nabla_y\big>^{\alpha}\Lambda_p(T_i) \big\|_{L^2}.\notag
 \bigg)
\end{align}
For all terms other than \eqref{tricky}, the superscript $i, T$ can be trivially added to $\Lambda$, $\Gamma$ on the RHS. In \eqref{tricky},
\begin{align*}
&\big\|\big|\partial_t\big|^{\frac{1}{4}}\big(\chi_{[T_i, T]}\Lambda_c\big)\big\|_{L^{\infty}(d(x-y))L^2(\mathbb R)L^2(d(x+y))}\\
&=\big\|\big|\partial_t\big|^{\frac{1}{4}}\big(\chi_{[T_i, T]}\Lambda^{i, T}_c\big)\big\|_{L^{\infty}(d(x-y))L^2(\mathbb R)L^2(d(x+y))}\\
&\lesssim \big\|\big|\partial_t\big|^{\frac{1}{4}}\Lambda^{i, T}_c\big\|_{L^{\infty}(d(x-y))L^2(\mathbb R)L^2(d(x+y))}\\
&\mbox{(as explained for  \eqref{troubleterm})}\\
&\le \big\|\big|\partial_t\big|^{\frac{1}{4}}\Lambda^{i, T}_c\big\|_{L^2(dt)L^{\infty}(d(x-y))L^2(d(x+y))}\\
&\le \|\Lambda^{i, T}_c\|_{\mathcal{N}(\Lambda)}.
\end{align*}

In the lemmas that follow, we estimate
the norms of the nonlinear terms in suitable dual Strichartz norms, using the bound \eqref{epsilon2} whenever possible.
\begin{lemma}\label{Lem1} Let $[T_i, T]$ be as above. There exists a universal constant $C$ such that
\begin{align*}
&\big\|\big<\nabla_x\big>^{\alpha}\big<\nabla_y\big>^{\alpha}\bigg(
\{V_N * \rho, \Lambda^{i, T}_p\}
+ V_N\bar \Gamma^{i, T}\circ \Lambda^{i, T}_p \bigg)
 \big\|_{L^{\frac{8}{5}}([T_i, T])L^{\frac{4}{3}}(dx)L^2(dy)}\\
  &\le C\epsilon_2 \|\big<\nabla_x\big>^{\alpha}\big<\nabla_y\big>^{\alpha}\Lambda^{i, T}_p\|_{\mathcal S^r_{x, y}}.\\
 \end{align*}
 The result depends on the a priori bounds for $\Gamma$, but is true with $\Lambda^{i, T}_p$ replaced with  any other function).
 \end{lemma}

\begin{proof}
The proof is easily reduced to estimating
\begin{align*}
&\sup_z\|\big<\nabla_x\big>^{\alpha}\big<\nabla_y\big>^{\alpha}\big(\Gamma(t, x, x+z)\Lambda^{i, T}_p(t, x+z, y)\big)\|_{L^{\frac{8}{5}}([T_i, T])L^{\frac{4}{3}}(dx)L^2(dy)}\\
& \le C\epsilon_2 \|\big<\nabla_x\big>^{\alpha}\big<\nabla_y\big>^{\alpha}\Lambda^{i, T}_p\|_{\mathcal S_{x, y}}.
\end{align*}
Using the fractional Leibniz rule (see the proof of Theorem \ref{leibniz}) we have the following estimate, uniformly in $z$:
\begin{align*}
&\|\big<\nabla_x\big>^{\alpha}\big<\nabla_y\big>^{\alpha}\big(\Gamma(t, x, x+z)\Lambda^{i, T}_p(t, x+z, y)\big)\|_{L^{\frac{8}{5}}([T_i, T])L^{\frac{4}{3}}(dx)L^2(dy)}\\
&\le C  \|\big<\nabla_x\big>^{\alpha} \Gamma(t, x, x+z)\|_{L^8([T_i, T])L^{\frac{4}{3}}(dx)}\|\big<\nabla_y\big>^{\alpha}\Lambda^{i, T}_p\|_{L^2(dt)L^{\infty}(dx)L^2(dy)}\\
&+\| \Gamma(t, x, x+z)\|_{L^8([T_i, T])L^{\frac{12}{7}}(dx)}\|\big<\nabla_x\big>^{\alpha}\big<\nabla_y\big>^{\alpha}\Lambda^{i, T}_p\|_{L^2(dt)L^{6}(dx)L^2(dy)}\\
& \le C \|\big<\nabla_x\big>^{\alpha} \Gamma(t, x, x+z)\|_{L^8([T_{i}, T])L^{\frac{4}{3}}(dx)}\|\big<\nabla_x\big>^{\alpha}\big<\nabla_y\big>^{\alpha}\Lambda^{i, T}_p\|_{L^2(dt)L^{6}(dx)L^2(dy)}\\
&\le C_2 \epsilon_2 \|\big<\nabla_x\big>^{\alpha}\big<\nabla_y\big>^{\alpha}\Lambda^{i, T}_p\|_{L^2(dt)L^{6}(dx)L^2(dy)}.
\end{align*}

\end{proof}
Since $\Lambda_c$ satisfies the same a priori estimates (based on interaction Morawetz and conservation of energy) as $\Gamma$, by the exact same argument we get
\begin{lemma} \label{Lem2}Let
 $[T_i, T]$ be as above.
There exists a universal constant $C$ such that
\begin{align*}
&\|\big<\nabla_x\big>^{\alpha}\big<\nabla_y\big>^{\alpha}\left((V_N\Lambda_c)\circ  \Gamma^{i, T}_p \right)\|_{L^{\frac{8}{5}}([T_i, T])L^{\frac{4}{3}}(dx)L^2(dy)}\\
& \le C\epsilon_2 \|\big<\nabla_x\big>^{\alpha}\big<\nabla_y\big>^{\alpha}\Gamma^{i, T}_p\|_{\mathcal S_{x, y}}.
\end{align*}
The result depends on the a priori bounds for $\Lambda_c$, but is true for any function $\Gamma^{i, T}_p=\Gamma^{i, T}_p(t, x, y)$.
\end{lemma}
We continue estimating nonlinear terms.

\begin{lemma} \label{Lem3}
 There exists a universal constant $C$  such that
\begin{align*}
&\|\big<\nabla_x\big>^{\alpha}\big<\nabla_y\big>^{\alpha}\left((V_N\Lambda^{i, T}_p)\circ  \Gamma^{i, T}_p \right)\|_{L^{\frac{4}{3}}(dt)L^{\frac{3}{2}}(dx)L^2(dy)}\\
&\le
 C\|\big<\nabla_{x+y}\big>^{\alpha}\Lambda^{i, T}_p\|_{L^{\infty}(d(x-y))L^2(dt)L^2(dx)}
 \|\big<\nabla_x\big>^{\alpha}\big<\nabla_y\big>^{\alpha}\Gamma^{i, T}_p(t, x, y)
\|_{L^{4}(dt)L^{3}(dx)L^2(dy)}.
\end{align*}
This result can be localized to any time interval $[T_i, T_{i+1}]$ and is true for any two functions, not just $\Lambda^{i, T}_p$ and
$\Gamma^{i, T}_p $.
\end{lemma}
\begin{proof}
It suffices to estimate
\begin{align*}
&\sup_z\|\big<\nabla_x\big>^{\alpha}\big<\nabla_y\big>^{\alpha}\left(\Lambda^{i, T}_p(t, x, x+z)  \Gamma^{i, T}_p(t, x+z, y)\right) \|_{L^{\frac{4}{3}}(dt)L^{\frac{3}{2}}(dx)L^2(dy)}.
\end{align*}
The following holds, uniformly in $z$:
\begin{align*}
&\|\big<\nabla_x\big>^{\alpha}\big<\nabla_y\big>^{\alpha}\left(\Lambda^{i, T}_p(t, x, x+z)  \Gamma^{i, T}_p(t, x+z, y)\right) \|_{L^{\frac{4}{3}}(dt)L^{\frac{3}{2}}(dx)L^2(dy)}\\
&\lesssim \|\big<\nabla_x\big>^{\alpha}\Lambda^{i, T}_p(t, x, x+z)\|_{L^2(dt)L^2(dx)}\|\big<\nabla_y\big>^{\alpha}\Gamma^{i, T}_p(t, x, y)
\|_{L^{4}(dt)L^{6}(dx)L^2(dy)}\\
&+\|\Lambda^{i, T}_p(t, x, x+z)\|_{L^2(dt)L^3(dx)}\|\big<\nabla_x\big>^{\alpha}\big<\nabla_y\big>^{\alpha}\Gamma^{i, T}_p(t, x, y)
\|_{L^{4}(dt)L^{3}(dx)L^2(dy)}\\
&\lesssim \|\big<\nabla_x\big>^{\alpha}\Lambda^{i, T}_p(t, x, x+z)\|_{L^2(dt)L^2(dx)}\|\big<\nabla_x\big>^{\alpha}\big<\nabla_y\big>^{\alpha}\Gamma^{i, T}_p(t, x, y)
\|_{L^{4}(dt)L^{3}(dx)L^2(dy)}.
\end{align*}

 \end{proof}

 We continue with estimates for $\|\Lambda^{i, T}_c\|_{\N(\Lambda)}$. This is an easy version of the previous theorem. Using the previous lemmas and the trivial version of Theorem \ref{mainthm} (without the potential term) we get

 \begin{theorem}\label{terms2}
 Let $[T_i, T]$ be as above. There exists a universal constant $C$ such that
 \begin{align*}
 &\|\Lambda^{i, T}_c\|_{\N(\Lambda)}\le C \big\|\big<\nabla_x\big>^{\alpha}\big<\nabla_y\big>^{\alpha}\Lambda_c^{i, T}(T_i, \cdot)
\big\|_{L^2} + C \epsilon_2 \|\Lambda_c^{i, T}\|_{\N(\Lambda)}\\
 &+ C \|\Lambda^{i, T}_p\|_{\N(\Lambda)}\|\Gamma^{i, T}_c\|_{\mathcal S^r_{x, y}}.
 \end{align*}
 \end{theorem}

 Using  Strichartz estimates for $\S_{\pm}$, and Lemmas \ref{Lem1}-\ref{Lem3} we get
 \begin{theorem}\label{terms3}
 Let $[T_i, T]$ be as above. There exists a universal constant $C$ such that
 \begin{align*}
 &\|\Gamma^{i, T}_c\|_{\N^1}\le C \big\|\big<\nabla_x\big>^{\alpha}\big<\nabla_y\big>^{\alpha}\Gamma_c^{i, T}(T_i, \cdot) \big\|_{L^2} + C \epsilon_2 \|\Gamma^{i, T}_c\|_{\N^1}\\
 &+ C \|\Lambda^{i, T}_p\|_{\N(\Lambda)}\|\Lambda^{i, T}_c\|_{\N(\Lambda)}\\
 &\|\Gamma_p\|_{\N^1}\le C \big\|\big<\nabla_x\big>^{\alpha}\big<\nabla_y\big>^{\alpha}\Gamma^{i, T}_p(T_i, \cdot) \big\|_{L^2} + C \epsilon_2 \|\Gamma^{i, T}_p\|_{\N^1}\\
 &+ C \epsilon_2 \|\Lambda^{i, T}_p\|_{\N^1}+ C \|\Lambda^{i, T}_p\|_{\N(\Lambda)}\|\Lambda^{i, T}_p\|_{\N(\Lambda)}.
 \end{align*}
 \end{theorem}
 At this stage we take $C\epsilon_2 < \frac{1}{2}$.
 This determines the number of intervals in the list $T_i, T_{i+1}$. Call that number $n$, and notice it is independent of $N$. Also,
   for $T \in [T_i, T_{i+1}]$ denote
 \begin{align*}
 &X_i(T)=\|\Lambda^{i, T}_c\|_{\N(\Lambda)}+\|\Gamma^{i, T}_c\|_{\N^1}\\
 &Y_i(T)=\|\Lambda^{i, T}_p\|_{\N(\Lambda)}+\|\Gamma^{i, T}_p\|_{\N^1}.
 \end{align*}

Since we trivially have bounds on $\|\Lambda_{p \, \, or \, \, c}\|_{L^{\infty}[0, T]H^s(dxdy)}$ and
   $\|\Gamma_{p \, \, or \, \, c}\|_{L^{\infty}[0, T]H^s(dxdy)}$ (for any $s$)
   which can grow with $T$ and $N$, then we do know $X_i$, $Y_i$ are continuous.

 We have established the following estimate for $T_i\le T \le T_{i+1}$:
 \begin{corollary}\label{XYcor} The functions $X_i$, $Y_i$ are continuous and satisfy
 \begin{align*}
 &X_i(T) \le C \big\|\big<\nabla_x\big>^{\alpha}\big<\nabla_y\big>^{\alpha}\Lambda_c(T_i, \cdot) \big\|_{L^2}
 +C \big\|\big<\nabla_x\big>^{\alpha}\big<\nabla_y\big>^{\alpha}\Gamma_c(T_i, \cdot) \big\|_{L^2}
 + C X_i(T)Y_i(T)\\
  &Y_i(T) \le C \big\|\big<\nabla_x\big>^{\alpha}\big<\nabla_y\big>^{\alpha}\Lambda_p(T_i, \cdot) \big\|_{L^2}
 +C \big\|\big<\nabla_x\big>^{\alpha}\big<\nabla_y\big>^{\alpha}\Gamma_p(T_i, \cdot) \big\|_{L^2}\\
 &
 + C Y_i^2(T)+ CN^{-\epsilon_1}X_i(T).
 \end{align*}
 \end{corollary}

 Now we can state and prove the main theorem of this section.
 \begin{theorem}\label{step1thm} Assume $\Lambda$, $\Gamma$ and $\phi$ are smooth solutions to the HFB system, with finite energy per particle, uniformly in $N$ (see \eqref{energy}).
 In particular,
 \begin{align*}
  \big\|\big<\nabla_x\big>^{\alpha}\big<\nabla_y\big>^{\alpha}\Lambda_c(0, \cdot) \big\|_{L^2}
 + \big\|\big<\nabla_x\big>^{\alpha}\big<\nabla_y\big>^{\alpha}\Gamma_c(0, \cdot) \big\|_{L^2} \le C.
 \end{align*}
 Assume, in addition,
  \begin{align*}
   \big\|\big<\nabla_x\big>^{\alpha}\big<\nabla_y\big>^{\alpha}\Gamma_p(0, \cdot) \big\|_{L^2} \le \frac{C}{N^{\epsilon_3}}
 \end{align*}
 ($\epsilon_3$ as in \eqref{lambdaestimp}; the corresponding estimate for $\Lambda_p$ holds globally in time).
 Let $\epsilon >0$. There exists $N_0$ such that, if $N \ge N_0$,
  then
 \begin{align*}
 &\|\Lambda_c\|_{\N(\Lambda)}+\|\Gamma_c\|_{\N^1} \lesssim 1\\
 & \|\Lambda_p\|_{\N(\Lambda)}+\|\Gamma_p\|_{\N^1} \le \epsilon.
 \end{align*}
 \end{theorem}
 \begin{proof}

Starting at $T_1=0$, we have $X_1(0) \le C$ and $Y_1(0) \le  CN^{- \epsilon_3}$, and
\begin{align*}
&X_1(T)\le C + C X_1(T)Y_1(T)\\
&Y_1(T)\le CN^{- \epsilon_3} +C  Y_1(T)^2+ CN^{-\epsilon_1} X_1(T).
\end{align*}

In the second line, either
$Y_1(T)\le 2 CY_1(T)^2$, but by continuity we rule this out, or else
 $Y_1(t)\le 2(  CN^{- \epsilon_3}+ CN^{-\epsilon_1} X_1(t))$, and if we plug this in the first line we get
 \begin{align*}
&X_1(T)\le C + C \big(  N^{- \epsilon_3}+ N^{-\epsilon_1} X_1(T)\big)X_1(T).
\end{align*}
If $N$ is sufficiently large, we get $X_1(T)\le 2 C$. We continue to the next interval, $[T_2, T_3]$. The argument is the same,
the initial conditions for $X_2$ have the same bound, but the initial conditions for $Y_2$ have grown:
\begin{align*}
Y_2(T_2) \le  CN^{- \epsilon_3} +  CN^{- \epsilon_1}.
\end{align*}
We can repeat the argument as long as $Y_i(T_i)$ is sufficiently small. This will be the case if $N$ is sufficiently large, because $n$, the number of intervals, is independent of $N$.
\end{proof}

\section{Higher order derivatives}

 Next, we refine the argument to include $x+y$ derivatives. This section uses additional smallness results. Denote $\mathcal S^r_{x, y}[T_1, T_2]$ the standard Strichartz norms subject to the restriction $2\le p \le p_1 < \infty$ for some large $p_1< \infty$ and $t \in [T_1, T_2]$.

 \begin{lemma}
  \label{smallness} Under the assumptions of Theorem \ref{step1thm},
given $\epsilon >0$, we can divide $[0, \infty)$ into finitely many intervals (depending only on $\epsilon$ and the above implicit bounds, as well as those used for \eqref{epsilon2}) such that
\begin{align*}
&\|\big<\nabla_{x+y}\big>^{\alpha}\Gamma_{p \, and \, c}\big\|_{L^{\infty}(d(x-y))L^8([T_j, T_{j+1}])L^{\frac{4}{3}}(d(x+y))} \le \epsilon\\
&\|\big<\nabla_{x+y}\big>^{\alpha}\Lambda_{p \, and \, c}\big\|_{L^2([T_j, T_{j+1}])L^{\infty}(d(x-y))L^2(d(x+y))} \le \epsilon\\
&\big\|\big<\nabla_x\big>^{\alpha}\big<\nabla_y\big>^{\alpha}\Gamma_{p \, and \, c}\|_{\mathcal S^r_{x, y}[T_j, T_{j+1}]} \le \epsilon\\
&\big\|\big<\nabla_x\big>^{\alpha}\big<\nabla_y\big>^{\alpha}\Lambda_{p \, and \, c}\|_{\mathcal S^r_{x, y}[T_j, T_{j+1}]} \le \epsilon.
\end{align*}

 \end{lemma}

\begin{proof}
The first estimate is \eqref{epsilon2}.
 Using Theorem \ref{step1thm}
 we  know the following quantities are bounded:
\begin{align*}
&\|\big<\nabla_{x+y}\big>^{\alpha}\Lambda_{p \, and \, c}\big\|_{L^2(dt)L^{\infty}(d(x-y))L^2(d(x+y))} \lesssim 1\\
&\big\|\big<\nabla_x\big>^{\alpha}\big<\nabla_y\big>^{\alpha}\Gamma_{p \, and \, c}\|_{\mathcal S_{x, y}} \lesssim 1\\
&\big\|\big<\nabla_x\big>^{\alpha}\big<\nabla_y\big>^{\alpha}\Lambda_{p \, and \, c}\|_{\mathcal S_{x, y}} \lesssim 1
\end{align*}
and the result follows.
\end{proof}

\begin{remark} However, it is not clear we can insure
\begin{align*}
&\||\nabla_{x+y}|^{\alpha}\Gamma_{p \, and \, c}\big\|_{L^{\infty}(d(x-y))L^2(t \in [T_1, T_2]) L^2(d(x+y))} \le \epsilon
\end{align*}
which is why we use
$L^{\infty}(d(x-y))L^8(t \in [T_1, T_2]) L^{\frac{4}{3}}(d(x+y))$ for which we have an
  a priori estimate, and which has the same scaling.
\end{remark}

\begin{theorem} \label{step2thm}
Under the assumption of Theorem \ref{step1thm}, if we also have
\begin{align}
&\big\|\big<\nabla_x\big>^{\alpha}\big<\nabla_y\big>^{\alpha}\nabla^j_{x+y} \Lambda_{p \, and \, c}(0, \cdot)\big\|_{L^2}+ \notag
\big\|\big<\nabla_x\big>^{\alpha}\big<\nabla_y\big>^{\alpha}\nabla^j_{x+y} \Gamma_{p \, and \, c}(0, \cdot)\big\|_{L^2}\\
& \lesssim 1\notag
\end{align}
for $j=1,  \cdots, j_0$, then
\begin{align}
 &\|\nabla^j_{x+y}\Lambda_c\|_{\N(\Lambda)}+\|\nabla^j_{x+y}\Gamma_c\|_{\N^1} \lesssim 1 \label{derivest0}\\
 & \|\nabla^j_{x+y}\Lambda_p\|_{\N(\Lambda)}+\|\nabla^j_{x+y}\Gamma_p\|_{\N^1} \lesssim 1. \label{derivest}
 \end{align}
\end{theorem}
\begin{proof} (Sketch)
At this stage, we don't have to distinguish between $\Lambda_p$ and $\Lambda_c$, or $\Gamma_p$ and $\Gamma_c$ and we work directly
with $\Lambda=\Lambda_p+\Lambda_c$, $\Gamma=\Gamma_p+\Gamma_c$. Schematically, the equations are
\begin{align*}
&\S  \Lambda +
\frac{V_N}{N} \Lambda= V_N \Lambda \circ \Gamma + V_N \Gamma \circ \Lambda\\
& \S_{\pm}\Gamma= V_N \Lambda \circ \Gamma + V_N\Lambda \circ \Lambda.
\end{align*}
Apply $\nabla_{x+y}$ and localize the RHS to $[T_i, T_{i+1}]$:
\begin{align}
&\S \nabla_{x+y} \Lambda^i + \notag
\frac{V_N}{N}\nabla_{x+y} \Lambda^i\\
&= \chi_{[T_i, T_{i+1}]} \bigg(V_N \nabla_{x+y}\Lambda \circ \Gamma + V_N \nabla_{x+y} \Gamma \circ \Lambda
+V_N \Lambda \circ\nabla_{x+y} \Gamma + V_N \Gamma \circ\nabla_{x+y} \Lambda\bigg)\notag\\
&:= RHS(1)\\
& \S_{\pm}\nabla_{x+y}\Gamma^i \notag\\
&= \chi_{[T_i, T_{i+1}]} \bigg(V_N \nabla_{x+y}\Gamma \circ \Gamma + V_N \nabla_{x+y}\Lambda \circ \Lambda \notag
+V_N \Lambda  \circ \nabla_{x+y}\Gamma \circ \Gamma + V_N\Lambda \circ \nabla_{x+y} \Lambda\bigg)\\
&:= RHS(2)\label{rhsgamma}
\end{align}
with initial conditions at $T_i$, so $\nabla_{x+y}\Lambda^i=\nabla_{x+y}\Lambda$ and $\nabla_{x+y}\Gamma^i=\nabla_{x+y}\Gamma$ in $[T_i, T_{i+1}]$.
By slight abuse of notation, $\nabla_{x+y} F$ is the function $(x, y) \to \nabla_{x+y} F(x, y)$.
 Now we can use
and use Theorems \ref{mainthm} and Theorem \ref{estforGamma}:

\begin{align*}
&\big\|\big<\nabla_x\big>^{\alpha}\big<\nabla_y\big>^{\alpha}\nabla_{x+y}\Lambda^i \big\|_{\mathcal{S}_{x, y}}
+\big\|\big<\nabla_{x+y}\big>^{\alpha}\nabla_{x+y}\Lambda^i\big\|_{L^2(dt)L^{\infty}(d(x-y)) L^2(d(x+y))}\\
&+\big\|\big|\partial_t\big|^{\frac{1}{4}}\nabla_{x+y}\Lambda^i\big\|_{L^2(dt)L^{\infty}(d(x-y)) L^2(d(x+y))}\\
&+\big\|\big<\nabla_x\big>^{\alpha}\big<\nabla_y\big>^{\alpha}\nabla_{x+y}\Gamma^i \big\|_{\mathcal{S}_{x, y}}
+\big\||\nabla_{x+y}|^{\alpha}\nabla_{x+y}\Gamma^i\big\|_{L^{\infty}(d(x-y))L^2(dt) L^2(d(x+y))}\\
&\lesssim
\big\|\big<\nabla_x\big>^{\alpha}\big<\nabla_y\big>^{\alpha}RHS(1) \big\|_{\mathcal{S}_r'}
+\big\|\big<\nabla_x\big>^{\alpha}\big<\nabla_y\big>^{\alpha}RHS(2) \big\|_{\mathcal{S}_r'}\\
&+ \big\|\big<\nabla_x\big>^{\alpha}\big<\nabla_y\big>^{\alpha}\nabla_{x+y}\Lambda(t=T_i) \big\|_{L^2}\\
&+ \big\|\big<\nabla_x\big>^{\alpha}\big<\nabla_y\big>^{\alpha}\nabla_{x+y}\Gamma(t=T_i) \big\|_{L^2}.
\end{align*}

We have to estimate $\big<\nabla_x\big>^{\alpha}\big<\nabla_y\big>^{\alpha}$ applied to the 8 terms on the RHS.
We will show they are all  $ \le C \epsilon LHS$, so if $C \epsilon \le \frac{1}{2}$, the theorem is proved.
 In the above compositions, the estimates are the same regardless whether the second term is $\Lambda$ or $\Gamma$, so we call the second term $B$. The estimates are the same estimates as those of Lemmas \ref{Lem1} and \ref{Lem3}, but now we can also use the estimates of Lemma \ref{smallness}. If $\nabla_{x+y} $ falls on $B$,
\begin{align*}
&\big\|\big<\nabla_x\big>^{\alpha}\big<\nabla_y\big>^{\alpha}\bigg(
 V_N \Gamma^{i}\circ \nabla_{x+y} B^{i} \bigg)
 \big\|_{L^{\frac{8}{5}}([T_i, T_{i+1}])L^{\frac{4}{3}}(dx)L^2(dy)}\\
  &\le \epsilon \|\big<\nabla_x\big>^{\alpha}\big<\nabla_y\big>^{\alpha}\nabla_{x+y} B^{i}\|_{\mathcal S^r_{x, y}}\\
 \end{align*}
and
\begin{align*}
&\|\big<\nabla_x\big>^{\alpha}\big<\nabla_y\big>^{\alpha}\left((V_N\Lambda^{i})\circ \nabla_{x+y} B^i \right)\|_{L^{\frac{4}{3}}([T_i, T_{i+1}])L^{\frac{3}{2}}(dx)L^2(dy)}\\
&\le
 \|\big<\nabla_{x+y}\big>^{\alpha}\Lambda^{i}\|_{L^{\infty}(d(x-y))L^2([T_i, T_{i+1}])L^2(d(x+y))}
 \|\big<\nabla_x\big>^{\alpha}\big<\nabla_y\big>^{\alpha}\nabla_{x+y} B^i
\|_{L^{4}(dt)L^{3}(dx)L^2(dy)}\\
&\le \epsilon \|\big<\nabla_x\big>^{\alpha}\big<\nabla_y\big>^{\alpha}\nabla_{x+y} B^i
\|_{L^{4}(dt)L^{3}(dx)L^2(dy)}.
\end{align*}
If $\nabla_{x+y} $ falls on $V_N \Gamma$,
\begin{align*}
&\big\|\big<\nabla_x\big>^{\alpha}\big<\nabla_y\big>^{\alpha}\bigg(
( V_N \nabla_{x+y}\Gamma^{i})\circ  B^{i} \bigg)
 \big\|_{L^{\frac{4}{3}}([T_i, T_{i+1}])L^{\frac{3}{2}}(dx)L^2(dy)}\\
  &\le \| \nabla_{x+y}\Gamma^i\|_{L^{\infty}(d(x-y))L^2(dt)L^{2}(d(x+y))} \|\big<\nabla_x\big>^{\alpha}\big<\nabla_y\big>^{\alpha} B^i
\|_{L^{4}([T_i, T_{i+1}])L^{3}(dx)L^2(dy)}\\
& \le \epsilon \| \nabla_{x+y}\Gamma^i\|_{L^{\infty}(d(x-y))L^2([T_j, T_{j+1}])L^{2}(d(x+y))}
 \end{align*}
 and similarly
 \begin{align*}
&\big\|\big<\nabla_x\big>^{\alpha}\big<\nabla_y\big>^{\alpha}\bigg(
( V_N \nabla_{x+y}\Lambda^{i})\circ  B^{i} \bigg)
 \big\|_{L^{\frac{4}{3}}([T_i, T_{i+1}])L^{\frac{3}{2}}(dx)L^2(dy)}\\
  &\le \| \nabla_{x+y}\Lambda^i\|_{L^{\infty}(d(x-y))L^2([T_j, T_{j+1}])L^{2}(d(x+y))} \|\big<\nabla_x\big>^{\alpha}\big<\nabla_y\big>^{\alpha} B^i
\|_{L^{4}([T_i, T_{i+1}])L^{3}(dx)L^2(dy)} \\
&\le \epsilon \| \nabla_{x+y}\Lambda^i\|_{L^{\infty}(d(x-y))L^2(dt)L^{2}(d(x+y))}.
 \end{align*}

 The proof for $\nabla_{x+y}^j$ is the same.

\end{proof}

\section{Estimates for $\sht$, $p_2=\shb \circ \sh$ and $\sh$}

\begin{proof} (of Theorem \ref{intro2}).
The equations for $\sht=N \Lambda_p$ and $p_2=N \Gamma_p$ are
\begin{align*}
&\S \,  \sht +\{V_N * \rho, \sht\}
+ \left((V_N\Gamma^T)\circ \sht + (V_N\Lambda)\circ  p_2 \right)_{symm}=-\frac{V_N}{2} \Lambda\\
&\bar \S_{\pm} p_2 + [V_N * \rho, p_2]+ \left((V_N\Gamma)\circ p_2 + (V_N\bar \Lambda)\circ { \sht} \right)_{skew}=0.
\end{align*}

Let $\epsilon >0$. As in the previous proofs, use \eqref{epsilon2} and
 divide $[0, \infty)$ into finitely many intervals $[T_i, T_{i+1}]$ so that
\begin{align*}
&\sup_z\| \Gamma(t, x+z, x)\|_{L^8([T_i, T_{i+1}])L^{\frac{12}{7}}(dx)} < \epsilon
\end{align*}
and
\begin{align*}
&\sup_z\| \Lambda(t, x+z, x)\|_{L^2([T_i, T_{i+1}])L^{3}(dx)} < \epsilon
\end{align*}
and estimate the dual Strichartz norms
\begin{align*}
&\|(V_N * \rho(t, x)) \sht(t, x, y)\|_{L^{\frac{8}{5}}([T_i, T_{i+1}]) L^{\frac{4}{3}}(dx)L^2(dy)}
+ \|(V_N\Gamma^T)\circ \sht \|_{L^{\frac{8}{5}}([T_i, T_{i+1}]) L^{\frac{4}{3}}(dx)L^2(dy)}\\
&\le C \sup_z\| \Gamma(t, x+z, x)\|_{L^8([T_i, T_{i+1}])L^{\frac{12}{7}}(dx)} \|\sht\|_{L^2([T_i, T_{i+1}])L^6(dx)L^2(dy)}\\
&\le C \epsilon \|\sht\|_{L^2([T_i, T_{i+1}])L^6(dx)L^2(dy)}
\end{align*}
and similarly
\begin{align*}
&\|(V_N * \rho(t, x)) p_2(t, x, y)\|_{L^{\frac{8}{5}}([T_i, T_{i+1}]) L^{\frac{4}{3}}(dx)L^2(dy)}
+\|
(V_N\Gamma)\circ p_2\|_{L^{\frac{8}{5}}([T_i, T_{i+1}]) L^{\frac{4}{3}}(dx)L^2(dy)}\\
&\le C \epsilon \|p_2\|_{L^2([T_i, T_{i+1}])L^6(dx)L^2(dy)}.
\end{align*}
Also,
\begin{align*}
&\|(V_N\Lambda)\circ  p_2\|_{L^{\frac{8}{5}}(dt)L^{\frac{4}{3}}(dx)L^2(dy)} \\
&\le  C \sup_z\| \Lambda(t, x+z, x)\|_{L^2([T_i, T_{i+1}])L^{3}(dx)}\|p_2\|_{L^4([T_i, T_{i+1}])L^{3}(dx)L^2(dy)}\\
&\le \epsilon \|p_2\|_{L^4([T_i, T_{i+1}])L^{3}(dx)L^2(dy)}.
\end{align*}
  We get, using the estimates of Theorem \ref{intro1} as well as
  proposition
\ref{Suusharp} and standard Strichartz estimates,
\begin{align*}
&\|\sht\|_{\mathcal S_{x, y}[T_i, T_{i+1}]}
+\|p_2\|_{\mathcal S_{x, y}[T_i, T_{i+1}]}\\
&\le C (\left(\|\sht(t=T_i)\|_{L^2} + \|p_2(t=T_i)\|_{L^2}\right)\\
&+  C \epsilon (\left(\|\sht\|_{\mathcal S_{x, y}[T_i, T_{i+1}]}
+\|p_2\|_{\mathcal S_{x, y}[T_i, T_{i+1}]} \right) + C \log N.
\end{align*}
Thus, if $C\epsilon \le \frac{1}{2}$ we get the desired result on each interval $[T_i, T_{i+1}]$. The number of such intervals is finite
(bounded by universal constants), and the result follows.

The proof for $\nabla_{x+y}^j$ is similar.

\end{proof}

\section{Estimates for the condensate $\phi$}

The non-linear equation for $\phi$ can be regarded as a linear equation on a background given by $\Gamma$ and $\Lambda$, for which we already have estimates:
\begin{align}
&  \left\{\frac{1}{i}\partial_{t}-\Delta_{x_1}\right\}\phi(x_{1})  \notag\\
&=-\int dy\left\{
v_N(x_{1}-y)\Gamma(y,y)\right\}\phi(x_{1})  \label{phi1}
\\
&-\int dy\big\{v_N(x_{1}-y)\Gamma_p(y,x_{1})\phi(y) \label{phi2}\\
&+ \int dy\big\{v_N(x_{1}-y)\Lambda_p(x_{1},y)\big\}\overline{\phi}(y ) \label{phi3}.
 \end{align}
Define the standard Strichartz spaces
\begin{align*}
\|\phi\|_{\mathcal S} =  \sup_{  p, q \, \, admissible
}\|\phi\|_{L^{p}(dt) L^{q}(dx)}.
\end{align*}

We prove the estimates of Corollary  \ref{corphi}.

\begin{proof}
Using \eqref{apriori3} and \eqref{est1} we split $[0, \infty)$ into finitely many intervals so that
\begin{align*}
&\|\big<\nabla_{x+y}\big>^{\alpha}\Gamma\|_{L^8(dt)L^{\infty}(d(x-y))L^{\frac{4}{3}}(d(x+y))} \le \epsilon\\
& \|\big<\nabla_{x+y}\big>^{\alpha}\Lambda\|_{L^2(dt)L^{\infty}(d(x-y))L^{2}(d(x+y))} \le \epsilon
\end{align*}
Using the fractional Leibniz rule, we easily estimate the RHS of the equation for $\phi$:
\begin{align*}
&\|\big<\nabla\big>^{\alpha} \eqref{phi1}\|_{L^{\frac{8}{5}}([T_i. T_{i+1}])L^{\frac{4}{3}}(dx)} \le C \epsilon \|\big<\nabla\big>^{\alpha} \phi\|_{L^2[T_i. T_{i+1})L^6(dx)}\\
&\|\big<\nabla\big>^{\alpha} \eqref{phi2}\|_{L^{\frac{8}{5}}[T_i. T_{i+1})L^{\frac{4}{3}}(dx)} \le C \epsilon \|\big<\nabla\big>^{\alpha} \phi\|_{L^2[T_i. T_{i+1})L^6(dx)}\\
&\|\big<\nabla\big>^{\alpha} \eqref{phi3}\|_{L^{1}[T_i. T_{i+1})L^{2}(dx)} \le C \epsilon \|\big<\nabla\big>^{\alpha} \phi\|_{L^2[T_i. T_{i+1})L^6(dx)}
\end{align*}
thus
\begin{align*}
\|\big<\nabla\big>^{\alpha}\phi\|_{\mathcal S[T_i, T_{i+1}]} \le C \|\big<\nabla\big>^{\alpha}\phi(t=T_i)\|_{L^2} + 3 C \epsilon  \|\big<\nabla\big>^{\alpha} \phi\|_{L^2[T_i. T_{i+1})L^6(dx)}.
\end{align*}
By taking $3 \epsilon < \frac{1}{2}$, the result  for $j=0$ follows. Next, differentiate the equation and use the same splitting. If the derivative falls on $\phi$, the argument is the same.
If the derivative falls on $\Gamma$, use
\begin{align*}
\sup_{x-y} \||\nabla_{x+y}|^{\alpha+1}\Gamma\|_{L^2(dt d(x+y))} \lesssim 1.
\end{align*}
While we don't know if this term can be made small by localizing to time intervals, such a term is coupled with $\phi$ without extra derivatives, which has been estimated already. For instance,
\begin{align*}
&\int dz|v_N(z)|\|\left(|\nabla_{x}|^{\alpha+1}\Gamma_p(x, x-z)\right)\phi(x-z) \|_{L^{1}[T_i. T_{i+1})]L^{2}(dx)}\\
&+\int dz |v_N(z)|\left(|\nabla_{x}|\Gamma_p(x, x-z)\right)\big<\nabla\big>^{\alpha}\phi(x-z) \|_{L^{1}[T_i. T_{i+1})]L^{2}(dx)}\\
&\lesssim \sup_{x-y} \||\nabla_{x+y}|^{\alpha+1}\Gamma_p\|_{L^2L^2}\|
||\nabla|^{\alpha}
\phi\|_{L^2(dt)L^{6}(dx)} \le C
\end{align*}
thus we get
\begin{align*}
\|\big<\nabla\big>^{\alpha+1}\phi\|_{\mathcal S[T_i, T_{i+1}]} \le C \|\big<\nabla\big>^{\alpha+1}\phi(t=T_i)\|_{L^2} + 3 C \epsilon  \|\big<\nabla\big>^{\alpha+1} \phi\|_{L^2[T_i. T_{i+1})L^6(dx)} +C
\end{align*}
which proves the result. The case of higher $j$ is similar.
\end{proof}

\section{Proof the square function estimates}\label{prelim}

\subsection{The double square function in standard coordinates}

This subsection covers well-known results, and is included for the reader's convenience.
Let $\psi_k$ ($k \ge 1$) be any functions satisfying  $\hat \psi_k=\hat \psi (\frac{\cdot}{2^k})$ with $\hat \psi \in C_0^{\infty}$  vanishing in a neighborhood of $0$. Let $ \hat \psi_0 \in C_0^{\infty}$.

 Define
 $\overrightarrow K(x)$ be the infinite column vector $\overrightarrow{\psi(x)}=(\psi_k(x))_{k \ge 0}$. This  satisfies the standard estimates for a  Calder\'on-Zygmund operator: in particular (in three space dimensions)
$|\overrightarrow{\psi(x)}|\lesssim \frac{1}{|x|^3}$,\\
 $| \overrightarrow{ \psi(x+y)}-\overrightarrow{ \psi(x)}|\lesssim \frac{|y|}{|x|^4}$ if $|x| > 2|y|$. Convolution with $\overrightarrow K(x)$ is bounded from $L^2$ to $L^2l^2$ by orthogonality.

Denote
\begin{align*}
&\overrightarrow K_1f(x, y)=\int \overrightarrow K(x')f(x-x', y)dx'\\
&\overrightarrow K_2f(x, y)=\int \overrightarrow K(y')f(x, y-y')dy'.
\end{align*}
We review the following known results

\begin{lemma}\label{C-Z}
Let $1< p, q < \infty$. Then
\begin{align*}
&\|\overrightarrow K_1f\|_{L^p(dx)L^q(dy)l^2} \lesssim \|f\|_{L^p(dx)L^q(dy)} \\
&\|\overrightarrow K_2f\|_{L^p(dx)L^q(dy)l^2} \lesssim \|f\|_{L^p(dx)L^q(dy)}.
\end{align*}
\end{lemma}

\begin{remark}
The above inequalities are just
"linear" formulation of the square function estimate
\begin{align}
&\|S_1f\|_{L^p(dx)L^q(dy)}=\|\left(\sum_{k=0}^{\infty} |f*(\psi_k \delta)|^2\right)^{\frac{1}{2}}\|_{L^p(dx)L^q(dy)}\lesssim \|f\|_{L^p(dx)L^q(dy)}\label{S1}\\
&\|S_2f\|_{L^p(dx)L^q(dy)}=\|\left(\sum_{k'=0}^{\infty} |f*(\delta \psi_{k'} )|^2\right)^{\frac{1}{2}}\|_{L^p(dx)L^q(dy)}\lesssim \|f\|_{L^p(dx)L^q(dy)}\notag.
\end{align}
\end{remark}

\begin{proof}
 The estimate for $\overrightarrow K_2$ (or $S_2$) follows right away from the standard square function estimate in $y$
 (for fixed $x$), followed by $L^p$ in $x$.

The operator $\overrightarrow K_1$
is a Calder\'on-Zygmund operator. It is bounded from $L^p(dx)L^q(dy)$ to $L^p(dx)L^q(dy)l^2$. See \cite{Ward}, Theorem 2.1 and Corollary 2.3. For the main case we need, $q=2$, this also follows from Section 5, Chapter 2 in \cite{Stein}

\end{proof}

In fact we can do more:
Let  $\overrightarrow K$ be the infinite column kernel $\overrightarrow{\psi(x)}$ as before, but now we multiply it
with $l^2$ vectors  $\overrightarrow {f_j}$:
\begin{align*}
\overrightarrow{\psi} \otimes \overrightarrow {f} = \left(\psi_k f_j\right)_{j, k}.
\end{align*}
 It maps $ l^2 \to l^2l^2$ with norm given by $|\overrightarrow{\psi}|_{l^2}$, and we  repeat
the argument, and get

\begin{lemma}
Define
\begin{align}
&\overrightarrow{\overrightarrow K}_1(\overrightarrow f)(x, y)=\int \overrightarrow K(x') \otimes \overrightarrow f(x-x', y)dx' \label{doubleK}
\end{align}
Then, if $1<p, q < \infty$,
\begin{align}
\|\overrightarrow{\overrightarrow K}_{1}\overrightarrow f \|_{L^p(dx) L^q(dy)l^2l^2} \lesssim
 \|\overrightarrow f\|_{L^p(dx) L^q(dy)l^2}
\end{align}
and thus
\begin{align}
&\|\overrightarrow{\overrightarrow K}_{1}(\overrightarrow K_2 f) \|_{L^p(dx) L^q(dy)l^2l^2} \lesssim
 \|\overrightarrow f\|_{L^p(dx) L^q(dy)l^2}\label{leib}
\end{align}
or, equivalently
\begin{align*}
&\|S_1S_2f\|_{L^p(dx)L^q(dy)}=\|\left(\sum_{k', k''=0}^{\infty} |f*(\psi_{k'} \delta)*(\delta \psi_{k''} )|^2\right)^{\frac{1}{2}}\|_{L^p(dx)L^q(dy)}\lesssim \|f\|_{L^p(dx)L^q(dy)}.
\end{align*}
The result is also true, uniformly in $M$,  if the dyadic intervals $2^i$ defining the square functions are replaced by $2^iM$.
\end{lemma}

\subsection{The double square function  in rotated coordinates}

 Recall
\begin{align*}
R=\frac{1}{\sqrt 2} \left(
\begin{matrix}
1 & 1\\
-1 & 1
\end{matrix}
\right)
\end{align*}
and $L_1$, $L_2$ non-singular matrices satisfying
satisfying
\begin{align*}
(RL_1)^{-1}=\left(
\begin{matrix}
1 & a\\
0 & b
\end{matrix}
\right)
\end{align*}
\begin{align*}
(RL_2)^{-1}=\left(
\begin{matrix}
c & 1\\
d & 0
\end{matrix}
\right).
\end{align*}

\begin{lemma}\label{Kf}
Let $K$, $f$ be functions or distributions on $\mathbb R^3$ so $K*f$ is defined, and let $K \delta$ be $K(x) \delta(y)$, and $\delta K$ $\delta(x) K(y)$.
Then
\begin{align}
&\left((K \delta) * (f\circ ( RL_1)^{-1})\right)(RL_1(x, y))=((K \delta)*f)(x, y)\label{K1f}\\
&\left((\delta K ) * (f\circ ( RL_2)^{-1})\right)(RL_2(x, y))=((K \delta)*f)(x, y)\label{K2f}
\end{align}
\end{lemma}
\begin{proof}
We have
\begin{align*}
&\int K(x')\delta(y') f((x, y) - (RL_1)^{-1}(x', y'))\\
&= \int K(x')\delta(y') f(x-x'-ay', y-by') dx'dy' \\
&=\int  K(x')f(x-x', y) dx'
\end{align*}
and
\begin{align*}
&\int \delta (x') K(y') f((x, y) - (RL_2)^{-1}(x', y'))\\
&= \int \delta (x') K(y') f(x-cx'-y', y-x') dx'dy' \\
&=\int  K(y')f(x-y', y) dy'.
\end{align*}
\end{proof}
 This calculation also works if $K$ is an $l^2$ valued function, such as $\overrightarrow K$.

Using this remark, we obtain

\begin{lemma} \label{doublesquarerotated} Let $ \overrightarrow K_1$, $\overrightarrow K_2$ as above.
 Let $R$ be the linear transformation defined by \eqref{Rdef}

Then
\begin{align}
\|\overrightarrow K_{1, 2}(f\circ R^{-1})(R(x, y))\|_{L^p(dx) L^q(dy)l^2} \lesssim \|f\|_{L^p(dx) L^q(dy)}\label{translation}.
\end{align}
\end{lemma}
\begin{proof} The proof uses only the invariance of $L^p(dx) L^q(dy)$ under lower triangular invertible matrices:
 $\|f\|_{L^p(dx) L^q(dy)}= c\|f\circ L_1\|_{L^p(dx) L^q(dy)}$
and \eqref{translation} is equivalent to
\begin{align*}
\|\overrightarrow K_1(f\circ (RL_1)^{-1})(RL_1(x, y))\|_{L^p(dx L^q(dy))} \lesssim \|f\|_{L^p(dx) L^q(dy)}.
\end{align*}

For $\overrightarrow K_1$ we are convolving with $\overrightarrow K(x)\delta(y)$. Using \eqref{K1f} we have
\begin{align*}
&\overrightarrow  K_1(f\circ (RL_1)^{-1})(RL_1(x, y))
= \overrightarrow  K_1(f)(x, y)
\end{align*}
and we already know from Lemma \ref{C-Z} this is bounded on $L^p(dx)L^q(dy)$.

The argument for $\overrightarrow K_2$ is similar, but uses $L_2$:
\begin{align*}
&\overrightarrow  K_2(f\circ (RL_2)^{-1})(RL_2(x, y)) \notag \\
&=\int \overrightarrow K(x')f(x-x', y) dx' = \overrightarrow  K_1(f)(x, y)
\end{align*}
which is bounded on $L^p(dx)L^q(dy)$, as in the previous case.

\end{proof}

Next, recall $\overrightarrow{\overrightarrow K}_1$ defined by \eqref{doubleK}. The same proof as before gives

\begin{lemma}\label{squarerotated}
Let $1< p, q < \infty$. Then
\begin{align}
\|\overrightarrow{\overrightarrow K}_{1, 2}(\overrightarrow f \circ R^{-1})(R(x, y))\|_{L^p(dx) L^q(dy)l^2l^2} \lesssim
 \|\overrightarrow f\|_{L^p(dx) L^q(dy)l^2}\label{translationtensor}
\end{align}
and, as a corollary,
\begin{align}
&\|\overrightarrow{\overrightarrow K}_{1}(\overrightarrow K_2 (f \circ R^{-1})(R(x, y))\|_{L^p(dx) L^q(dy)l^2l^2} \lesssim
 \|\overrightarrow f\|_{L^p(dx) L^q(dy)l^2}\label{translationtensorcomp}\\
\end{align}
or, equivalently,
\begin{align*}
&\|(S_1S_2f)\circ R\|_{L^p(dx)L^q(dy)}\\
&=\|\left(\sum_{k', k''=0}^{\infty} |f*(\psi_{k'} \delta)*(\delta \psi_{k''} )\circ R|^2\right)^{\frac{1}{2}}\|_{L^p(dx)L^q(dy)}\lesssim \|f\circ R\|_{L^p(dx)L^q(dy)}.
\end{align*}
\end{lemma}
In other words, we have the double square function estimate in $x-y$, $x+y$ coordinates.

\end{document}